\documentclass[10pt]{amsart}

\topskip  \usepackage{amsmath,amsthm,cite}
\usepackage{amssymb,longtable,array,pstricks}
\usepackage{mathdots}
\def\gf{generating function\xspace}
\def\al{\alpha}
\def\be{\beta}
\def\ga{\gamma}

\usepackage{xspace}

\newtheorem{theorem}{Theorem}
\newtheorem{lemma}[theorem]{Lemma}
\newtheorem{proposition}[theorem]{Proposition}
\newtheorem{corollary}[theorem]{Corollary}

\begin{document}
\title[Avoiding 1324 and two other $4$-letter patterns]{Enumeration of small Wilf classes avoiding 1324 and two other $4$-letter patterns}
\author[D. Callan]{David Callan}
\address{Department of Statistics, University of Wisconsin, Madison, WI 53706}
\email{callan@stat.wisc.edu}
\author[T.~Mansour]{Toufik Mansour}
\address{Department of Mathematics, University of Haifa, 3498838 Haifa, Israel}
\email{tmansour@univ.haifa.ac.il}

\begin{abstract}
Recently, it has been determined that there are 242 Wilf
classes of triples of 4-letter permutation patterns by showing
that there are 32 non-singleton Wilf classes. Moreover,
the generating function for each triple lying in a non-singleton Wilf class has been explicitly determined.
In this paper, toward the goal of enumerating avoiders for the singleton Wilf classes, we obtain the generating function for all but one of the triples containing
1324. (The exceptional triple is conjectured to be intractable.) Our methods are both combinatorial and analytic, including generating trees,
recurrence relations, and decompositions by left-right maxima.
Sometimes this leads to an algebraic equation for the generating
function, sometimes to a functional equation or a multi-index
recurrence amenable to the kernel method.
\bigskip

\end{abstract}
\maketitle

\section{Introduction}
In recent decades pattern avoidance has received a lot of attention. It has a
prehistory in the work of MacMahon \cite{macmahon1915} and Knuth \cite{K}, but the
current interest was sparked by a paper of Simion and Schmidt \cite{SiS}. They thoroughly analyzed 3-letter patterns in permutations, including a bijection
between 123- and 132-avoiding permutations, thereby explaining the first
(nontrivial) instance of what is, in modern terminology, a Wilf class. Since then
the problem has been addressed on several other discrete structures, such as
compositions, $k$-ary words, and set partitions; see, e.g., the texts \cite{SHM,TM}
and references contained therein.

Permutations avoiding a single 4-letter pattern have been well studied (see, e.g.,
\cite{St0,St,W,wikipermpatt}), and the latter form 7 symmetry classes and 3 Wilf
classes.  As for pairs of 4-letter patterns, there are 56 symmetry classes, for all
but 5 of which the avoiders have been enumerated \cite{wikipermpatt}. Le \cite{L}
established that these $56$ symmetry classes form $38$ distinct Wilf classes.

The $\binom{24}{3}= 2024$ triples of 4-letter patterns split into 317 symmetry classes. It is known \cite{CMS3patI,CMS3patII} that the 317 symmetry classes split into $242$ Wilf classes, 32 of which are large (a Wilf class is called large if it contains at least two symmetry classes, and small if it consists of a singleton symmetry class) and the large Wilf classes are all explicitly enumerated in \cite{CMS3patI,CMS3patII}, where it is shown that each has an algebraic
generating function.

Our goal here is to enumerate (with one exception, see \cite{gp2015} and \cite{SlA257562}) all the small Wilf classes that contain the pattern 1324.
Running the INSENC algorithm (regular insertion encoding, see \cite{ALR,V})
over all the 210 small Wilf classes
determines the generating function for 126 of them, as presented in the Appendix.
The remaining small classes that contain 1324 are listed in Table \ref{tabgf1324}
along with their generating functions, where the numbering follows that of Table 2
in the appendix to\cite{HYL}, based on lex order of counting sequences.

Section \ref{prelim} contains some preliminary remarks, and Section \ref{proofs}
gives the proofs for all the results in Table \ref{tabgf1324} not proved elsewhere.

{\scriptsize\begin{longtable}[c]{|c|c|c|c|}
\caption{Small Wilf classes of three 4-letter patterns not counted by INSENC that include the pattern $1324$\label{tabgf1324}}\\ \hline
No. & $T$&$F_T(x)$&Thm./[Ref]\\\hline
\endfirsthead  \hline
No. & $T$&$F_T(x)$&Thm./[Ref]\\ \hline
\endhead \hline
\endfoot \hline
\endlastfoot
\raisebox{-.5mm}{29} & \raisebox{-.5mm}{$\{1324,2143,3421\}$} & \raisebox{-.5mm}{$\frac{1 - 8x + 27x^2 - 48x^3 + 50x^4 - 30x^5 + 6x^6}{(1 - x)^5(1 - 2x)^2}$ } & \raisebox{-.5mm}{\cite{AA20160607}} \\[2mm]\hline
\raisebox{-.5mm}{30} & \raisebox{-.5mm}{$\{4231,2143,1324\}$} & \raisebox{-.5mm}{$\frac{1 - 6 x + 14 x^2 - 14 x^3 + 8 x^4 - 2 x^6}{(1 - x)^3 (1 - 2x)^2 }$ } & \raisebox{-.5mm}{\cite{AA20160607}} \\[2mm]\hline

\raisebox{-.5mm}{35} & \raisebox{-.5mm}{$\{1324,2143, 3412\}$}& \raisebox{-.5mm}{$\frac{1-9x+33x^2-62x^3+64x^4-38x^5+10x^6}{(1-3x+x^2)(1-2x)^2(1-x)^3}$}  & \raisebox{-.5mm}{\cite{AA20160607}}  \\[2mm] \hline

\raisebox{-.5mm}{49} & \raisebox{-.5mm}{$\{1324,2341,4123\}$} & \raisebox{-.5mm}{$C(x)+\frac{x^3 - 3 x^4 +3 x^5 -5 x^6 + 9 x^7 -4 x^8}{(1-x)^6 (1-2x)^2}$} & \raisebox{-.5mm}{\ref{th49a}}\\[2mm]\hline
\raisebox{-.5mm}{69} & \raisebox{-.5mm}{$\{1234, 1324, 3412\}$} & \raisebox{-.5mm}{$\frac{1-9x+35x^2-75x^3+98x^4-78x^5+34x^6-10x^7}{(1-2x)^2(1-x)^6}$} & \raisebox{-.5mm}{\ref{th69a}}\\[2mm]\hline
\raisebox{-.5mm}{72} & \raisebox{-.5mm}{$\{1243,1324,3412\}$} & \raisebox{-.5mm}{$1+\frac{x(1-11x+54x^2-152x^3+268x^4-311x^5+237x^6-109x^7+30x^8-4x^9)}{(1-x)^6(1-2x)^2(1-3x+x^2)}$} & \raisebox{-.5mm}{\ref{th72a}}\\[2mm]\hline
\raisebox{-.5mm}{75} & \raisebox{-.5mm}{$\{1243,1324,4231\}$} & \raisebox{-.5mm}{$\frac{x}{1-3x+x^2}-\frac{2-4x-4x^2-x^3}{(1-2x)^2}
+\frac{3-20x+55x^2-83x^3+74x^4-38x^5+12x^6-2x^7}{(1-x)^8}$} & \raisebox{-.5mm}{\ref{th75a}}\\[2mm]\hline
\raisebox{-.5mm}{76} & \raisebox{-.5mm}{$\{3412,1324,2341\}$} & \raisebox{-.5mm}{$\frac{1-10x+44x^2-110x^3+173x^4-176x^5+114x^6-45x^7+12x^8-4x^9}{(1-x)^7(1-2x)^2}$} & \raisebox{-.5mm}{\ref{th76a}}\\[2mm]\hline
\raisebox{-.5mm}{80} & \raisebox{-.5mm}{$\{1324,2341,3421\}$} & \raisebox{-.5mm}{$\frac{1-7x+20x^2-29x^3+25x^4-10x^5+2x^6}{(1-x)^5(1-3x+x^2)}$} & \raisebox{-.5mm}{\ref{th80a}}\\[2mm]\hline
\raisebox{-.5mm}{84} & \raisebox{-.5mm}{$\{4231,1324,2341\}$} & \raisebox{-.5mm}{$\frac{1-9x+33x^2-62x^3+64x^4-36x^5+7x^6}{(1-3x+x^2)(1-2x)^2(1-x)^3}$} & \raisebox{-.5mm}{\ref{th84a}}\\[2mm]\hline
\raisebox{-.5mm}{86} & \raisebox{-.5mm}{$\{3412,2431,1324\}$} & \raisebox{-.5mm}{$\frac{1-7x+19x^2-24x^3+16x^4-4x^5-x^6+2x^7}{(1-x)^3(1-2x)(1-3x+x^2)}$} & \raisebox{-.5mm}{\ref{th86a}}\\[2mm]\hline
\raisebox{-.5mm}{88} & \raisebox{-.5mm}{$\{3412,3421,1324\}$} & \raisebox{-.5mm}{$\frac{(1-x)^2(1-5x+7x^2+x^3)}{(1-2x)^4}$} & \raisebox{-.5mm}{\ref{th88a}}\\[2mm]\hline
\raisebox{-.5mm}{93} & \raisebox{-.5mm}{$\{1324,2413,3421\}$} & \raisebox{-.5mm}{$\frac{1-10x+42x^2-94x^3+120x^4-86x^5+31x^6-3x^7}{(1-x)^3(1-2x)^4}$} & \raisebox{-.5mm}{\ref{th93a}}\\[2mm]\hline
\raisebox{-.5mm}{99} & \raisebox{-.5mm}{$\{1324,3142,4231\}$} & \raisebox{-.5mm}{$\frac{1-8x+25x^2-36x^3+23x^4-4x^5+x^6}{(1-x)(1-2x)^4}$} & \raisebox{-.5mm}{\ref{th99a}}\\[2mm]\hline
\raisebox{-.5mm}{132} & \raisebox{-.5mm}{$\{1324,2341,2413\}$} & \raisebox{-.5mm}{$\frac{1-8x+23x^2-27x^3+12x^4-4x^5+x^6}{(1-3x+x^2)^3}$} & \raisebox{-.5mm}{\ref{th132a}}\\[2mm]\hline
\raisebox{-.5mm}{150} & \raisebox{-.5mm}{$\{1324,3421,3241\}$} & \raisebox{-.5mm}{$\frac{1-11x+52x^2-136x^3+214x^4-204x^5+111x^6-28x^7}{(1-x)^3(1-2x)^3(1-3x+2x^2)}$} & \raisebox{-.5mm}{\ref{th150a}}\\[2mm]\hline
\raisebox{-.5mm}{151} & \raisebox{-.5mm}{$\{1324,1342,3421\}$} & \raisebox{-.5mm}{$\frac{1-12x+61x^2-169x^3+275x^4-263x^5+136x^6-29x^7+x^8}{(1-3x+x^2)(1-2x)^4(1-x)^2}$} & \raisebox{-.5mm}{\ref{th151a}}\\[2mm]\hline
\raisebox{-.5mm}{153} & \raisebox{-.5mm}{$\{4231,1324,1342\}$} & \raisebox{-.5mm}{$\frac{1-10x+41x^2-87x^3+101x^4-61x^5+15x^6-x^7}{(1-x)^2(1-2x)^3(1-3x+x^2)}$} & \raisebox{-.5mm}{\ref{th153a}}\\[2mm]\hline
\raisebox{-.5mm}{156} & \raisebox{-.5mm}{$\{1324,2341,2431\}$} & \raisebox{-.5mm}{$\frac{1-8x+23x^2-25x^3+3x^4+7x^5}{(1-2x)^2(1-3x+x^2)(1-2x-x^2)}$} & \raisebox{-.5mm}{\ref{th156a}}\\[2mm]\hline
\raisebox{-.5mm}{158} & \raisebox{-.5mm}{$\{1324,1342,3412\}$} & \raisebox{-.5mm}{$\frac{1-10x+40x^2-81x^3+88x^4-50x^5+11x^6}{(1-x)^3(1-2x)(1-3x)(1-3x+x^2)}$} & \raisebox{-.5mm}{\ref{th158a}}\\[2mm]\hline
\raisebox{-.5mm}{172} & \raisebox{-.5mm}{$\{2143,4132,1324\}$} & \raisebox{-.5mm}{ $\frac{(2 - 10x + 16x^2 - 8x^3 + x^4)C(x) - 1 + 4x - 5x^2 + x^3}{(1 - x)^2(1 - 3x + x^2)}$ } & \raisebox{-.5mm}{\cite{AA20160607}}\\[2mm]\hline
\raisebox{-.5mm}{180} & \raisebox{-.5mm}{$\{1342,2314,4231\}$} & \raisebox{-.5mm}{$\frac{1-7 x +18 x^2 -22 x^3 +16 x^4 -6  x^5 +x^6 - \left(x- 5 x^2+8 x^3 - 2 x^4 -2 x^5 +x^6\right)C(x)  }{(1-2 x) (1-x)^2 \left(1-5 x+4 x^2-x^3\right)}$} & \raisebox{-.5mm}{\ref{th180a}}\\[2mm]\hline
\raisebox{-.5mm}{184} & \raisebox{-.5mm}{$\{1324,2431,3241\}$} & \raisebox{-.5mm}{$\frac{1-8x+24x^2-32x^3+19x^4-3x^5}{(1-x)(1-2x)(1-3x+x^2)^2}$} & \raisebox{-.5mm}{\ref{th184a}}\\[2mm]\hline
\raisebox{-.5mm}{187} & \raisebox{-.5mm}{$\{1324,2314,2431\}$} & \raisebox{-.5mm}{$\frac{1-9x+31x^2-49x^3+34x^4-7x^5}{(1-3x+x^2)^2(1-2x)^2}$} & \raisebox{-.5mm}{\ref{th187a}}\\[2mm]\hline
\raisebox{-.5mm}{193} & \raisebox{-.5mm}{$\{1324,2431,3142\}$} & \raisebox{-.5mm}{$\frac{x-1+ \left(x^2-5 x+2\right)C(x)}{1-3 x+x^2}$} & \raisebox{-.5mm}{\ref{th193a}}\\[2mm]\hline
\raisebox{-.5mm}{195} & \raisebox{-.5mm}{$\{1324,2341,1243\}$} & \raisebox{-.5mm}{$\frac{ (1 - 7 x + 19 x^2 - 25 x^3 + 13 x^4 + 4 x^5 - 5 x^6 + x^7)C(x) -1 + 7 x - 19 x^2 + 23 x^3 - 7 x^4 - 7 x^5 + 4 x^6}{x(1 - x)^2  (1 - 3 x + x^2) (1 - x - x^2)}$} & \raisebox{-.5mm}{\ref{th195a}}\\[2mm]\hline
\raisebox{-.5mm}{210} & \raisebox{-.5mm}{$\{1243,1324,2431\}$} & \raisebox{-.5mm}{$\frac{\left(1-6 x+13 x^2-11 x^3+4 x^4\right)}{x^2(1-x)^2 }C(x)-
\frac{1-6 x+12 x^2-8 x^3+2 x^4}{x^2(1-x)  (1-2 x)}$} & \raisebox{-.5mm}{\ref{th210a}}\\[2mm]\hline
\raisebox{-.5mm}{211} & \raisebox{-.5mm}{$\{1234,1324,2341\}$} & \raisebox{-.5mm}{$\frac{(1 - 4 x + 5 x^2 - 3 x^3)\, C(x) - (1 - 4 x + 6 x^2 - 4 x^3)}{x (1 - x)^2 (1 - 2 x)}$} & \raisebox{-.5mm}{\ref{th211a}}\\[2mm]\hline
\raisebox{-.5mm}{212} & \raisebox{-.5mm}{$\{1324,2413,2431\}$} & \raisebox{-.5mm}{$1+\frac{x(1-4x+4x^2-x^3-x(1-4x+2x^2)C(x))}{(1-3x+x^2)(1-3x+x^2-x(1-2x)C(x))}$} & \raisebox{-.5mm}{\ref{th212a}}\\[2mm]\hline
\raisebox{-.5mm}{213} & \raisebox{-.5mm}{$\{2431,1324,1342\}$} & \raisebox{-.5mm}{$\frac{ (1 - 5 x + 8 x^2 - 5 x^3)\,C(x)-1 + 4 x - 4 x^2 + x^3}{x^2(1 - 2 x) }$} & \raisebox{-.5mm}{\ref{th213a}}\\[2mm]\hline
\raisebox{-.5mm}{227} & \raisebox{-.5mm}{$\{2143,1432,1324\}$} & \raisebox{-.5mm}{ $\frac{1 - 6x + 12x^2 - 12x^3 + 6x^4 - x^5 - x^2(1 - x + x^2)^2 C(x)}{1 - 7x + 16x^2 - 19x^3 + 11x^4 - 2x^5 - x^6}$} & \raisebox{-.5mm}{\cite{AA20160607}}\\[2mm]\hline
\raisebox{-.5mm}{231} & \raisebox{-.5mm}{$\{1324,1342,2341\}$} & \raisebox{-.5mm}{$\frac{(1 - 3 x) \big(1 - 2x-xC(x)\big)}{(1 - 4 x) (1 - 3 x + x^2)}$} & \raisebox{-.5mm}{\ref{th231a}}\\[2mm]\hline
\raisebox{-.5mm}{237} & \raisebox{-.5mm}{$\{1432,1324,1243\}$} & \raisebox{-.5mm}{ ?  } & \raisebox{-.5mm}{\cite{SlA257562}}\\[2mm]\hline

\raisebox{-.5mm}{241}& \raisebox{-.5mm}{$\{1432,1324,1243\}$} & \raisebox{-.5mm}{$\frac{(v_--1)(v_+-1)((v_-+v_+)(v_-^2+v_+^2-x^2)+(x-1)(v_-^2+v_+^2+v_-v_+))}{x-(v_--1)(v_+-1)(v_-^2+v_+^2+v_-v_++x(v_-+v_++2-x))}$}  & \raisebox{-.5mm}{\ref{th241a}}  \\[3mm] \hline

\end{longtable}}

\section{Preliminaries}\label{prelim}
We say a permutation is \emph{standard} if its support set is an initial segment of the positive integers, and for a permutation $\pi$ whose support is any set of positive integers, St($\pi$) denotes the standard permutation obtained by replacing the smallest entry of $\pi$ by 1, the next smallest by 2, and so on.
Typically, for a a given triple $T$, we consider cases and analyze the structure of a $T$-avoider in each case to the point where we say that $T$-avoiders have such and such a form in that case. It is  always to be understood that we are also asserting, without explicit mention, that a permutation of the specified form is a $T$-avoider, and this enables us to determine the various ``contributions'' to the \gf $F_T(x)$ for $T$-avoiders, yielding an equation for $F_T(x)$. The equation may be an explicit expression for $F_T(x)$ or an algebraic or functional equation. For all but one symmetry class, the \gf turns out to be algebraic of degree $\le 4$. For the exceptional class (Case 237, where $\{1432, 1324, 1243\}$ and $\{4123, 4231, 4312\}$ are representative triples), the \gf is conjectured not to satisfy any ADE (algebraic differential equation), see \cite{gp2015} and \cite[Seq. A257562]{Sl}.

A permutation $\pi$ expressed as $\pi=i_1\pi^{(1)}i_2\pi^{(2)}
\cdots i_m\pi^{(m)}$ where $i_1<i_2<\cdots<i_m$ and
$i_j>\max(\pi^{(j)})$ for $1 \leq j \leq m$ is said to have $m$
\emph{left-right maxima} (at $i_1,i_2,\ldots,i_m$).  Given
nonempty sets of numbers $S$ and $T$, we will write $S<T$ to mean
$\max(S)<\min(T)$ (with the inequality vacuously holding if $S$ or
$T$ is empty).  In this context, we will often denote singleton
sets simply by the element in question. Also, for a number $k$,
$S-k$ means the set $\{s-k:s\in S\}$.

Throughout, $C(x)=\frac{1-\sqrt{1-4x}}{2x}$ denotes the generating function for
the Catalan numbers $C_n:=\frac{1}{n+1}\binom{2n}{n}=\binom{2n}{n}-\binom{2n}{n-1}$.
As is well known \cite{K,wikipermpatt}, $C(x)$ is the generating function for
$(|S_n(\pi)|)_{n\ge 0}$ where $\pi$ is any one of the six 3-letter
patterns. The identity $C(x)=\frac{1}{1-xC(x)}$ or, equivalently,
$xC(x)^2=C(x)-1$ is used to simplify some of our results.
Also throughout,  $L(x)=\frac{1-x}{1-2x}$ denotes the generating
function for $\{213,231\}$-avoiders (resp. $\{213,123\}$-avoiders,
resp. $\{132,123\}$-avoiders), see \cite{SiS}, and $K(x),K'(x)$ etc. are variously
used for other known generating functions.

\section{Proofs}\label{proofs}

\subsection{Case 49: $\{1324,2341,4123\}$}
For this case, we need the following lemmas.
\begin{lemma}\label{lem49a1}
Let $T=\{1324,2341,4123\}$. The generating function for the number of permutation $(n-1)\pi'n\pi''\in S_n(T)$ is given by
$$H(x)=\frac{x^3C(x)^3}{1-x}+\frac{x^2}{1-x}+\frac{x^4}{(1-x)(1-2x)}+\frac{x^5}{(1-x)^4}.$$
\end{lemma}
\begin{proof}
Let us write an equation for $H(x)$. Let $\pi=(n-1)\pi'n\pi''\in S_n(T)$. If $n=2$ then we have a contribution of $x^2$. So let us assume that $n>2$, so there are two cases, either $n-2$ belongs to $\pi'$ or to $\pi''$.
\begin{itemize}
\item $n-2$ belongs to $\pi'$: If $\pi''=\emptyset$ then we have a
contribution of $x^3(F_{\{123,132\}}-1)=\frac{x^3}{1-2x}$, see
\cite{SiS}. So, we can assume that $\pi''\neq\emptyset$. If $\pi'$
has a letter between $n-1$ and $n-2$, then $\pi$ can be written as
$$\pi=(n-1)(i-1)(i-2)\cdots
j(n-2)(j-1)(j-2)\cdots1n(n-3)(n-4)\cdots i,$$ which counted by
$\frac{x^5}{(1-x)^3}$. Otherwise, $\pi'$ has no letter between
$n-1$ and $n-2$, which gives a contribution of
$xH(x)-\frac{x^3(1-x)}{1-2x}$.

\item $n-2$ belongs to $\pi''$: In this case, we have a
contribution of $x(C(x)-1-xC(x))=x^3C(x)^3$, where $C(x)$ counts
the $\{123\}$-avoiders.
\end{itemize}
Hence, by adding all the contributions, we have
$$H(x)=x^2+\frac{x^3}{1-2x}+\frac{x^5}{(1-x)^3}+xH(x)-\frac{x^3(1-x)}{1-2x}+x^3C(x)^3,$$
which completes the proof.
\end{proof}

\begin{lemma}\label{lem49a2}
Let $T=\{1324,2341,4123\}$. The generating function for the number of $T$-avoiders with exactly $2$ left-right maxima is given by
$$G_2(x)=\frac{1}{1-x}\left(x^4C(x)^5+\frac{x^5}{(1-x)^5}+\frac{x^5}{(1-x)^4}+\frac{x^4}{1-2x}+H(x)\right),$$
where $H(x)$ is given in Lemma $\ref{lem49a1}$.
\end{lemma}
\begin{proof}
Let us write an equation for $G_2(x)$. Let $\pi=i\pi'n\pi''\in
S_n(T)$ be a permutation with exactly $2$ left-right maxima. We
consider the following cases:
\begin{itemize}
\item $i=n-1$: We have a contribution of $H(x)$ as defined in
Lemma \ref{lem49a1}. So from now, we assume that $\pi''$ contains
the letter $n-1$.

\item $\pi''=(n-1)\pi'''$: We have a contribution of $xG_2(x)$. So
we can assume that there is at least one letter between $n$ and
$n-1$. Since $\pi''$ avoids $1324$ and $4123$, we see that there
are exactly at most one letter between $n$ and $n-1$ that greater
than $i$.
\begin{itemize}
 \item if there is a letter in $n\pi''$ between $n$ and $n-1$ that
it is greater than $i$, then $\pi$ can be written as
$\pi=i\pi'nk\beta(n-1)\cdots(k+1)\alpha$ such that $\beta<i$ and
$\alpha>i$. By considering either $\beta$ is empty or not, we
obtain the contributions $\frac{x^4}{1-x}L(x)^2$ and
$\frac{x^5}{(1-x)^4}$, respectively. Recall $L(x)=\frac{1-x}{1-2x}$
is the generating function for $\{213,231\}$-avoiders
(also for $\{213,123\}$-avoiders).

\item Thus, we can assume that $\pi=\pi'n\beta(n-1)\alpha$ such
that $\alpha$ contains the subsequence $(n-1)(n-2)\cdots(i+1)$,
$\beta<i$ and $\beta$ is decreasing (since $\pi$ avoids $4123$,
and $\beta$ is not empty). Suppose that $\beta=ee'\beta'$ then
$e>e'>\beta'$ and $\pi'>e'$. If $\pi'$ has a letter between $e$
and $e'$, then easy to see that the contribution is given by
$\frac{x^{4+d}}{(1-x)^4}$, where $d$ is the number of the letters
in $\pi''$ that are greater than $i$. Otherwise, the contribution
is given by $x^{3+d}C(x)^{3+d}$, where $d$ is the number of the
letters in $\pi''$ that are greater than $i$. Therefore, we have a
contribution of
$$\sum_{d\geq1}\frac{x^{4+d}}{(1-x)^4}+\sum_{d\geq1}x^{3+d}C(x)^{3+d},$$
which equals
$$\frac{x^{5}}{(1-x)^5}+\frac{x^{4}C(x)^{4}}{1-xC(x)}.$$
\end{itemize}
\end{itemize}
Hence, the various contributions give
$$G_2(x)=xG_2(x)+x^4C(x)^5+\frac{x^5}{(1-x)^5}+\frac{x^5}{(1-x)^4}+\frac{x^4}{1-2x}+H(x),$$
which completes the proof.
\end{proof}

\begin{theorem}\label{th49a}
Let $T=\{1324,2341,4123\}$. Then
\[
F_T(x)=C(x)+\frac{x^3 - 3 x^4 +3 x^5 -5 x^6 + 9 x^7 -4 x^8}{(1-x)^6 (1-2x)^2}\, .
\]
\end{theorem}
\begin{proof}
Let $G_m(x)$ be the generating function for $T$-avoiders with $m$
left-right maxima. Clearly, $G_0(x)=1$ and
$G_1(x)=xF_{\{123\}}(x)=xC(x)$, see \cite{K}. By Lemmas
\ref{lem49a1} and \ref{lem49a2}, we have that
$$G_2(x)=\frac{1}{1-x}\left(x^4C(x)^5+\frac{x^5}{(1-x)^5}+\frac{x^5}{(1-x)^4}+\frac{x^4}{1-2x}+H(x)\right),$$
where
$$H(x)=\frac{x^3C(x)^3}{1-x}+\frac{x^2}{1-x}+\frac{x^4}{(1-x)(1-2x)}+\frac{x^5}{(1-x)^4}.$$
Now, let us write an equation for $G_3(x)$. Let
$\pi=i_1\pi'i_2\pi''n\pi'''\in S_n(T)$ with exactly $3$ left-right
maxima ($i_1,i_2,n$). Since $\pi$ avoids $1324$ and $2341$, then
$\pi'>\pi''$ and $\pi'''=\beta\alpha$ such that
$\beta>i_2>\alpha>i_1$. By considering the cases $\alpha,\beta$
are empty or not, we obtain the contributions $x^3L(x)^2$,
$x^3(L(x)-1)L(x)/(1-x)$, $x^3(L(x)-1)L(x)/(1-x)$, and
$x^5/(1-x)^4$. Hence,
$$G_3(x)=x^3L(x)^2+\frac{2x^3}{1-x}(L(x)-1)L(x)+\frac{x^5}{(1-x)^4}$$.

By very similar techniques as in case $G_3(x)$, we obtain that
$$G_4(x)=x^4(L(x)+(L(x)-1)/(1-x))^2.$$

Now, let us write an equation for $G_m(x)$ and $m\geq5$. Let
$\pi=i_1\pi^{(1)}i_2\pi^{(2)}\cdots i_m\pi^{(m)}\in S_n(T)$ with
exactly $m$ left-right maxima. Since $\pi$ avoids $T$, we see that
$\pi^{(s)}=\emptyset$ for all $s=3,4,\ldots,m-1$,
$\pi^{(1)}>\pi^{(2)}$ and $\pi^{(m)}=\beta\alpha$ such that
$\beta>i_{m-1}>\alpha>i_{m-2}$. Thus, $G_m(x)=xG_{m-1}(x)$, for all $m\geq5$.
Therefore,
$$F_T(x)-1-xC(x)-G_2(x)-G_3(x)=\frac{G_4(x)}{1-x}.$$
By substituting the expressions for $G_2(x),G_3(x),G_4(x)$ and simplifying, we obtain the stated generating function.
\end{proof}

\subsection{Case 69: $\{1234, 1324, 3412\}$}
A permutation $\pi=\pi_1\pi_2\cdots\pi_n\in S_n$ determines $n+1$ positions, called \emph{sites}, between its entries. The sites are denoted $1,2,\dots,n+1$ left to right. In particular, site $i$ is the space between $\pi_{i-1}$ and $\pi_{i}$, $2\le i \le n$. Site $i$ in $\pi$ is said to be {\em active} (with respect to $T$) if, by inserting $n+1$ into $\pi$ in site $i$, we get a permutation in $S_{n+1}(T)$, otherwise \emph{inactive}.

Say $j$ is an \emph{ascent index} for a permutation $\pi=\pi_1\pi_2 \cdots \pi_n$ of $[n]$ if $\pi_j<\pi_{j+1}$, and then $\pi_j$ is an \emph{ascent bottom}.

To construct the generating forest for $T$-avoiders, we first specify the labels. For $n\ge 2$, define the \emph{label} of $\pi \in S_n(T)$ to be $(k,s)$, where $k$ is the number of active sites in $\pi$ and $s$ is the number of active sites greater than the largest ascent index (LAI for short) with LAI taken to be 0 if there are no ascents, that is, if $\pi$ is decreasing.

For instance, the active sites for $\pi=12$ are $\{1,2,3\}$ and LAI $=1$, so the label for 12 is $(3,2)$. Also, 12
has three children $312$, $132$ and $123$ with active sites $\{1,3,4\},\,\{1,2,3\}$ and $\{1,2,3\}$, respectively,
and LAI $ = 2,1,2$, respectively; hence labels $(3,2),(3,2)$, and $(3,1)$. Similarly, all 3 sites for $21$ are active and  and LAI $=0$, so its label is $(3,3)$, and it has three children $321$, $231$ and $213$ with active sites $\{1,2,3,4\}$ in all three cases, and LAI $ = 0,1,2$, respectively; hence labels $(4,4),(4,3)$, and $(4,2)$.
An avoider $\pi\in S_n(T)$ has a label $(k,s)$ with $k=s$ only if $\pi$ is decreasing, in which case $k=s=n+1$.
Otherwise, $0\le s < k$.
\begin{proposition}\label{prop69}
The roots for the generating forest $\mathcal{T}$ of $S_n(T)$ are $12$ and $21$ with labels $(3,2)$ and $(3,3)$ respectively, and the succession rules for the labels of children,
in order of increasing insertion site, are given by
\[
\begin{array}{llll}
(k,s) & \rightsquigarrow & (1,0)\ (2,0)\ \dots \ (k,0) & \textrm{\quad for $s=0$ and $k\ge 1$,}\\[1mm]
  & \rightsquigarrow & (2,1)\ (2,0)\ (3,0)\ \dots\ (k-1,0)\  (k,1) & \textrm{\quad for $s=1$ and $k\ge 2$,}\\[1mm]
  & \rightsquigarrow & (s+1,s)\ (3,1)\ (4,1)\ \dots \ (k-s+1,1)\  & \\
 & & (k-s+2,2)\ (k-s+2,1)\ (k-s+3,1)\ \dots \ (k,1) & \raisebox{1.5ex}[0pt]{\textrm{\quad for $2\le s \le k-1$,}} \\[1mm]
  & \rightsquigarrow & (k+1,k+1)\  (k+1,k) \ \dots \  (k+1,2) & \textrm{\quad for $s=k$,}
\end{array}
\]
\end{proposition}

As an example (bullets denote active sites, an underscore denotes last ascent bottom), the label of
$\pi=\,\textrm{\raisebox{.2ex}{\tiny{$\bullet$}}}\, 3 \,\textrm{\raisebox{.2ex}{\tiny{$\bullet$}}}\,
\underline{2}\,\textrm{\raisebox{.2ex}{\tiny{$\bullet$}}}\, 6 \,\textrm{\raisebox{.2ex}{\tiny{$\bullet$}}}\,
5\, 4\, 1 \in S_6(T)$ is $(k,s)=(4,2)$; its children are
$\,\textrm{\raisebox{.2ex}{\tiny{$\bullet$}}}\,7\, 3 \,
\underline{2}\,\textrm{\raisebox{.2ex}{\tiny{$\bullet$}}}\, 6 \,\textrm{\raisebox{.2ex}{\tiny{$\bullet$}}}\,
5\, 4\, 1 $,
$\,\textrm{\raisebox{.2ex}{\tiny{$\bullet$}}}\,3\,\textrm{\raisebox{.2ex}{\tiny{$\bullet$}}}\, 7 \,
\underline{2} \,\textrm{\raisebox{.2ex}{\tiny{$\bullet$}}}\, 6 \,
5\, 4\, 1 $,
$\,\textrm{\raisebox{.2ex}{\tiny{$\bullet$}}}\,3\,  \,\textrm{\raisebox{.2ex}{\tiny{$\bullet$}}}\,
\underline{2} \,\textrm{\raisebox{.2ex}{\tiny{$\bullet$}}}\, 7 \,\textrm{\raisebox{.2ex}{\tiny{$\bullet$}}}\, 6 \,
5\, 4\, 1 $,
$\,\textrm{\raisebox{.2ex}{\tiny{$\bullet$}}}\,3\,  \,\textrm{\raisebox{.2ex}{\tiny{$\bullet$}}}\,
2 \,\textrm{\raisebox{.2ex}{\tiny{$\bullet$}}}\, \underline{6} \,\textrm{\raisebox{.2ex}{\tiny{$\bullet$}}}\, 7 \,
5\, 4\, 1 $, in that order, with labels $(3,2), (3,1),\,(4,2),\,(4,1)$ respectively.

The proof of Proposition \ref{prop69} is based on induction by a routine checking of cases, and is left to the reader. We note a few properties of the active sites for $\pi \in S_n(T)$. Site 1 is always active. If site $n+1$ is active, then so is site $n$. The active sites always form either an interval of integers (necessarily an initial segment of the positive integers) or a pair of intervals of integers. In the latter case, at least one of the intervals is of length 1. For example, $5 3 2 4 1 6 \in S_6(T)$ has active sites $\{1,4,5,6\}$ and
$5 3 4 2 7 6 1 \in S_7(T)$ has active sites $\{1,2,5\}$.

{\bf Enumeration}: Let $A_{k,s}=A_{k,s}(t)$ be the generating function for the number of vertices labeled $(k,s)$ in level $n$ in the generating forest $\mathcal{T}$, where the roots 12 and 21 are at level $2$. Define $$A(u,v)=A(t;u,v)=\sum_{k\geq1,s\geq0}A_{k,s}u^kv^s,\qquad L(u,v)=\frac{u^3v^3t^2}{1-uvt}.$$
Proposition \ref{prop69} leads to
\begin{align}
A(u,v)&=t^2u^3v^2(1+v)+\frac{ut}{1-u}\left(A(1,0)-A(u,0)\right)+vt\frac{d}{dv}A(u,v)\mid_{v=0}\notag\\
&+\frac{t}{1-u}\left(u^2\frac{d}{dv}A(1,v)\mid_{v=0}-\frac{d}{dv}A(u,v)\mid_{v=0}\right)+utA(1,uv)-utL(1,uv)-utA(1,0)\notag\\
&+\frac{vu^3t}{1-u}\left(A(1,1)-A(1,0)-\frac{d}{dv}A(1,v)\mid_{v=0}-L(1,1)\right)\label{eq69m1}\\
&-\frac{uvt}{1-u}\left(A(u,1)-A(u,0)-\frac{d}{dv}A(u,v)\mid_{v=0}-L(u,1)\right)\notag\\
&+u^2v^2t\left(A(u,1/u)-A(u,0)-\frac{1}{u}\frac{d}{dv}A(u,v)\mid_{v=0}-L(1,1)\right)\notag\\
&+\frac{u^4v^2t^3}{(1-v)(1-ut)}-\frac{u^4v^5t^3}{(1-v)(1-uvt)}.\notag
\end{align}
By substituting $v=1/u$ into \eqref{eq69m1}, we obtain
\begin{align*}
A(u,1/u)=\frac{t(t-u^2t+uA(1,1)-A(u,1))}{(1-u)(1-x)},
\end{align*}
which implies
\begin{align}
A(u,v)&=t^2u^3v^2(1+v)+\frac{ut}{1-u}\left(A(1,0)-A(u,0)\right)+vt\frac{d}{dv}A(u,v)\mid_{v=0}\notag\\
&+\frac{t}{1-u}\left(u^2\frac{d}{dv}A(1,v)\mid_{v=0}-\frac{d}{dv}A(u,v)\mid_{v=0}\right)+utA(1,uv)-utL(1,uv)-utA(1,0)\notag\\
&+\frac{vu^3t}{1-u}\left(A(1,1)-A(1,0)-\frac{d}{dv}A(1,v)\mid_{v=0}-L(1,1)\right)\label{eq69m2}\\
&-\frac{uvt}{1-u}\left(A(u,1)-A(u,0)-\frac{d}{dv}A(u,v)\mid_{v=0}-L(u,1)\right)\notag\\
&+u^2v^2t\left(\frac{t(t-u^2t+uA(1,1)-A(u,1))}{(1-u)(1-x)}-A(u,0)-\frac{1}{u}\frac{d}{dv}A(u,v)\mid_{v=0}-L(1,1)\right)\notag\\
&+\frac{u^4v^2t^3}{(1-v)(1-ut)}-\frac{u^4v^5t^3}{(1-v)(1-uvt)}.\notag
\end{align}
Substitute $v=0$ into \eqref{eq69m2} and into the derivative of \eqref{eq69m2} respect to $v$.
Solve the resulting system for the variables $\frac{d}{dv}A(u,v)\mid_{v=0}$ and $\frac{d}{dv}A(1,v)\mid_{v=0}$ to get
\begin{align}
\frac{d}{dv}A(u,v)\mid_{v=0}&=\frac{2ut-2u+1}{1-2t}A(u,0)-\frac{ut}{(1-u)(1-2t)}A(u,1)-\frac{(1-u)ut}{1-2t}A(1,0)\label{eq69m3}\\
&+\frac{u^3t}{(1-u)(1-2t)}-\frac{u^3t^3}{(1-t)(1-2t)(1-ut)}.\notag
\end{align}
Thus, by using \eqref{eq69m3} twice, \eqref{eq69m2} can be written as
{\small\begin{align*}
A(u,v)&=\frac{(uvt^2-(t-1)^2)uvt}{(1-2t)(1-t)(1-u)}A(u,1)+\frac{(uv-1)(uvt-vt+2t-1)}{1-2t}A(u,0)+tuA(1,uv)\\
&+\frac{(uvt(t-1)+vt(1-2t)+(t-1)^2)tvu^3}{(1-t)(1-u)(1-2t)}A(1,1)+\frac{ut(1-uv)(uvt-vt+2t-1)}{1-2t}A(1,0)\\
&-\frac{(uv(3uv+2v-1)t^3-((u^2+5u+2)v^2+uv+1)t^2+(2(1+u)v^2+(v+1))t-v(1+v))u^3vt^2}{(1-ut)(1-t)(1-uvt)(1-2t)}
\end{align*}}
By a routine computer check, the solution of this equation is given by
\begin{align*}
A(u,v)&=\frac{t^2uK(u,v)}{(1-uvt)(1-t)^6(1-ut)^2(1-2t)^2},
\end{align*}
where
{\small\begin{align*}
K(u,v)&=u^2v^2(1+v)-u^2v(2(u+5)v^2+(u+8)v-1)t\\
&+u(u(u^2+19u+43)v^3+u(6u+29)v^2+(1-5u)v+1)t^2\\
&+(1-2u-u(3u^2-6u+1)v+u^2(u+3)(u-21)v^2-u^2(9u^2+80u+104)v^3)t^3\\
&+((1+u)(1-4u)+u(16u^2-3u-4)v-u^2(3u^2-39u-86)v^2+u^2(39u^2+192u+155)v^3)t^4\\
&+\bigl(14u^2+u-5+u(2u^3-26u^2+11u+2)v-3u^2(18u+25)v^2\\
&\qquad\qquad\qquad\qquad\qquad\qquad\qquad\qquad\qquad\qquad-2u^2(u^3+49u^2+142u+73)v^3\bigr)t^5\\
&+\bigl(2u^3-12u^2+8u-1-u(8u^3-12u^2+7u-5)v-u^2(2u^3-4u^2-43u-46)v^2\\
&\qquad\qquad\qquad\qquad\qquad\qquad\qquad\qquad\qquad\qquad+u^2(10u^3+147u^2+265u+85)v^3\bigr)t^6\\
&-u\bigl(6u^2-2u-3+(1-u^2)(8u-1)v-u(u+1)(8u^2-7u-20)v^2\\
&\qquad\qquad\qquad\qquad\qquad\qquad\qquad\qquad\qquad\qquad+u(18u^3+133u^2+154u+28)v^3\bigr)t^7\\
&+u^2(4u-4+(2u^2+2u-3)v-(10u^3+3u^2-16u-4)v^2+(14u^3+73u^2+52u+4)v^3)t^8\\
&-4u^3v((u^2+6u+2)v^2+(1-u^2)v+u-1)t^9+4u^4v^3t^{10}.
\end{align*}}
Since $A(1,1)=\sum_{n\geq2}|S_n(T)|t^n$ and so $F_T(t)=1+t+A(1,1)$, we get the following result.
\begin{theorem}\label{th69a}
Let $T=\{1234, 1324, 3412\}$. Then
$$F_T(x)=\frac{1-9x+35x^2-75x^3+98x^4-78x^5+34x^6-10x^7}{(1-2x)^2(1-x)^6}.$$
\end{theorem}

\subsection{Case 72: $\{1243,1324,3412\}$}
First, we look at $G_2(x)$. For $\pi=i\pi' n \pi''$ with 2 left-right maxima and $d\ge 0$
letters in $\pi''$, let $H_d(x)$ and $J_d(x)$ denote the
generating functions in the repective cases $i=n-1$ and $i<n-1$. Note $d\ge 1$ in case $i<n-1$.
Thus, with $H(x):=\sum_{d\ge 0}H_d(x)$ and $J(x):=\sum_{d\ge 1}J_d(x)$,
we have $G_2(x)=H(x)+J(x)$.

\begin{lemma}\label{lem72a1}
$$H(x)=\frac{x^2(1-9x+34x^2-69x^3+80x^4-54x^5+21x^6-3x^7)}{(1-x)^4(1-2x)^2(1-3x+x^2)}\,.$$
\end{lemma}
\begin{proof}
Clearly, $H_0(x)=x^2K(x)$ where $K(x)=F_{\{132,3412\}}(x)$, so by \cite[Seq. A001519]{Sl} we have that $H_0(x)=\frac{x^2(1-2x)}{1-3x+x^2}$. For $d\geq1$, by considering the position of the letter $n-2$ in $\pi''$ (either leftmost, rightmost, or in the middle), we obtain
$$
H_d(x)=xH_d(x)+x(H_d(x)-x^{d+2})+\frac{x^{d+2}}{(1-x)^d}+\frac{(K(x)-1)x^{d+4}}{1-x}+\frac{dx^{d+5}}{(1-x)^2}\,.
$$
Note that the last equation also holds for $d=0$. Now sum over $d\geq0$.
\end{proof}

By similar arguments, one can obtain the the following result for $J(x)$.
\begin{lemma}\label{lem72a2}
We have
   $$J(x)=\frac{x^3(1-4x+9x^2-11x^3+4x^4-2x^5)}{(1-x)^5(1-2x)^2}.$$ \qed
\end{lemma}

\begin{theorem}\label{th72a}
Let $T=\{1243,1324,3412\}$. Then
$$F_T(x)=1+\frac{x(1-11x+54x^2-152x^3+268x^4-311x^5+237x^6-109x^7+30x^8-4x^9)}{(1-x)^6(1-2x)^2(1-3x+x^2)}.$$
\end{theorem}
\begin{proof}
Let $G_m(x)$ be the generating function for $T$-avoiders with $m$ left-right maxima. Clearly, $G_0(x)=1$ and $G_1(x)=xF_T(x)$. By Lemma \ref{lem72a1} and Lemma \ref{lem72a2}, we have that
\begin{align}
G_2(x)=H(x)+J(x)=\frac{x^2(1-9x+36x^2-81x^3+107x^4-88x^5+50x^6-14x^7+x^8)}{(1-x)^5(1-2x)^2(1-3x+x^2)}.\label{eq72a1}
\end{align}
Now, let us write an equation for $G_m(x)$ with $m\geq3$. Let $\pi=i_1\pi^{(1)}\cdots i_m\pi^{(m)}\in S_n(T)$ with exactly $m$ left-right maxima. Since $\pi$ avoids $1324$,
$\pi^{(s)}<i_1$ for all $s=1,2,\ldots,m-1$ and, since $\pi$ avoids $1243$, $\pi^{(m)}<i_2$.
If $\pi^{(m-1)}=\emptyset$ then we have a contribution of $xG_{m-1}(x)$.
Otherwise, since $m\geq3$ and $\pi$ avoids $3412$, we see that $\pi^{(m)}<i_1$.
Moreover, since $\pi$ avoids $3412$ and $1324$, we see that
$\pi^{(1)}>\pi^{(2)}>\cdots>\pi^{(m)}$ and $\pi^{(2)}\cdots\pi^{(m)}$ is decreasing,
while $\pi^{(m-1)}$ is not empty, and $\pi^{(1)}$ avoids $132$ and $3412$.
Therefore, we have a contribution of $x^{m+1}K(x)/(1-x)^{m-1}$, where $K(x)$ is given in the proof of Lemma \ref{lem72a1}. Hence,
$$G_m(x)=xG_{m-1}(x)+\frac{x^{m+1}}{(1-x)^{m-1}}\frac{1-2x}{1-3x+x^2}.$$
Summing over all $m\geq3$ and using the expressions for $G_0(x)$ and $G_1(x)$, we obtain
$$F_T(x)-1-xF_T(x)-G_2(x)=x(F_T(x)-1-xF_T(x))+\frac{x^4}{(1-x)(1-3x+x^2)}\, .$$
Solve for $F_T(x)$ using \eqref{eq72a1} to complete the proof.
\end{proof}

\subsection{Case 75: $\{1243,1324,4231\}$} We make use of the following generating functions:
\begin{align*}
F_{\{132,231\}}(x)&=\frac{1-x}{1-2x},\\
F_{\{132,213,4231\}}(x)&=1+\frac{x}{(1-x)^2},\\
F_{\{231,1243,1324\}}(x)&=1+\frac{(x^4-6x^3+7x^2-4x+1)x}{(1-2x)(1-x)^4},
\end{align*}
denoted, respectively, $L$, $A$ and $B$ for short (all three follow from the main result in \cite{MV}).

Let $G_m(x)$ be the generating function for $T$-avoiders $\pi=i_1\pi^{(1)}\cdots i_m\pi^{(m)}$ with $m$ left-right maxima. Clearly, $G_0(x)=1$ and $G_1(x)=xB$.  To get an explicit formula for $G_m(x)$, we look at the cases $m=2$ and $m\geq3$, and in the latter case we consider whether $\pi^{(m)}$ has a letter between $i_1$ and $i_2$ or not. We split the rather lengthy treatment into subsections.

\subsubsection{$H_m(x)$ and $H'_m(x)$}
Define $H_m(x)$ to be the generating function for permutations
$$\pi=(n-m)\pi^{(0)}(n-m+1)\pi^{(1)}\cdots n\pi^{(m)}\in S_n(T),\quad 1\le m \le n-1,$$
and define $H'_m(x)$ to be the generating function for permutations
$$\pi=(n-m)\pi^{(0)}(n-m+1)\pi^{(1)}\cdots n\pi^{(m)}(n-m-1)\in S_n(T),\quad 1\le m \le n-2.$$

\begin{lemma}\label{lem75a0}
For $m\ge 1$, $H'_m(x)=x^{m+2}L^{m+1}$ and $H_m(x)$ satisfies
\begin{align*}
H_m(x)&=x^{m+1}+\sum_{j=0}^{m-1}(x^{j+1}H_{m-j}(x))+\frac{mx^{m+3}A}{1-x}+mx^{m+2}(L-1/(1-x))\\
&+H'_m(x)+x^{m+2}(1/(1-x)^{m+1}-1)(A-1)+x^{m+2}(B-1).
\end{align*}
\end{lemma}
\begin{proof}
First, we treat $H'_m(x)$. Let $\pi=(n-m)\pi^{(0)}(n-m+1)\pi^{(1)}\cdots n\pi^{(m)}(n-m-1)\in S_n(T)$. In the case $n=m+2$, we have a contribution of $x^{m+2}$. Otherwise, the letter $n-m-2$ belongs to $\pi^{(j)}$, where $j=0,1,\ldots,m$. For each $j=0,1,\ldots,m-1$, we have a contribution $x^{j+1}H'_{m-j}$. For $j=m$, $\pi^{(m)}$ can be written as $\alpha(n-m-2)\beta$ where $\pi^{(0)}\cdots\pi^{(m-1)}\alpha<\beta<n-m-2$. The contribution for the case $\beta=\emptyset$ is $xH'_m(x)$ and for the case $\beta\neq\emptyset$ is $x^{m+3}(L-1)$. Hence,
$$H'_m(x)=x^{m+2}+xH'_m(x)+\cdots+x^mH'_1(x)+xH'_m(x)+x^{m+3}(L-1).$$
Setting $m=1$, we can solve for $H'_1(x)$ and then the result follows by induction on $m$.

Next, $H_m(x)$. Let $\pi=(n-m)\pi^{(0)}(n-m+1)\pi^{(1)}\cdots n\pi^{(m)}\in S_n(T)$. In the case $m=n-1$, we have a contribution of $x^{m+1}$. Otherwise, the letter $n-m-1$ belong to $\pi^{(j)}$ for some $j\in \{0,1,\ldots,m\}$. If this $j$ is $<m$, then $\pi{(0)}=\cdots=\pi^{(j-1)}=\emptyset$,
and $\pi^{(j)}$ can be written as $\alpha(n-m-1)\beta$ where
$\alpha<\beta\pi^{(j+1)}\cdots\pi^{(m)}$. Then we have the contributions $x^{j+1}H_{m-j}$, $\frac{x^{m+3}}{1-x}A$ and $x^{m+2}(L-1/(1-x))$ for the cases $\alpha=\emptyset$, $\alpha$ is nonempty decreasing, and $\alpha$ contains a rise, respectively. If the letter $n-m-1$ belongs to $\pi^{(m)}$ then $\pi^{(m)}$ has the form $\alpha(n-m-1)\beta$
and $\pi^{(0)}\cdots\pi^{(m-1)}\alpha$ is both decreasing and $<\beta$. So we have the contributions $H'_m(x)$, $x^{m+2}(1/(1-x)^{m+1}-1)(A-1)$, and $x^{m+2}(B-1)$, for the cases $\beta=\emptyset$, $\beta\neq\emptyset$ and $\pi^{(0)}\cdots\pi^{(m-1)}\alpha=\emptyset$, $\beta\neq\emptyset$ and $\pi^{(0)}\cdots\pi^{(m-1)}\alpha$ is not empty, respectively. Hence, $H_m(x)$ satisfies the claimed relation,
\end{proof}

\subsubsection{$D_k^m(x)$ and ${D'}_k^m(x)$}
 Define
\begin{itemize}
\item $D_k^m(x)$ to be the generating function for $T$-avoiders $\pi=i_1\pi^{(1)}\cdots i_m\pi^{(m)}$ such that $i_2=i_1+k+1$ and $\pi^{(m)}$ contains the subsequence $(i_1+k)(i_1+k-1)\cdots(i_1+1)$, for all $k\geq1$.
\item ${D'}_k^m(x)$ to be the generating function for $T$-avoiders $\pi=i_1\pi^{(1)}\cdots i_m\pi^{(m)}$ such that $i_2=i_1+k+3$ and $\pi^{(m)}$ contains the subsequence $(i_1+k+2)(i_1+k+1)\cdots(i_1+3)(i_1+1)(i_2+2)$, for all $k\geq0$.
\end{itemize}

\begin{lemma}\label{lem75a1}
For all $k\geq1$,
$$D_k^2(x)=x^{k+2}+xD_k^2(x)+\frac{(k-1)x^{k+3}}{1-x}+E_k(x)+F_k(x),$$
where
\begin{align*}
E_k(x)&=E'_k(x)+x^{k+3}(L-1),\\
E'_k(x)&=x^{k+3}+xE'_k(x)+\frac{(k-1)x^{k+4}}{1-x}+xE_k(x),\\
F_k(x)&=F'_k(x)+x^{k+3}(B-1)+\frac{x^{k+4}(A-1)(1+(k+1)\frac{x}{1-x})}{1-x}+\frac{(k+1)x^{k+4}(A-1)}{1-x},\\
F'_k(x)&=x^{k+3}+xF'_k(x)+\frac{kx^{k+4}}{1-x}+xF'_k(x)+x^{k+4}(L-1).
\end{align*}
\end{lemma}
\begin{proof}
The proof is based on considering all possible positions of the next smallest element whose position we have not yet fixed. Let us write an equation for $D_k^2(x)$. Let $\pi=i\pi'n\pi''\in S_n(T)$ such that $\pi''$ contains the subsequence $(i+k)(i+k-1)\cdots(i+1)$ and $n=i+k+1$. If $i=1$, then we have a contribution of $x^{k+2}$. Otherwise, let $i>1$. If the letter $i-1$ belongs to $\pi'$ then it is the leftmost letter of $\pi'$, so we have a contribution of $xD_k^2(x)$.  If the letter $i-1$ belongs to $\pi''$ and is on the left side of $i+2$, then there exists $j\in [3,k-1]$ such that $\pi=in(i+k)\cdots(i+j)(i-1)\cdots1(i+j-1)\cdots(i+1)$, which gives a
contribution of $\frac{x^{k+3}}{1-x}$, for all $j=3,4,\ldots,k-1$. Let us denote the contribution of the case when the letter $i-1$ is between the letters $i+2$ and $i+1$ (resp. on right side of $i+1$) by $E_k(x)$ (resp. $F_k(x)$). Then
$$D_k^2(x)=x^{k+2}+xD_k^2(x)+\frac{(k-1)x^{k+3}}{1-x}+E_k(x)+F_k(x).$$

To write an equation for $E_k(x)$,
let $\pi=i\pi'n\pi^{(k)}(i+k)\cdots\pi^{(2)}(i+2)\pi^{(1)}(i+1)\pi^{(0)}\in S_n(T)$ with $n=i+k+1$ and $\pi^{(1)}=\alpha(i-1)\beta$. Define $E'_k(x)$ to be the contribution in case $\beta=\emptyset$. If $\beta\neq\emptyset$, then $\pi=in(i+k)\cdots(i+2)(i-1)\beta(i+1)$ which implies a contribution of $x^{k+3}(L-1)$. Thus
\[
E_k(x)=E'_k(x)+x^{k+3}(L-1)\, .
\]
By considering the possible positions of the letter $i-2$ in $\pi=i\pi'n\pi^{(k)}(i+k)\cdots\pi^{(1)}(i+1)\in S_n(T)$ with $n=i+k+1$ and $\pi^{(1)}=\alpha(i-1)$, we obtain the equation
\[
E'_k(x)=x^{k+3}+xE'_k(x)+\frac{(k-1)x^{k+4}}{1-x}+xE_k(x)\,.
\]

To write an equation for $F_k(x)$, let $\pi=i\pi'n\pi^{(k)}(i+k)\cdots\pi^{(1)}(i+1)\pi^{(0)}\in S_n(T)$ with $n=i+k+1$ and $\pi^{(0)}=\alpha(i-1)\beta$.
Define $F'_k(x)$ to be the contribution of the case $\beta=\emptyset$.
When $\beta\neq\emptyset$, by considering the possible positions of the letter $i-2$ in $\pi$, we obtain the contributions $x^{k+3}(B-1)$ when $\pi'\pi^{(k)}\cdots\pi^{(1)}\alpha=\emptyset$, $\frac{x^{k+4}}{1-x}\big(1+(k+1)\frac{x}{1-x}\big)(A-1)$ when $\pi'$ is
nonempty (and necessarily decreasing), and $\frac{(k+1)x^{k+4}}{1-x}(A-1)$ when $\pi'=\empty$ and $\pi^{(k)}\cdots\pi^{(1)}\alpha\neq\emptyset$ (here either $\alpha\neq\emptyset$ and $\pi^{(k)}\cdots\pi^{(1)}=\emptyset$, or $\alpha=\emptyset$
and there exists $j\in [1,k]$ such that $\pi^{(j)}\neq\emptyset$ and $\pi^{(i)}=\emptyset$ for all $i\neq j$). Thus,
$$F_k(x)=F'_k(x)+x^{k+3}(B-1)+\frac{x^{k+4}(A-1)\big(1+(k+1)\frac{x}{1-x}\big)}{1-x}+\frac{(k+1)x^{k+4}(A-1)}{1-x}.$$

By considering the possible positions of the letter $i-2$ in $\pi=i\pi'n\pi^{(k)}(i+k)\cdots\pi^{(1)}(i+1)\in S_n(T)$ with $n=i+k+1$ and $\pi^{(0)}=\alpha(i-1)$, we find
that
\[
F'_k(x)=x^{k+3}+xF'_k(x)+\frac{kx^{k+4}}{1-x}+xF'_k(x)+x^{k+4}(L-1)\, ,
\]
which completes the proof.
\end{proof}

The next lemma, giving ${D'}_k^2(x)$, can be established by similar arguments.

\begin{lemma}\label{lem75a2}
Let $P_k(x)$ be the generating function for the number of permutations $\pi=i\pi'n\pi^{(k)}(i+k+2)\cdots\pi^{(1)}(i+3)\pi^{(0)}\in S_n(T)$ such that $n=i+k+3$, $\pi^{(0)}$ contains the subsequence $(i+1)(i+2)$. For all $k\geq0$, ${D'}_k^2(x)=\frac{1}{1-x}P_k(x)$, where
$$P_k(x)=x^{k+4}+xP_k(x)+\frac{kx^{k+5}}{1-x}+E''_k(x)+x^{k+4}(L-1)$$
and the generating function $E''_k(x)$ for the number of permutations $\pi=i\pi'n\pi^{(k)}(i+k+2)\cdots\pi^{(1)}(i+3)\pi^{(0)}\in S_n(T)$ such that $n=i+k+3$ and  $\pi^{(0)}=\alpha(i-1)\beta(i+1)(i+2)$ satisfies
$E''_k(x)=E'''_k(x)+x^{k+5}(L-1)$, where $E'''_k(x)$ is the generating function for the number of permutations $\pi=i\pi'n\pi^{(k)}(i+k+2)\cdots\pi^{(1)}(i+3)\pi^{(0)}\in S_n(T)$ such that $n=i+k+3$ and  $\pi^{(0)}=\alpha(i-1)(i+1)(i+2)$ which satisfies satisfies
$E'''_k(x)=x^{k+5}+xE'''_k(x)+kx^{k+6}/(1-x)+xE''_k(x)$.  \qed
\end{lemma}

Now, we are ready to write a formula for the generating function $G_2(x)$ for $T$-avoiders with $2$ left-right maxima.
\begin{proposition}\label{pro75a1}
The generating function $G_2(x)$ is given by
$$G_2(x)=\frac{x^2(1-8x+31x^2-71x^3+100x^4-93x^5+64x^6-32x^7+11x^8-2x^9)}{(1-x)^7(1-2x)^4}.$$
\end{proposition}
\begin{proof}
Let $\pi=i\pi'n\pi''\in S_n(T)$ with $2$ left-right maxima. If $n=i+1$, then the contribution is $H_1(x)$ by Lemma \ref{lem75a0}. If $n>i$, then the letters of $\pi''$ that are greater than $i$ form either a decreasing sequence $(i+k)(i+k-1)\cdots(i+1)$ or a decreasing sequence followed by an increasing sequence of at least two
terms $(i+k)(i+k-1)\cdots(i+1+s)(i+1)(i+2)\cdots(i+s)$ with $s\ge 2$.
By Lemmas \ref{lem75a1} and \ref{lem75a2}, we obtain the contributions $D_k^2(x)$ with $k\geq1$ and ${D'}_k^2(x)$ with $k\geq0$, respectively. Thus,
$$G_2(x)=H_1(x)+\sum_{k\geq1}D_k^2(x)+\sum_{k\geq0}{D'}_k^2(x).$$
After working out explicit expressions for $H_1(x)$, $\sum_{k\geq1}D_k^2(x)$ and $\sum_{k\geq0}{D'}_k^2(x)$ from Lemmas \ref{lem75a0}, \ref{lem75a1} and \ref{lem75a2}, respectively, we complete the proof.
\end{proof}

Lemmas \ref{lem75a1} and \ref{lem75a2} suggest a method to study the generating functions $D_k^m(x),\ k\geq1$ and ${D'}_k^m(x),\ k\geq0$.

\subsubsection{Formula for $D_1^m(x)$}
\begin{lemma}\label{lem75b1}
Let $m\geq3$. Then $D_1^m(x)=x^{m+1}L^m+(S_m+S'_m)/(1-x)$, where
$$S_m=\frac{1}{1-2x}(x^{m+2}+x^{m+3}\sum_{j=1}^{m-1}(1-x)^{-j}+x^{m+3}(L-1))$$
and
$$S'_m=x^{m+2}(B-1)+\frac{x^{m+3}(A-1)}{1-x}+x^{m+2}(1/(1-x)^m-1)(A-1)+\frac{x^{m+4}}{(1-x)^2}(A-1).$$
\end{lemma}
\begin{proof}
Let $\pi=i_1\pi^{(1)}\cdots i_m\pi^{(m)}\in S_n(T)$ with $m$ left-right maxima such that $\pi^{(m)}$ has exactly one letter between $i_1$ and $i_2$. So, $i_j=i_1+j$ for all $j=2,3,\ldots,m$ and $\pi^{(m)}=\alpha(i_1+1)\beta$.
We denote the generating function for permutations $\pi$ with $\beta=\emptyset$ by $E_1^m$, and we denote the generating function for permutations $\pi$ with $\beta=\emptyset$ and $i_1-1 \in \pi^{(j)}$ by $E_1^{m,j}$, for $j=2,3,\ldots,m$. Then, by considering the position of letter $i-1$ in $\pi$ with $\beta=\emptyset$, we obtain
$$E_1^m=x^{m+1}+xE_1^m+\sum_{j=2}^{m-1}E_1^{m,j}+E_1^{m,m},$$
where
$$E_1^{m,s}=x^{m+2}+xE_1^{m,s}+\cdots+xE_1^{m,m-1}+x^{m+3}(L-1)$$
with $E_1^{m,m}=xE_1^m+x^{m+2}(L-1)$. Hence, by induction on $s$, we have
$$E_1^{m,s}=\sum_{i=0}^{m-s-1}\binom{m-s-1}x^{i+1}(1-2x)^{-i-1}(x^{m+1}+x^{m+2}(L-1)).$$
Thus, $E_1^m=x^{m+1}L^m$.

Denote the generating function for permutations $\pi$ with $\beta=\beta'(i-1)$ by $S_m$, and the generating function for permutations $\pi$ with $\beta=\beta'(i-1)\beta''$ with $\beta''\neq\emptyset$ by $S'_m$. Clearly, $D_1^m=E_1^m+(S_m+S'_m)/(1-x)$. By considering the position of the letter $i-2$ in the permutations counted by $S_m$, we obtain
$$S_m=x^{m+2}+xS_m+x^{m+3}\sum{j=1}^{m-1}(1-x)^{-j}+xS_m+x^{m+3}(L-1).$$
By considering the four possibilities where $\pi^{(1)}\cdots\pi^{(m-1)}\alpha$ and $\beta'$ are empty or not, we complete the proof.
\end{proof}

\subsubsection{A formula for ${D'}_0^m$}
\begin{lemma}\label{lem75c0}
Let $m\geq3$. Then ${D'}_0^m(x)=\frac{x^{m+2}}{1-x}(L+x(L+L^2+\cdots+L^m-1))$.
\end{lemma}
\begin{proof}
As in the proof of Lemma \ref{lem75b1}, by considering the position of the letter $i_1-1$ in permutations
$\pi=i_1\pi^{(1)}\cdots i_m\pi^{(m)}\in S_n(T)$ with $m$ left-right maxima and $\pi^{(m)}=\alpha(i_1+1)\beta(i_1+2)$, we obtain
$${D'}_0^m(x)=x^{m+2}+x{D'}_0^m+\sum_{j=2}^m{D'}_0^{m,j},$$
where ${D'}_0^{m,j}$ is the generating function for the $T$-avoiders $\pi=i_1\pi^{(1)}\cdots i_m\pi^{(m)}\in S_n(T)$ with $m$ left-right maxima and $\pi^{(m)}=\alpha(i_1+1)\beta(i_1+2)$ such that the letter $i_1-1$ belongs to $\pi^{(j)}$. Again, by considering the position of the letter $i_1-2$, we have
$${D'}_0^{m,j}=x^{m+3}+x{D'}_0^{m,j}+x{D'}_0^{m,j+1}+...+x{D'}_0^{m,m-1}+x^{m+4}(L-1)$$
with (see Lemma \ref{lem75a0})
$${D'}_0^{m,m}=x^{m+3}(L-1)+x^{m+2}(L-1)+x^2H'_{m-1}(x)=x^{m+3}(L-1)+x^{m+2}(L-1)+x^{m+3}L^m.$$
By induction on $s$, we see that $${D'}_0^{m,s}=x^{m+3}L^{m-s+1}.$$ Thus, by substituting the expression for ${D'}_0^{m,j}$ into the equation for ${D'}_0^m$, we complete the proof.
\end{proof}

\subsubsection{A formula for $D_k^m$ and ${D'}_k^m$} In the next two lemmas we study the generating functions $D_k^m(x)$ with $k\geq2$ and ${D'}_k^m$ with $k\geq1$.

As before, by considering all possible positions of the next smallest element whose position we have not yet fixed, we obtain the following results.

\begin{lemma}\label{lem75b2}
Let $m\geq3$ and $k\geq2$. Define $S_k^{m,m-1}$ (resp. $S_k^{m,m}$, ${E}_k^m$, ${E'}_k^m$) to be the generating function for the $T$-avoiders $\pi=i_1\pi^{(1)}\cdots i_m\pi^{(m)}\in S_n(T)$ with $m$ left-right maxima such that the letters between $i_1$ and $i_2$ in $\pi^{(m)}$ form a decreasing sequence $j_1j_2\cdots j_k$ and the letter $i_1-1$ lies between $j_{k-1}$ and $j_k$ (resp. $i_1-1$ lies on the right side
of $j_k$, $\pi^{(m)}$ has no letter smaller than $i_1$ on the right side of the letter $j_1$, $\pi^{(m)}$ has at least one letter smaller than $i_1$ on the right side of the letter $j_1$). Then $D_k^m(x)=E_k^m+{E'}_k^m$, where $E_k^m=\frac{x^{m+k}}{(1-x)^m}$,
\begin{align*}
{E'}_k^m&=x{E'}_k^m+\frac{(k-2)x^{m+k+1}}{1-x}+S_k^{m,m-1}+{S'}_k^{m,m},\\
S_k^{m,m-1}&=x^{m+k+1}(L-1)+T_k^m,\\
T_k^m&=x^{m+k+1}+xT_k^m+x^{m+k+2}\sum_{j=2}^{m-1}(1-x)^{-j}+\frac{(k-1)x^{m+k+2}}{1-x}+xT_k^m+x^{m+k+2}(L-1),\\
S_k^{m,m}&=x^{m+k+1}(B-1)+x^{m+k+1}\frac{(k+1)x}{1-x}(A-1)+\frac{x^{m+k+2}}{1-x}(A-1)\left(1+\frac{(k+1)x}{1-x}\right)\\
&+x^{m+k+1}\left(\frac{1}{(1-x)^{m-2}}-1\right)\frac{A-1}{1-x}+\frac{x^{m+k+2}(A-1)}{(1-x)^2}\left(\frac{1}{(1-x)^{m-2}}-1\right)+{T'}_k^m,\\
{T'}_k^m&=x^{m+k+1}+x{T'}_k^m+x^{m+k+2}\sum_{j=2}^{m-1}(1-x)^{-j}+\frac{kx^{m+k+2}}{1-x}+x{T'}_k^m+x^{m+k+2}(L-1).
\end{align*}
\end{lemma}

\begin{lemma}\label{lem75c1}
Let $m\geq3$ and $k\geq1$.
Define $U_k^m$ to be the generating function for the $T$-avoiders $\pi=i_1\pi^{(1)}\cdots i_m\pi^{(m)}\in S_n(T)$ with $m$ left-right maxima such that the letters between $i_1$ and $i_2$ in $\pi^{(m)}$ form a subsequence $(i_2-1)\cdots(i_1+3)(i_1-1)(i_1+1)(i_1+2)$ of $k+1$ terms. Then
\begin{align*}
{D'}_k^m&=x^{m+k+2}+x{D'}_k^m+x^{m+k+3}\sum_{j=2}^{m-1}(1-x)^{-j}+\frac{kx^{m+k+3}}{1-x}\\
&+x^{m+k+3}(L-1)+U_k^m+x^{m+k+2}(L-1),\\
U_k^m&=x^{m+k+3}+xU_k^m+x^{m+k+4}\sum_{j=2}^{m-1}(1-x)^{-j}+\frac{kx^{m+k+4}}{1-x}\\
&+x^{m+k+4}(L-1)+xU_k^m.
\end{align*}
\end{lemma}

\subsubsection{A formula for $F_T(x)$}
Now, we are ready to find an explicit formula for the generating function $F_T(x)$ by  getting a formula for $G_m(x)$.

\begin{theorem}\label{th75a}
Let $T=\{1243,1324,4231\}$. Then
$$F_T(x)=\frac{x}{1-3x+x^2}-\frac{2-4x-4x^2-x^3}{(1-2x)^2}
+\frac{3-20x+55x^2-83x^3+74x^4-38x^5+12x^6-2x^7}{(1-x)^8}.$$
\end{theorem}
\begin{proof}
Fix $m\geq3$. Let $G'_m(x)$ be the generating function for the $T$-avoiders $\pi=i_1\pi^{(1)}\cdots i_m\pi^{(m)}\in S_n(T)$ with $m$ left-right maxima and $i_2>i_1+1$. Let $\gamma$ be the subsequence of letters of $\pi^{(m)}$ that are larger than $i_1$ and smaller than $i_2$. Since $\pi$ avoids $T$, $\gamma$ can be written as either
$(i_2-1)(i_2-2)\cdots(i_1+1)$ or $(i_2-1)(i_2-2)\cdots(i_1+s+1)(i_1+1)(i_1+2)\cdots(i_1+s)$ for $s\geq2$. It follows that
$$G_m(x)=H_{m-1}(x)+\sum_{k\geq0}D_k^m+\frac{1}{1-x}\sum_{k\geq1}{D'}_k^m.$$
Thus,
$$\sum_{m\geq3}G_m(x)=\sum_{m\geq2}H_m(x)+\sum_{m\geq3}\sum_{k\geq0}D_k^m+\frac{1}{1-x}\sum_{m\geq3}\sum_{k\geq1}{D'}_k^m.$$
By Lemmas \ref{lem75a0} and \ref{lem75b1}-\ref{lem75c1}, we obtain after simplifying
$$\sum_{m\geq3}G_m(x)=\frac{x^3p(x)}{(1-x)^8(1-2x)^2(1-3x+x^2)},$$
where $p(x)=1-10x+49x^2-149x^3+296x^4-403x^5+408x^6-322x^7+187x^8-74x^9+18x^{10}-2x^{11}$.
Since $$F_T(x)=1+xB+G_2(x)+\sum_{m\geq3}G_m(x),$$
and we have determined each of the summands,
the result follows.
\end{proof}

\subsection{Case 76: $\{3412,1324,2341\}$} Here we use the fact that $$F_{\{132,2341,3412\}}(x)=F_{\{213,2341,3412\}}(x)=\frac{1-3x+x^2}{(1-x)^2(1-2x)},$$ denoted $A$ for short (followed from the main result of see \cite{MV}).
\begin{lemma}\label{lem76a1}
The generating function $G_2(x)$ for $T$-avoiders with $2$ left-right maxima is given by
$$G_2(x)=\frac{x^2(1-7x+23x^2-40x^3+39x^4-22x^5+9x^6-4x^7)}{(1-x)^6(1-2x)^2}.$$
\end{lemma}
\begin{proof}
Suppose $\pi=i\al n\be \in S_n(T)$ with $2$ left-right maxima. If $i=1$, then $\pi=1n\be$
and where $\be$ avoids $\{213,2341,3412\}$, giving a contribution to $G_2(x)$ of $x^2A$.
Now suppose $i>1$. If $i-1$ is the leftmost letter of $\al$, we obtain a contribution of
$xG_2(x)$ by deleting $i-1$. If $i-1 \in \al$ and is not the leftmost letter, then  $\pi$
decomposes as $\pi=i\alpha'(i-1)\alpha''n\beta'\beta''$,
where $\beta'>i>i-1>\beta''>\alpha'>\alpha''$ (to avoid 1324) and $\beta''$ is decreasing
(to avoid 3412), as in Figure \ref{figAK2}.
\begin{figure}[htp]
\begin{center}
\begin{pspicture}(4,3.6)
\psset{xunit=.6cm}
\psset{yunit=.4cm}
\psline[linecolor=lightgray](4,2)(4,6.5)
\psline[linecolor=lightgray](2,4)(6,4)
\psline(0,2)(0,4)(2,4)(2,0)(4,0)(4,2)(0,2)
\psline(4,6.5)(6,6.5)(6,8.5)(4,8.5)(4,6.5)
\psline(6,4)(8,4)(8,6)(6,6)(6,4)
\psdots[linewidth=1.5pt](0,6.5)(2,6)(4,8.5)
\rput(1,3){\textrm{{\footnotesize $\al'\ne \emptyset$}}}
\rput(3,1){\textrm{{\footnotesize $\al''$}}}
\rput(5,7.5){\textrm{{\footnotesize $\be'$}}}
\rput(7,5){\textrm{{\footnotesize $\be''\downarrow$}}}
\rput(-.4,6.8){\textrm{{\footnotesize $i$}}}
\rput(1.2,6.0){\textrm{{\footnotesize $i\!-\!1$}}}
\rput(3.6,8.8){\textrm{{\footnotesize $n$}}}
\end{pspicture}
\caption{An avoider with 2 left-right maxima, $i-1$ before $n$ but not adjacent to $i$}\label{figAK2}
\end{center}
\end{figure}
Consider four cases according as $\alpha'',\beta''$ are empty or not:

- $\alpha'',\,\beta''$ both empty. Note that $\al'$  is nonempty and avoids $\{132,2341,3412\}$ and $\be'$ avoids $\{213,2341,3412\}$. So the contribution is $x^3(A-1)A$.

- $\alpha''=\emptyset,\ \beta''\ne \emptyset$. Here, $\be'$ is decreasing (to avoid 2341),
while $\al'$ is nonempty and avoids $\{132,2341,3412\}$. So the contribution is $x^4(A-1)/(1-x)^2$.

- $\alpha''\ne\emptyset,\ \beta'' = \emptyset$. Here, $\al'$ is decreasing (to avoid 2341)
and nonempty, $\al''$ is also decreasing (to avoid 3412) and nonempty, and $\be'$ avoids $\{213,2341,3412\}$. So the contribution is $x^5A/(1-x)^2$.

- $\alpha''\ne\emptyset,\ \beta'' \ne \emptyset$. Here, $\al'$ and $\be'$ are decreasing
(both to avoid 2341) and $\al'$ is nonempty, $\al''$ is also decreasing (to avoid 3412) and nonempty, and $\be''$ is nonempty (and decreasing, as always). So the contribution is $x^6/(1-x)^4$.

Next, suppose $i>1$ and $i-1 \in \be$ so that $\pi=i \al n \be' (i-1) \be''$.
Then $\be'>i$ since $j \in \be'$ with $j<i$ makes $inj(i-1)$ a $3412$.
For $d\ge 0$, let $B_d$ be the generating function for such permutations where $\be'$ has $d$ letters.
\begin{lemma}\label{lem76aB0}
We have
 $$B_0=\frac{x^3(1-5x+12x^2-13x^3+4x^4+4x^5-4x^6)}{(1-x)^4(1-2x)^2}\,.$$
\end{lemma}
\begin{proof}
First, let $K$ be the generating function for $T$-avoiders $i(i-2)(i-4)\pi' (i+1)(i-1)(i-3)\pi''$ with $2$ left-right maxima such that $i\geq5$ where $n=i+1$. Let us write an equation for $K$. When $i=5$ the contribution is $x^6$.
Otherwise, we consider the position of $i-5$. If $i-5$ is either the leftmost letter of $\pi'$
or of $\pi''$ then we have contribution of $xK$. Otherwise $\pi'=\alpha'(i-5)\alpha''$ with $\alpha'\neq\emptyset$ such that $i-5>\pi''>\alpha'>\alpha''$ with $\alpha''$ and $\pi''$ are decreasing. By consider whether $\alpha''$ is empty or not, we obtain the contribution $x^7(A-1+x^2/(1-x)^2)/(1-x)$. Thus
$$K=2xK+x^6+x^7(A-1+x^2/(1-x)^2)/(1-x)\,,$$
which leads to
$$K=\frac{(1-4x+6x^2-2x^3)x^6}{(1-2x)^2(1-x)^2}\,.$$

Next, let $K'$ be the generating function for $T$-avoiders $i(i-2)\pi'n(i-1)(i-3)\pi''$ with $2$ left-right maxima such that $n>i\geq4$. Let us write an equation for $K'$. When $i=4$ we have a contribution of $x^5A$. If $i-4$ is the leftmost letter of $\pi'$ then we have a contribution $K$, and if $i-4$ is leftmost letter of $\pi''$ then we have a contribution $xK'$. Otherwise $\pi'=\alpha'(i-4)\alpha''$  with $\alpha'\neq\emptyset$ such that $i-4>\pi''>\alpha'>\alpha''$ with $\alpha''$ and $\pi''$  decreasing. By considering whether $\alpha''$ is empty or not, we obtain the contribution $x^6(A-1+x^2/(1-x)^2)/(1-x)$. Thus $K'=x^5A+xK'+K+x^6(A-1+x^2/(1-x)^2)/(1-x)$, which leads to
$$K'=\frac{(1-4x+6x^2-2x^3-2x^4)x^5}{(1-2x)^2(1-x)^3}\,.$$

Next, let $K''$ be the generating function for $T$-avoiders $i(i-2)\pi'n(i-1)\pi''$ with $2$ left-right maxima such that $n>i\geq3$. By using similar techniques as for finding formulas $K$ and $K'$, we have
\begin{align*}
K''=x^4A+xK''+K'+x^5A(A-1)+\frac{x^7A}{(1-x)^2}+\frac{x^6(A-1)}{1-x}+\frac{x^8}{(1-x)^3}\,,
\end{align*}
which implies
$$K''=\frac{(1-3x+4x^2)x^4}{(1-2x)^2(1-x)^2}\,.$$

Lastly, let $K'''$ be the generating function for $T$-avoiders $i\pi'n(i-1)j(i-2)\pi''$ with
$2$ left-right maxima such that $n>j>i\geq3$.
By using similar techniques as for finding formulas $K$ and $K'$, we obtain
$$K'''=\frac{(1-x+2x^3)x^5}{(1-2x)(1-x)^2}\,.$$

Now, we are ready to write an equation for the generating function $B_0$ for $T$-avoiders
$i\pi'n(i-1)\pi''$ with $2$ left-right maxima such that $n>i\geq2$. Again, we decompose
the structure by looking at position of $i-2$. If $i=2$ then we have a contribution of $x^3A$,
if $i-2$ is leftmost letter of $\pi'$ then we have a contribution of $K''$,  if $i-2$ belongs
to $\pi''$ and is not the leftmost letter of it, we have a contribution
of $x^4A(A-1)+\frac{x^6A}{(1-x)^2}+\frac{x^5(A-1)}{1-x}+\frac{x^7}{(1-x)^3}$, and
if $i-2$ belongs to $\pi''$ we have a contribution of $K'''$. Thus,
$$B_0=x^3A+K''+x^4A(A-1)+\frac{x^6A}{(1-x)^2}+\frac{x^5(A-1)}{1-x}+\frac{x^7}{(1-x)^3}+K'''\,,$$
which completes the proof.
\end{proof}

One can similarly show (details omitted) that
\[
 B_1=\frac{x^4(1-2x+2x^2)(1-2x+2x^2-2x^4)}{(1-x)^4(1-2x)^2}\, ,
\]
and $B_d=xB_{d-1}$ for all $d\geq2$. Hence, the total contribution for the case $i>1$ and $i-1 \in \be$ is given by $B$, where $B-B_1-B_0=x(B-B_0)$, which leads to
$$B=\frac{x^3(1-5x+13x^2-17x^3+9x^4+2x^5-4x^6)}{(1-x)^5(1-2x)^2}.$$

Adding all the contributions,
$$G_2(x)=x^2A+xG_2(x)+x^3(A-1)A+ \frac{x^4(A-1)}{(1-x)^2} +\frac{x^5 A}{(1-x)^2}+\frac{x^6}{(1-x)^4}+B,$$
with solution for $G_2(x)$ as claimed.
\end{proof}

\begin{theorem}\label{th76a}
Let $T=\{3412,1324,2341\}$. Then
$$G_T(x)=\frac{1-10x+44x^2-110x^3+173x^4-176x^5+114x^6-45x^7+12x^8-4x^9}{(1-x)^7(1-2x)^2}.$$
\end{theorem}
\begin{proof}
Let $G_m(x)$ be the generating function for $T$-avoiders with $m$ left-right maxima. Clearly, $G_0(x)=1$, $G_1(x)=xF_T(x)$, and $G_2(x)$ is given above.

Now, let us write an equation for $G_m(x)$ with $m\geq3$. Let $\pi=i_1\pi^{(1)}\cdots i_m\pi^{(m)}\in S_n(T)$ with exactly $m$ left-right maxima. Since $\pi$ avoids $1324$ and $2341$, we can write $\pi$ as
$$\pi=i_1\pi^{(1)}i_2\pi^{(2)}i_3i_4\cdots i_m\gamma\beta,$$
where $\pi^{(1)}>\pi^{(2)}$ and $\gamma>i_{m-1}>\beta>i_{m-2}$. By considering the four possibilities for whether $\pi^{(2)},\beta$ are empty of not, we obtain the contributions $x^mA^2$ (both empty), $\frac{x^{m+1}}{(1-x)^2}A$ ($\pi^{(2)}$ empty, $\be$ nonempty), $\frac{x^{m+1}}{(1-x)^2}A$ ($\pi^{(2)}$ nonempty, $\be$ empty) and $\frac{x^{m+2}}{(1-x)^4}$ (both nonempty). Thus,
$$G_m(x)=x^m\left(A+\frac{x}{(1-x)^2}\right)^2.$$
Therefore,
$$F_T(x)-1-xF_T(x)-G_2(x)=\sum_{m\geq3}G_m(x)=\frac{x^3}{1-x}\left(A+\frac{x}{(1-x)^2}\right)^2.$$
Substituting for $G_2(x)$ and solving for $F_T(x)$ completes the proof.
\end{proof}

\subsection{Case 80: $\{1324,2341,3421\}$}
In order to study this case, we need the following lemmas.
\begin{lemma}\label{lem80a1}
Let $T=\{1324,2341,3421\}$. Let $G_2(x)$ be the generating function for the number of permutations $\pi=i\pi'n\pi''\in S_n(T)$ with exactly $2$ left-right maxima. Then
$$G_2(x)=\frac{x^2(1-8x+28x^2-52x^3+50x^4-22x^5+5x^6)}{(1-x)^4(1-2x)^2(1-3x+x^2)}.$$
\end{lemma}
\begin{proof}
In order to find a formula for the generating function $G_2(x)$, we refine it as follows. Let $G_2(x;d)$ be the generating function for the number of permutations $\pi=i\pi'n\pi''\in S_n(T)$ with exactly $2$ left-right maxima and where $\pi''$ has exactly $d$ points smaller than $i$. Clearly,
$$G_2(x;0)=x^2F_{\{132,2341,3421\}}(x)F_{\{213,2341,3421\}}(x)=x^2K(x)^2,$$
where $K(x)=\frac{1-3x+3x^2}{(1-x)^2(1-2x)}$ (where we leave the proof to the reader). For $d=1$, our permutations can be written as $i(i-1)\cdots (i'+1)\alpha n(n-1)\cdots i''i'\beta$ with $\alpha<i'$ and $i<\beta<i''$. By considering whether $\alpha$ and $\beta$ are empty or not, one can show that $$G_2(x;1)=\frac{x^3}{(1-x)^2}+\frac{x^3(K(x)-1)^2}{(1-x)^2}+\frac{2x^4(1-x+x^2)}{(1-x)^4(1-2x)}.$$

For $d\geq2$, our permutations can be written as $$\pi=i\alpha^{(1)}\cdots\alpha^{(d)}n\beta^{(1)}j_1\cdots\beta^{(d)}j_d\beta^{(d+1)}$$
such that all the letters that are greater than $j_1$ in $i\alpha^{(1)}\cdots\alpha^{(d)}$ are decreasing and  $n\beta^{(1)}\cdots\beta^{(d)}$ is decreasing. We denote all the letters between $j_1$ and $j_2$ in $\pi'$ by $\delta$. Now, let us write an equation for $G_2(x;d)$. If $\delta=\beta^{(2)}=\emptyset$, then we have a contribution of $xG_2(x;d-1)$. If $\delta\neq\emptyset$, then $\pi$ can be written as
$$\pi=i\alpha^{(1)}\cdots\alpha^{(d-1)}\gamma'\gamma''n\beta^{(1)}j_1\cdots\beta^{(d)}j_d\beta^{(d+1)}$$
such that $\alpha^{(1)}\alpha^{(2)}\cdots\alpha^{(d-1)}\gamma'$ is decreasing, $j_{d+2-s}>\alpha^{(s)}>j_{d+1-s}$ for all $s=1,2,\ldots,d-1$ with $j_{d+1}=i$, $j_2>\gamma'=\delta>j_1$ and $j_1>\gamma''$ where $\gamma''$ avoids $132,2341,3421$ and $\beta^{(d+1)}$ avoids $213,2341,3421$. Thus, we have a contribution of $\frac{x^{d+3}}{(1-x)^{2d}}K(x)^2$. Otherwise, $\delta=\emptyset$ and $\beta^{(2)}\neq\emptyset$, so similarly, we have a contribution of $\frac{x^{d+3}}{(1-x)^{2d-1}}K(x)^2$. Therefore, for all $d\geq2$,
$$G_2(x;d)=xG_2(x;d-1)+\frac{x^{d+3}}{(1-x)^{2d}}K(x)^2+\frac{x^{d+3}}{(1-x)^{2d-1}}K(x)^2.$$
By summing over all $d\geq2$, we obtain
$$G_2(x)-G_2(x;0)-G_2(x;1)=x(G_2(x)-G_2(x;0))+\frac{x^5K(x)^2}{(1-x)^2(1-3x+x^2)}+\frac{x^5K(x)^2}{(1-x)(1-3x+x^2)}.$$
By using the values of $G_2(x;0)$ and $G_2(x;1)$, and then solving for $G_2(x)$, we complete the proof.
\end{proof}

\begin{theorem}\label{th80a}
Let $T=\{1324,2341,3421\}$. Then
$$F_T(x)=\frac{1-7x+20x^2-29x^3+25x^4-10x^5+2x^6}{(1-x)^5(1-3x+x^2)}.$$
\end{theorem}
\begin{proof}
Let $G_m(x)$ be the generating function for $T$-avoiders with $m$ left-right maxima. Clearly, $G_0(x)=1$ and $G_1(x)=xF_T(x)$. Note that the generating function $G_2(x)$ is given in Lemma \ref{lem80a1}.

Let $m=3$. Each permutation of $\pi\in S_n(T)$ with exactly $3$ left-right maxima can be represented as either
\begin{itemize}
\item $\pi=i_1\pi^{(1)}(i_1+1)n\pi^{(3)}$ with $\pi^{(1)}<i_1$ and $\pi^{(3)}>i_1+1$, where $\pi^{(1)}$ avoids $132,2341,3421$ and $\pi^{(3)}$ avoids $213,2341,3421$.
\item $\pi=i_1\pi^{(1)}i_2n(n-1)\cdots(i_2+1)(i_1+1)(i_1+2)\cdots(i_2-1)$, where $i_1+1<i_2$ and $\pi^{(1)}$ avoids $132,2341,3421$.
\item $\pi=i_1(i_1-1)\cdots i'(i_1+1)12\cdots(i'-1)n\pi^{(3)}$, where $i'>1$ and $\pi^{(3)}$ avoids $213,2341,3421$.
\item $\pi=i_1(i_1-1)\cdots i'i_212\cdots(i'-1)n(n-1)\cdots(i_2+1)(i_1+1)(i_1+2)\cdots(i_2-1)$, where $i_1+1<i_2$ and $i'>1$.
\end{itemize}
By finding the corresponding generating functions, we obtain
$$G_3(x)=x^3K(x)^2+\frac{2x^4K(x)}{(1-x)^2}+\frac{x^5}{(1-x)^4}=\frac{x^3}{(1-2x)^2}.$$

Now, let us write an equation for $G_m(x)$ with $m\geq4$. Let $\pi=i_1\pi^{(1)}\cdots i_m\pi^{(m)}\in S_n(T)$ with exactly $m$ left-right maxima. Since $\pi$ avoids $1324$ and $2341$, we see that $\pi^{(s)}<i_1$ for all $s=2,3,\ldots,m-1$ and $\pi^{(m)}>i_{m-2}$. Hence, we have a contribution of $xG_{m-1}(x)$ (by removing the letter $i_{m-2}$). Thus, $G_m(x)=x^{m-3}G_3(x)$ for all $m\geq3$. Therefore,
$$F_T(x)-G_0(x)-G_1(x)-G_2(x)=\sum_{m\geq3}G_m(x)=\frac{1}{1-x}G_3(x),$$
which, by substituting the values of $G_j(x)$, $j=0,1,2,3$, we complete the proof.
\end{proof}

\subsection{Case 84: $\{4231,1324,2341\}$}
\begin{lemma}\label{lem84a1}
Let $\pi=i\alpha n\beta\in S_n(T)$ with exactly $2$ left-right maxima, $i\geq2$ and $i-1$
in $\alpha$. Then the generating function $H(x)$ for such permutations is given by
$$H(x)=xG_2(x)+x^3(L(x)-1)\left(\frac{L(x)}{1-x}+L(x)-\frac{1}{1-x}\right),$$
where $G_2(x)$ is the generating function for permutations in $S_n(T)$ with exactly $2$ left-right maxima and $L(x)=\frac{1-x}{1-2x}$.
\end{lemma}
\begin{proof}
Let us write an equation for $H(x)$. If $i-1$ is the first letter of $\alpha$, then the contribution is $xG_2(x)$. Otherwise, $i-1$ is the last letter of $\al$ (to avoid 1324 and 4231), and $\pi$ can be written as $\pi=i\alpha'(i-1)n\beta'\beta''$ such that $\emptyset\neq\alpha'<\beta''<i<\beta'$, where $\alpha'$ avoids both $132$ and $231$ and $\be''$ avoids both $213$ and $231$. Thus, by considering whether $\beta'=n(n-1)\cdots(i+1)$ or not, we get contributions of $x^3\big(L(x)-1\big)L(x)/(1-x)$ and $x^3\big(L(x)-1\big)\big(L(x)-1/(1-x)\big)$
(recall that $F_{\{132,231\}}=F_{\{213,231\}}=L(x)$). Hence,
$$H(x)=xG_2(x)+x^3\big(L(x)-1\big)\left(\frac{L(x)}{1-x}+L(x)-\frac{1}{1-x}\right),$$
as required.
\end{proof}

\begin{lemma}\label{lem84a2}
Let $\pi=i\alpha n\beta\in S_n(T)$ with exactly $2$ left-right maxima, $i\geq2$, and $i-1$ in $\beta$. Then the generating function $H'(x)$ for such permutations is given by
$$H'(x)=\frac{x^3(1-4x+9x^2-12x^3+6x^4-x^5)}{(1-3x+x^2)(1-2x)^2(1-x)^2}.$$
\end{lemma}
\begin{proof}
In order to write an equation for $H'(x)$, we define $A_d(x)$ to be the generating function for permutations $\pi=i\alpha n\beta'(i-1)\beta''\in S_n(T)$ with exactly 2 left-right maxima, $i\geq2$, and such that $\beta'$ has $d$ letters that are greater than $i$ (these $d$ letters form a decreasing subsequence because $\pi$ avoids $4231$). We leave to the reader to show that
$A_0(x)=\frac{x^3(1-3x+5x^2-4x^3)}{(1-x)(1-2x)^3}$, $A_1(x)=\frac{x^4(1-3x+4x^2-5x^3+4x^4)}{(1-x)(1-2x)^4}$, and for $d\geq2$,
$$A_d(x)-xA_{d-1}(x)=\frac{x^{d+4}(1-x)^d}{(1-2x)^{d+3}}.$$
Summing over all $d\geq0$, we obtain
$$H'(x)=\frac{x^3(1-4x+9x^2-12x^3+6x^4-x^5)}{(1-3x+x^2)(1-2x)^2(1-x)^2},$$
as required.
\end{proof}

\begin{lemma}\label{lem84a3}
The generating function $G_2(x)$ for permutations in $S_n(T)$ with exactly $2$ left-right maxima is given by
$$G_2(x)=\frac{x^2(1-7x+22x^2-35x^3+29x^4-16x^5+6x^6-x^7)}{(1-3x+x^2)(1-2x)^2(1-x)^3}.$$
\end{lemma}
\begin{proof}
Permutations in $S_n(T)$ with exactly 2 left-right maxima whose first letter is 1 (and hence second letter is $n$) have the generating function $x^2L(x)$, where $L(x)=F_{\{231,213\}}(x)=\frac{1-x}{1-2x}$. Thus, by Lemmas \ref{lem84a1} and \ref{lem84a2}, we obtain
\begin{align*}
G_2(x)&=x^2L(x)+xG_2(x)+x^3(L(x)-1)\left(\frac{L(x)}{1-x}+L(x)-\frac{1}{1-x}\right)\\
&+\frac{x^3(1-4x+9x^2-12x^3+6x^4-x^5)}{(1-3x+x^2)(1-2x)^2(1-x)^2},
\end{align*}
and solving for $G_2(x)$ completes the proof.
\end{proof}

\begin{theorem}\label{th84a}
Let $T=\{4231,1324,2341\}$. Then
$$F_T(x)=\frac{1-9x+33x^2-62x^3+64x^4-36x^5+7x^6}{(1-3x+x^2)(1-2x)^2(1-x)^3}.$$
\end{theorem}
\begin{proof}
Let $G_m(x)$ be the generating function for $T$-avoiders with $m$
left-right maxima. Clearly, $G_0(x)=1$ and $G_1(x)=xF_{\{231,1324\}}(x)=
x \frac{1 - 4 x + 5 x^2 - x^3)}{(1 - 2 x)^2 (1 - x)}$ \cite{MV}, and $G_2(x)$
is given by Lemma \ref{lem84a3}. Now suppose $m\geq3$ and
let $\pi=i_1\pi^{(1)}i_2\pi^{(2)}\cdots i_m\pi^{(m)}\in S_n(T)$ with exactly $m$ left-right maxima. Since $\pi$ avoids $1324$ and $2341$, we see that $\pi^{(3)}\cdots\pi^{(m-1)}=\emptyset$, $\pi^{(1)}>\pi^{(2)}$, and $\pi^{(m)}=\alpha\beta$ with $\alpha>i_{m-1}>\beta>i_{m-2}$. By considering the four possibilities, $\pi^{(2)},\,\beta$ empty or not, we obtain the  contributions $x^mL(x)^2$, $\frac{x^m}{1-x} L(x)\big(L(x)-1\big)$,
$\frac{x^m}{1-x} L(x)\big(L(x)-1\big)$ and $\frac{x^m}{(1-x)^2}\big(L(x)-1\big)^2$. Thus,
$$G_m(x)=x^m\left(L(x)+\frac{L(x)-1}{1-x}\right)^2.$$
Summing over $m\geq3$, we obtain
$$F_T(x)-1-G_1(x)-G_2(x)=\frac{x^3}{1-x}\left(L(x)+\frac{L(x)-1}{1-x}\right)^2\,.$$
Using the expressions for $G_1(x),G_2(x)$, we complete the proof.
\end{proof}

\subsection{Case 86: $\{3412,2431,1324\}$}
In this subsection, let $A=\frac{1-2x}{1-3x+x^2}$ denote the generating function
for $F_{\{132,3412\}}(x)$ and let $B=\frac{1-3x+3x^2}{(1-x)^2(1-2x)}$ denote the
generating function for $F_{\{213,2431,3412\}}(x)$ (they can be derived from results in \cite{MV}).

\begin{lemma}\label{lem86a1}
Let $H(x)$ be the generating function permutations in $S_n(T)$ whose first letter is $n-1$. Then
$$H(x)=\frac{x^2(x^4-x^3+5x^2-4x+1)}{(1-x)(1-2x)(1-3x+x^2)}.$$
\end{lemma}
\begin{proof}
Refine $H(x)$ to $H_d(x)$, the generating function for $T$-avoiders $\pi=(n-1)\pi'n\pi''$
where $\pi''$ has $d$ letters.
Since $\pi$ avoids $3412$, $\pi''$ is a decreasing subsequence say $j_1j_2\cdots j_d$.
Clearly, $H_0(x)=x^2A$. If $d\geq2$, then there is no letter between $j_1$ and
$j_d$ in $\pi'$, otherwise $\pi$ contains $2431$. Thus $H_d(x)=x^{d-1}H_1(x)$, for all $d\geq1$.
Now, let us compute $H_1(x)$. If $j_1=1$, then we have a contribution of $x^3A$.
So, we can assume that $j_1>1$. Define $H_1(x;e,d)$ to be the generating function
for $T$-avoiders $\pi=(n-1)\alpha1\beta nj_1$, where $\beta$ has $e$ letters smaller
than $j_1$ and $d$ letters greater than $j_1$. Clearly, $\beta$ is an increasing subsequence.
Now let us examine the following cases:
\begin{itemize}
\item $e=d=0$. We have a contribution of $xH_1(x)$.
\item $d=0$ and $e\geq1$. Here we see that $\alpha=\alpha'\alpha''$ with $\alpha'$ a decreasing subsequence and $\alpha''<e_1$, where $e_1$ is the smallest letter in $\beta$. So we have a contribution of $\frac{x^{e+4}}{(1-x)^{e+1}}A$.
\item $e=0$ and $d\geq1$. Here we see that either $\alpha=\alpha'\alpha''$ with $\alpha'$ decreasing and $\alpha''<j_1$, or $\alpha=\alpha'\alpha''(j_1+1)\alpha'''$ with $\alpha'$  decreasing, $j_1>\alpha''>\alpha'''$. So we have a contribution of $x^4\left(\frac{A}{1-x}+\frac{x(A-1)}{(1-x)^2}\right)\frac{x^d}{(1-x)^d}$.
\item $e,d\geq1$. Here we see that $\alpha=\alpha'\alpha''$ with $\alpha'$  decreasing
and $\alpha''<e_1$, where $e_1$ is the smallest letter in $\beta$.
So we have a contribution of
    $\frac{x^{d+e+6}A}{(1-x)^{d+e+1}}$.
\end{itemize}
Hence,
$$(1-x)H_1(x)=x^3A+\sum_{e\geq1}\frac{x^{e+4}}{(1-x)^{e+1}}A+\sum_{d\geq1}x^4\left(\frac{A}{1-x}+\frac{x(A-1)}{(1-x)^2}\right)\frac{x^d}{(1-x)^d}+
\sum_{d,e\geq1}\frac{x^{d+e+6}A}{(1-x)^{d+e+1}},$$
which leads to
$$H_1(x)=\frac{x^3(1-3x+3x^2+x^3)}{(1-2x)(1-3x+x^2)}.$$
Since $H(x)=x^2A(x)+\sum_{d\geq1}H_d(x)=x^2A(x)+\frac{1}{1-x}H_1(x)$, the result follows.
\end{proof}

\begin{lemma}\label{lem86a2}
The generating function for $T$-avoiders with $2$ left-right maxima is given by
$$G_2(x)=\frac{x^2(1-5x+10x^2-8x^3+x^4+x^5-x^6)}{(1-x)^3(1-2x)(1-3x+x^2)}.$$
The generating function for $T$-avoiders with $3$ left-right maxima is given by
$$G_3(x)=\frac{x^3(1-4x+4x^2+3x^3-5x^4-3x^5+4x^6-x^7)}{(1-2x)(1-x)^4(1-3x+x^2)}.$$
\end{lemma}
\begin{proof}
We treat $G_2(x)$ and leave the similar derivation of $G_3(x)$ to the reader. Let $\pi=i\pi'n\pi''\in S_n(T)$ with 2 left-right maxima. By Lemma \ref{lem86a1}, the contribution for the case $i=n-1$ is $H(x)$. So let us assume that $i<n-1$, that is, $\pi''$ contains the letter $i+1$. Since $\pi$ avoids $2431$, we can write $\pi$ as $\pi=i\alpha'n\alpha''\beta'$ where $\alpha'\alpha''<i<\beta'$, $\beta'\neq\emptyset$ and  $\alpha''$ is decreasing. If $\alpha''=\emptyset$, then we have contribution of $x^2A(B-1)$. If $\alpha''$ has exactly one letter then we have a contribution of $\frac{x^3}{1-x}A(B-1)+x^3\big(A-1/(1-x)\big)(B-1)$ corresponding to the cases $\alpha'$ decreasing or not. If $\alpha''$ has at least two letters, then $\beta'>i>\alpha'>\alpha''$, so the contribution is $x^4A(B-1)/(1-x)$.
By adding all the contributions, we complete the proof.
\end{proof}

\begin{theorem}\label{th86a}
Let $T=\{3412,2431,1324\}$. Then
$$F_T(x)=\frac{1-7x+19x^2-24x^3+16x^4-4x^5-x^6+2x^7}{(1-x)^3(1-2x)(1-3x+x^2)}.$$
\end{theorem}
\begin{proof}
Let $G_m(x)$ be the generating function for $T$-avoiders with $m$
left-right maxima. Clearly, $G_0(x)=1,\ G_1(x)=xF_T(x)$, and $G_2(x),\,G_3(x)$ are given in Lemma \ref{lem86a2}. Now let us write an equation for $G_m(x)$ with $m\geq4$. Suppose $\pi=i_1\pi^{(1)}i_2\pi^{(2)}\cdots i_m\pi^{(m)}\in S_n(T)$ with exactly $m$ left-right maxima. Since $\pi$ avoids $1324$ we see that $\pi^{(j)}<i_1$ for all $j=1,2,\ldots,m-1$. With $\alpha=\pi^{(m-1)}$ and
$\beta$ the subsequence of all letters between $i_{m-2}$ and $i_{m-1}$ in $\pi^{(m)}$, we have contributions as follows:
\begin{itemize}
\item If $\alpha=\beta=\emptyset$, then we have $xG_{m-1}(x)$;
\item If $\alpha=\emptyset$ and $\beta\neq\emptyset$, then we have $x^{m+1}A/(1-x)^{m-2}$;
\item If $\alpha\neq\emptyset$ and $\beta=\emptyset$, then we have $x^{m+1}AB/(1-x)^{m-1}$;
\item If $\alpha\neq\emptyset$ and $\beta\neq\emptyset$, then we have $x^{m+2}A/(1-x)^{m-1}$.
\end{itemize}
Thus, $G_m(x)=xG_{m-1}(x)+x^{m+1}A/(1-x)^{m-2}+x^{m+1}AB/(1-x)^{m-1}+x^{m+2}A/(1-x)^{m-1}$, for all $m\geq4$. Summing over $m\geq4$, we obtain
\begin{align*}
&F_T(x)-1-xF_T(x)-G_2(x)-G_3(x)\\
&=x\big(F_T(x)-1-xF_T(x)-G_2(x)\big)+\frac{x^5A}{(1-x)(1-2x)}+\frac{x^5AB}{(1-x)^2(1-2x)}+\frac{x^6A}{(1-x)^2(1-2x)}.
\end{align*}
The result follows by substituting the expressions for $G_2(x)$, $G_3(x)$, $A$ and $B$, and solving for $F_T(x)$.
\end{proof}

\subsection{Case 88: $\{3412,3421,1324\}$}
For this case we outline a proof based on a labelled generating forest and also give a proof based on left-right maxima.
\subsubsection{Labelled generating forest}
The table in the following lemma both recursively defines labels and gives valid succession rules. The $j$-th entry on the right hand side in the rules gives the label when $n+1$ is inserted into the $j$-th active site left to right. A label $k^i,\ 1\le i \le 5$, always indicates $k$ active sites. The proof is omitted.
\begin{lemma}\label{lem88a1}
The generating forest $\mathcal{F}$ is given by
$$\begin{array}{ll}
\mbox{\bf Roots: }&3^1,3^3\\
\mbox{\bf Rules: }&k^1\rightsquigarrow 3^12^53^5\cdots(k-2)^5(k+1)^1(k+1)^2,\quad k\geq3,\\ &k^2\rightsquigarrow3^32^53^5\cdots(k-3)^5(k-1)^4k^4(k+1)^3,\quad k\geq4,\\
&k^3\rightsquigarrow3^32^53^5\cdots(k-2)^5k^4(k+1)^3,\quad k\geq3,\\
&2^3\rightsquigarrow2^33^4,\\
&k^4\rightsquigarrow2^32^53^5\cdots(k-1)^5(k+1)^4,\quad k\geq3,\\
&k^5\rightsquigarrow1^52^53^5\cdots k^5,\quad k\geq1\,.\\
\end{array}$$ \qed
\end{lemma}

\begin{theorem}\label{th88a}
Let $T=\{3412,3421,1324\}$. Then
$$F_T(x)=\frac{(1-x)^2(1-5x+7x^2+x^3)}{(1-2x)^4}.$$
\end{theorem}
\begin{proof}
Let $a_k(x)$, $b_k(x)$, $c_k(x)$, $d_k(x)$ and $e_k(x)$ be the generating functions for the number of permutations in the $n$th level of the generating forest $\mathcal{F}$ with label $k^1$, $k^2$, $k^3$, $k^4$ and $k^5$, respectively.
By Lemma \ref{lem88a1}, we have
\begin{align*}
a_k(x)&=xa_{k-1}(x),\quad k\geq4,\\
b_k(x)&=xa_{k-1}(x),\quad k\geq4,\\
c_k(x)&=x(b_{k-1}+c_{k-1}(x)),\quad k\geq4,\\
d_k(x)&=x(d_{k-1}(x)+c_k(x)+b_k(x)+b_{k+1}(x)),\quad k\geq4,\\
e_k(x)&=xe_k(x)+x\sum_{j\geq k+1}(d_j(x)+e_j(x))+x\sum_{j\geq k+2}(a_j(x)+c_j(x))+x\sum_{j\geq k+3}b_j(x),\quad k\geq2
\end{align*}
with $a_3(x)=x^2+x\sum_{j\geq3}a_j(x)$, $c_2(x)=xc_2(x)+x\sum_{j\geq3}d_j(x)$, $c_3(x)=x^2+xc_3(x)+x\sum_{j\geq4}(b_j(x)+c_j(x))$, $d_3(x)=xc_2(x)+xc_3(x)+xb_4(x)$ and $e_1(x)=x\sum_{j\geq1}e_j(x)$.

Now let $A(x,v)=\sum_{k\geq3}a_k(x)v^k$, $B(x,v)=\sum_{k\geq3}b_k(x)v^k$, $C(x,v)=\sum_{k\geq3}c_k(x)v^k$,  $D(x,v)=\sum_{k\geq2}d_k(x)v^k$ and $E(x,v)=\sum_{k\geq2}e_k(x)v^k$.

By the above recurrences, we see that $A(x,v)-(x^2+xA(x,1))v^3=xvA(x,v)$. Thus, by substituting $v=1$, we obtain $A(x,1)=\frac{x^2}{1-2x}$, which implies $A(x,v)=\frac{x^2v^3(1-x)}{(1-xv)(1-2x)}$.
Thus, by the equation for $b_k(x)$, we have $B(x,v)=\frac{x^3v^4(1-x)}{(1-xv)(1-2x)}$.

By the equations for $c_k(x),d_k(x),e_k(x)$, we have
\begin{align*}
&C(x,v)-c_3(x)v^3-c_2(x)v^2=xv(B(x,v)+C(x,v)-c_2(x)v^2),\\
&D(x,v)-d_3(x)=x(vD(x,v)+C(x,v)-c_3(x)v^3-c_2(x)v^2+B(x,v)+(B(x,v)-b_4(x)v^4)/v),\\
&E(x,v)-e_1(x)v=\frac{x}{1-v}(v^2E(x,1)-vE(x,v)+v^2D(x,1)-D(x,v)+v^3C(x,1)-C(x,v))\\
&+xv^2c_2(x)+\frac{x}{1-v}(v^4B(x,1)-B(x,v)+v^3A(x,1)-A(x,v))
\end{align*}
with $c_2(x)=\frac{x}{1-x}D(x,1)$, $c_3(x)=x^2+xB(x,1)+xC(x,1)-\frac{x^2}{1-x}D(x,1)$, $d_3(x)=\frac{x^2}{(1-x)^2}D(x,1)$ and $e_1(x)=xE(x,1)v$. Note that $b_4(x)=\frac{x^3(1-x)}{1-2x}$ (from $B(x,v)$).

Substituting $v=1$ in the first two equations, and solving for $C(x,1)$ and $D(x,1)$, we obtain
$$C(x,1)=\frac{x^2(1-4x+7x^2-4x^3-2x^4)}{(1-2x)^3},\quad
D(x,1)=\frac{x^3(1-x)(1-2x^2)}{(1-2x)^3},$$
which implies
\begin{align*}
C(x,v)&=\frac{x^2v^2(x^2(1-2x^2)+(1-5x+4x^5+9x^2-8x^3)v)}{(1-xv)^2(1-2x)^3},\\
&+\frac{x^3(1-x)v^4((x^2+2x-1)(2x^2-2x+1)+x(1-2x)^2v)}{(1-xv)^2(1-2x)^3},\\
D(x,v)&=\frac{x^3v^3(1-4x+5x^2+2x^3-6x^4+x^2(8x^3-6x^2+2x-1)v+x^4(1-2x^2)v^2)}{(1-xv)^3(1-2x)^3}.
\end{align*}

We solve the equation for $E(x,v)$ by using the kernel method (see, e.g., \cite{HM} for an exposition), taking $v=\frac{1}{1-x}$. Using the expressions for $A(x,v)$, $B(x,v)$, $C(x,v)$ and $D(x,v)$, this gives
$$E(x,1)=\frac{x^4(x+1)(3-4x)}{(1-2x)^4}.$$
Since $F_T(x)=1+x+A(x,1)+B(x,1)+C(x,1)+D(x,1)+E(x,1)$, the result follows.
\end{proof}

\subsection{Case 93: $\{1324,2413,3421\}$}
In order to study this case, we need the following lemmas.
\begin{lemma}\label{lem93a1}
Let $T=\{1324,2413,3421\}$. Let $H_m(x)$ be the generating function for the number of $T$-avoiders of the form $\pi=(n+1-m)\pi^{(1)}(n+2-m)\pi^{(2)}\cdots n\pi^{(m)}$ such that $\pi(n+1)\in S_{n+1}(T)$. Then
$$H_m(x)=x^mK(x)+\frac{(m-1)x^{m+1}L(x)}{1-x},$$
where $K(x)=F_{\{132,3421\}}(x)=1+\frac{x(1-3x+3x^2)}{(1-x)(1-2x)^2}$ (see \cite[Seq. A005183]{Sl}) and $L(x)=\frac{1-x}{1-2x}$.
\end{lemma}
\begin{proof}
Let us write a formula for $H_m(x)$. Let $\pi=(n+1-m)\pi^{(1)}(n+2-m)\pi^{(2)}\cdots n\pi^{(m)}$ such that $\pi(n+1)\in S_{n+1}(T)$. Since $\pi$ avoids $1324$, we see that $\pi^{(1)}>\cdots>\pi^{(m)}$. If $\pi^{(2)}=\cdots=\pi^{(m)}=\emptyset$, then $\pi^{(1)}$ avoids $132$ and $3412$, which gives a contribution of $x^mK(x)$. Otherwise, since $\pi$ avoids $3421$, there exists a unique $j$ such that $\pi^{(j)}\neq\emptyset$ with $2\leq j\leq m$. Moreover, $\pi^{(1)}$ avoids $132,231$. Thus, we have a contribution of $\frac{x^{m+1}L(x)}{1-x}$. Hence, $H_m(x)=x^mK(x)+\frac{(m-1)x^{m+1}L(x)}{1-x}$, as required.
\end{proof}

\begin{lemma}\label{lem93a2}
Let $T=\{1324,2413,3421\}$. Let $H'_m(x)$ be the generating function for the number of permutations $\pi=(n+1-m)\pi^{(1)}(n+2-m)\pi^{(2)}\cdots n\pi^{(m)}\in S_{n}(T)$. Then
$$H'_m(x)=xH'_{m-1}(x)+\frac{x^{m+1}}{(1-x)(1-2x)},$$
with $H'_2(x)=\frac{x^2(1-4x+5x^2)}{(1-2x)^3}$.
\end{lemma}
\begin{proof}
We leave the  formula for $H'_2(x)$ to the reader. Now we write a formula for $H'_m(x)$ with $m\geq3$. Let $\pi=(n+1-m)\pi^{(1)}(n+2-m)\pi^{(2)}\cdots n\pi^{(m)}\in S_n(T)$. If $\pi^{(m-1)}=\emptyset$, then we have a contribution of $xH'_{m-1}(x)$. Thus, we can assume that $\pi^{(m-1)}\neq\emptyset$. Since $\pi$ avoids $1324$ and $2413$, we have that $\pi^{(2)}=\cdots = \pi^{(m-2)}=\emptyset$, $\pi^{(1)}>\pi^{(m)}>\pi^{(m-1)}$ and $\pi^{(m-1)}\pi^{(m)}$ is increasing. Note $\pi^{(1)}$ avoids $132$ and $231$. Thus, we have a contribution of $\frac{x^{m+1}}{(1-x)^2}L(x)$. Hence,
$H'_m(x)=xH'_{m-1}(x)+\frac{x^{m+1}}{(1-x)(1-2x)}$, as required.
\end{proof}

Similar consideration yield the following result.
\begin{lemma}\label{lem93a3}
Let $T=\{1324,2413,3421\}$. Let $H''_m(x)$ be the generating function for the number of permutations $\pi=(n+1-m)\pi^{(1)}(n+2-m)\pi^{(2)}\cdots n\pi^{(m)}\in S_{n}(T)$ such that $\pi^{(m)}$ has a letter smaller than $n+1-m$. Then
$$H''_m(x)=xH''_{m-1}(x)+\frac{x^{m+2}}{(1-x)(1-2x)},$$
with $H''_2(x)=\frac{x^3(1-4x+6x^2-2x^3)}{(1-x)(1-2x)^3}$.
\end{lemma}

Now we are ready to state the formula for $F_T(x)$.
\begin{theorem}\label{th93a}
Let $T=\{1324,2413,3421\}$. Then
$$F_T(x)=\frac{1-10x+42x^2-94x^3+120x^4-86x^5+31x^6-3x^7}{(1-x)^3(1-2x)^4}.$$
\end{theorem}
\begin{proof}
Let $G_m(x)$ be the generating function for $T$-avoiders with $m$ left-right maxima. Clearly, $G_0(x)=1$ and $G_1(x)=xF_T(x)$.

Now let us write an equation for $G_m(x)$ with $m\geq2$. Let $\pi=i_1\pi^{(1)}\cdots i_m\pi^{(m)}\in S_n(T)$ with exactly $m$ left-right maxima. Since $\pi$ avoids $1324$ and $2413$, we see that $\pi^{(s)}<i_1$ for all $s=1,2,\ldots,m-1$ and $\pi^{(m)}=\beta^{(m)}\cdots\beta^{(1)}$ such that $i_{s-1}<\beta^{(s)}<i_{s}$ for all $s=1,2,\ldots,m$ (with $i_0=0$). We consider  cases:
\begin{itemize}
\item $\beta^{(2)}=\cdots\beta^{(j-1)}=\emptyset$ and $\beta^{(j)}\neq\emptyset$ with $j=2,3,\ldots,m-1$. Since $\pi$ avoids $T$, we see that $\pi^{(s)}=\emptyset$ for all $s=j,j+1,\ldots,m-1$ and $\beta^{(s)}=\emptyset$ for all $s\neq j,m$, where $\beta^{(m)}$ avoids $132$ and $231$, and $\beta^{(j)}$ is nonempty increasing. Thus, by Lemma \ref{lem93a1}, we have a contribution of $\frac{x^{m+2-j}H_{j-1}(x)L(x)}{1-x}$.
\item $\beta^{(2)}=\cdots\beta^{(m-1)}=\emptyset$ and $\beta^{(m)}\neq\emptyset$. If $\beta^{(1)}=\emptyset$, then we have a contribution of $xH_{m-1}(x)(K(x)-1)$, see Lemma \ref{lem93a1}. Otherwise, we have a contribution of $H''_m(x)(L(x)-1)$, see Lemma \ref{lem93a3}.
\item $\beta^{(2)}=\cdots\beta^{(m)}=\emptyset$. For this case, we have a contribution of $H'_m(x)$, see Lemma \ref{lem93a2}.
\end{itemize}
By summing over all contributions, we obtain
$$G_m(x)=H'_m(x)+(xH_{m-1}(x)(K(x)-1)+H''_m(x)(L(x)-1)+\sum_{j=1}^{m-2}\frac{x^{m+1-j}H_j(x)L(x)}{1-x},$$
for all $m\geq2$. Hence, by summing over all $m\geq2$ and using the initial conditions $G_0(x)$ and $G_1(x)$, we have
$$F_T(x)-1-xF_T(x)=H'(x)+(L(x)-1)H''(x)+\left(\frac{x^3}{(1-x)^2}L(x)+x(K(x)-1)\right)H(x),$$
where, by Lemmas \ref{lem93a1}, \ref{lem93a2} and \ref{lem93a3},
\begin{align*}
H(x)&=\sum_{m\geq1}H_m(x)=\frac{xK(x)}{1-x}+\frac{x^3}{(1-x)^2(1-2x)}=\frac{x(1-3x+3x^2)}{(1-x)(1-2x)^2},\\
H'(x)&=\sum_{m\geq2}H'_m(x)=\frac{H'_2(x)}{1-x}+\frac{x^4}{(1-x)^3(1-2x)}=\frac{x^2(1-3x+3x^2)^2}{(1-x)^3(1-2x)^3},\\
H''(x)&=\sum_{m\geq2}H''_m(x)=\frac{H''_2(x)}{1-x}+\frac{x^5}{(1-x)^3(1-2x)}=\frac{x^3(1-3x+3x^2)(1-2x+2x^2)}{(1-x)^3(1-2x)^3}.
\end{align*}
Hence, by solving for $F_T(x)$, we complete the proof.
\end{proof}

\subsection{Case 99: $\{1324,3142,4231\}$}
Throughout this case, we abbreviate
$F_{\{132,4231\}}(x)$ by $K(x)$. Recall $K(x)=1+\frac{x(1-3x+3x^2)}{(1-x)(1-2x)^2}$ \cite[Seq. A005183]{Sl}. To find $F_T(x)$, we use the following lemmas.
\begin{lemma}\label{lem99a1}
Let $T=\{1324,3142,4231\}$. Define $J_m(x)$ to be the generating function for the number of permutations $\pi=i_1i_2\cdots i_m\pi'\in S_n(T)$ with exactly $m$ left-right maxima, all occurring at the start, such that $i_1>1$. Set $J(x)=\sum_{m\geq2}J_m(x)$. Then
$J_2(x)=\frac{x^3(1-2x+3x^2-x^3)}{(1-x)^2(1-2x)^2}$ and
\[
J(x)=\frac{x^3(1-2x+4x^2)}{(1-x)(1-2x)^3}\, .
\]
\end{lemma}
\begin{proof}
Let us write an equation for $J_m(x)$ with $m\geq3$. Let $\pi=i_1i_2\cdots i_m\pi'\in S_n(T)$ with exactly $m$ left-right maxima. If $i_{m-1}=n-1$, then we have a contribution of $xJ_{m-1}(x)$. Thus, we can assume that $i_{m-1}<n-1$. If $\pi'=(n-1)\pi''$, then we have a contribution of $xJ_m(x)$, otherwise, since $\pi$ avoids $T$ (note that $1$ belongs to $\pi'$), we can write $\pi$ as
$\pi'=i'(i'-1)\cdots(i_{m-1}+1)\alpha\beta$, such that $\alpha<i_1$, $i'<\beta<n$ and $n-1$ belongs to $\beta$. Note that $\beta$ avoids $213$ and $231$, and $\alpha$ avoids $132$ and $231$. Thus, we have a contribution of
$x^m(1+x)(L(x)-1)^2$.
Hence,
$$J_m(x)=xJ_{m-1}(x)+xJ_m(x)+x^m(1+x)(L(x)-1)^2.$$

By very similar techniques, we find that $J_2(x)=x(K(x)-1)+xJ_2(x)+x^2(1+x)(L(x)-1)^2$. Thus,
$J_2(x)=\frac{x^3(1-2x+3x^2-x^3)}{(1-x)^2(1-2x)^2}$ and the displayed recurrence for $J_m$ then readily yields the stated expression for $J(x)$.
\end{proof}

\begin{lemma}\label{lem99a2}
Let $T=\{1324,3142,4231\}$. Define $J'_m(x)$ to be the generating function for the number of permutations $\pi=i_1i_2\cdots i_m\pi'\in S_n(T)$ with exactly $m$ left-right maxima such that $i_1=1$ (and $i_m=n$). Then
$$J'(x):=\sum_{m\geq3}J'_m(x)=\frac{x^3(1-2x+2x^2)}{(1-x)(1-2x)^2}.$$
\end{lemma}
\begin{proof}
Let us write an equation for $J'_m(x)$ with $m\geq3$. Let $\pi=i_1i_2\cdots i_m\pi'\in S_n(T)$ with exactly $m$ left-right maxima such that $i_1=1$. By considering whether $i_2=2$ or not (we leave the details to the reader), we obtain the contributions
$xJ'_{m-1}(x)$ and $\frac{x^{m+1}L(x)}{(1-x)^{m-1}}$. Hence,
$$J'_m(x)=xJ'_{m-1}(x)+\frac{x^{m+1}L(x)}{(1-x)^{m-1}}$$
with $J'_2(x)=x^2L(x)$. The result follows by solving this recurrence.
\end{proof}

\begin{theorem}\label{th99a}
Let $T=\{1324,3142,4231\}$. Then
$$F_T(x)=\frac{1-8x+25x^2-36x^3+23x^4-4x^5+x^6}{(1-x)(1-2x)^4}.$$
\end{theorem}
\begin{proof}
Let $G_m(x)$ be the generating function for $T$-avoiders with $m$ left-right maxima. Clearly, $G_0(x)=1$ and $G_1(x)=xF_{\{132,4231\}}(x)=xK(x)$.

Now, let us write an equation for $G_m(x)$ with $m\geq2$. Let $\pi=i_1\pi^{(1)}\cdots i_m\pi^{(m)}\in S_n(T)$ with exactly $m$ left-right maxima. Since $\pi$ avoids $1324$, we see that $\pi^{(s)}<i_1$ for all $s=2,3,\ldots,m-1$. We consider the following cases:
\begin{itemize}
\item $\pi^{(m)}$ has a letter smaller than $i_1$. Since $\pi$ avoids $4231$, we see that $\pi^{(1)}\cdots\pi^{(m-1)}$ is decreasing. Assume that $\pi^{(j)}\neq\emptyset$ and $\pi^{(j+1)}=\cdots\pi^{(m-1)}=\emptyset$. Since $\pi$ avoids $3142$, we have that
    $\pi^{(m)}$ has no letters between $i_1$ and $i_j$. Thus, we have a contribution of
    $\frac{1}{1-x}J_m(x)$ when $j=1$, and $\frac{x^j}{(1-x)^j}J_{m+1-j}(x)$ when $j=2,3,\ldots,m-1$, where $J_d(x)$ is defined in Lemma \ref{lem99a1}. Hence,
    the total contribution for this case is given by
    $$\frac{1}{1-x}J_m(x)+\sum_{j=2}^{m-1}\frac{x^j}{(1-x)^j}J_{m+1-j}(x).$$

\item $\pi^{(m)}>i_1$ and $\pi^{(2)}=\cdots\pi^{(m-1)}=\emptyset$. Then we have a contribution of $L(x)J'_m(x)$, where $J'_m(x)$ is given in Lemma \ref{lem99a2}.

\item  $\pi^{(m)}>i_1$ and there exists $j$ such that $\pi^{(j)}\neq\emptyset$ and $\pi^{(j+1)}=\cdots\pi^{(m-1)}=\emptyset$, where $j=2,3,\ldots,m-1$. As before, we obtain the contributions $\frac{x^{j-1}(L(x)-1)}{(1-x)^{j-1}}J'_{m+1-j}(x)$ when $j\leq m-2$, and $\frac{x^mL(x)(L(x)-1)}{(1-x)^{m-2}}$ when $j=m-1$. Hence, the total contribution from this case is $$\frac{x^mL(x)(L(x)-1)}{(1-x)^{m-2}}+\sum_{j=1}^{m-3}\frac{x^{j}(L(x)-1)}{(1-x)^{j}}J'_{m-j}(x).$$
\end{itemize}
Hence,
\begin{align*}
G_m(x)&=L(x)J'_m(x)+\frac{x^mL(x)(L(x)-1)}{(1-x)^{m-2}}+\sum_{j=1}^{m-3}\frac{x^j(L(x)-1)}{(1-x)^j}J'_{m-j}(x)\\
&+\frac{1}{1-x}J_m(x)+\sum_{j=2}^{m-1}\frac{x^j}{(1-x)^j}J_{m+1-j}(x).
\end{align*}
Summing over $m\geq2$ and using the expressions for $G_0(x)$ and $G_1(x)$, we obtain
\begin{align*}
F_T(x)-1-xK(x)&=\sum_{m\geq2}G_m(x)\\
&=x^2L(x)^2+\frac{1}{1-x}J_2(x)+L(x)J'(x)+\frac{x^3L(x)(L(x)-1)}{1-2x}\\
&+\frac{x(L(x)-1)J'(x)}{1-2x}+\frac{1}{1-x}(J(x)-J_2(x))+\frac{x^2}{(1-x)(1-2x)}J(x),
\end{align*}
where $J(x)=\sum_{d\geq2}J_d(x)$, $J_2(x)$ and $J'(x)=\sum_{d\geq3}J'_d(x)$ are given in Lemmas \ref{lem99a1} and \ref{lem99a2}. After several algebraic operations, we complete the proof.
\end{proof}

\subsection{Case 132: $\{1324,2341,2413\}$}
\begin{theorem}\label{th132a}
Let $T=\{1324,2341,2413\}$. Then
$$F_T(x)=\frac{1-8x+23x^2-27x^3+12x^4-4x^5+x^6}{(1-3x+x^2)^3}.$$
\end{theorem}
\begin{proof}
Let $G_m(x)$ be the generating function for $T$-avoiders with $m$
left-right maxima. Clearly, $G_0(x)=1$ and $G_1(x)=xF_T(x)$.

Let us write an equation for $G_2(x)$. Let $\pi=i\pi'n\pi''\in S_n(T)$ with exactly $2$ left-right maxima. The contribution of the case $\pi''>i$ is $x^2K(x)^2$, where
$K(x)=\frac{1-2x}{1-3x+x^2}$ is the generating function for $\{132,2341\}$-
(or $\{213,2341\}$-) avoiders (for example, see \cite{MV}). If there is a letter in $\pi''$ smaller than $i$,
we can write $\pi$ as $\pi=i\pi'n(n-1) \cdots(i+1)\pi''$ such that $\pi''$ is not empty.
Thus, we have a contribution of $\frac{x}{1-x}\big(F_T(x)-1-xK(x)\big)$,
where $\frac{x^2}{1-x}K(x)$ counts the permutations of the form $i\pi'n(n-1)\cdots(i+1)$.
Hence, $G_2(x)=x^2K(x)^2+\frac{x}{1-x}\big(F_T(x)-1-xK(x)\big)$.

Next, let us write an equation for $G_3(x)$. Let $\pi=i_1\pi'i_2\pi''i_3\pi'''\in S_n(T)$ with exactly $3$ left-right maxima. Since $\pi$ avoids $1324$ and $2413$, we can write $\pi=i_1\pi'i_2\pi''i_3\alpha\beta$ such that
$\pi''<\pi'<i_1<\beta<i_2<\alpha<i_3$. By considering the four cases where $\pi'',\beta$
are empty or not (if both are nonempty we get an occurrence of $2413$, so we
actually have three cases),
we obtain $$G_3(x)=x^3K(x)^2+\frac{x^3}{1-x}K(x)\big(K(x)-1\big)+
\frac{x^3}{1-x}K(x)\big(K(x)-1\big).$$

Finally, let us write an equation for $G_m(x)$ with $m\geq4$. Let $\pi=i_1\pi^{(1)}\cdots i_m\pi^{(m)}\in S_n(T)$ with exactly $m$ left-right maxima. Since $\pi$ avoids $T$, we see that $i_1>\pi^{(1)}>\pi^{(2)}$, $\pi^{(j)}=\emptyset$ for $j=3,4,\ldots,m-1$, and $\pi^{(m)}=\alpha\beta$ such that $i_m>\alpha>i_{m-1}>\beta$. Again, by considering the four cases where
$\pi^{(2)},\beta$ are empty or not, we obtain that
$$G_m(x)=x^mK(x)^2+\frac{x^m}{1-x}K(x)\big(K(x)-1\big)+\frac{x^m}{1-x}K(x)\big(K(x)-1\big)+\frac{x^m}{(1-x)^2}(K(x)-1)^2.$$

By summing for $m\geq4$ and using the expressions for $G_0(x),G_1(x),G_2(x)$ and $G_3(x)$, we obtain
\begin{align*}
F_T(x)&=1+xF_T(x)+\frac{x}{1-x}\big(F_T(x)-1-xK(x)\big)+x^2K(x)^2 \\
&+x^3K(x)^2+\frac{2x^3}{1-x}K(x)\big(K(x)-1\big) +\frac{x^4}{1-x}\left(K(x)+\frac{K(x)-1}{1-x}\right)^2,
\end{align*}
which, by solving for $F_T(x)$, completes the proof.
\end{proof}

\subsection{Case 150: $\{1324,3421,3241\}$}
In this section, define $L=F_{\{213,231\}}(x)=\frac{1-x}{1-2x}$ (see \cite{SiS}). We also have, by the simple decomposition $\pi'n\pi''$ of $132$-avoiders,
$$A:=F_{\{132,3421,3241\}}(x)=\frac{1-3x+3x^2}{(1-x)^2(1-2x)}$$ and
$$B:=F_{\{213,3421\}}(x)=F_{\{132,4312\}}(x)=\frac{1-4x+5x^2-x^3}{(1-x)(1-2x)^2}$$.

\begin{lemma}\label{lem150a1}
The generating function for $T$-avoiders of the form
\begin{itemize}
\item $(d+1)n\pi'(d+2)$ with $n-3\geq d\geq1$ is given by
$$E_d(x)=\frac{x^{d+3}+\sum_{j=2}^{d+1}\frac{x^{d+4}L}{(1-x)^j}}{1-x/(1-x)}.$$
\item $(d+1)n\pi'$ with $n-2\geq d\geq1$ is given by
$$D_d(x)=\frac{x^{d+2}+\sum_{j=2}^d\frac{x^{d+3}L}{(1-x)^j}+\frac{E_d(x)}{1-x}+\frac{x^{d+3}}{(1-x)^{d+1}}(B-1/(1-x))}{1-x/(1-x)}.$$
\end{itemize}
Moreover,
$$E(x):=\sum_{d\geq1}E_d(x)=\frac{(1-3x+4x^2)x^4}{(1-2x)^3}\mbox{ and }D(x):=\sum_{d\geq1}D_d(x)=\frac{(1-5x+10x^2-5x^3)x^3}{(1-2x)^4}.$$
\end{lemma}
\begin{proof}
Let us write an equation for $E_d(x)$. Let $\pi=(d+1)n\pi'(d+2)\in S_n(T)$. If $n=d+3$, we have a contribution of $x^{d+3}$, so let us assume that $n\geq d+4$ and consider 3 cases:
\begin{itemize}
\item If $d+3$ on left side of letter $1$, $\pi$ can be written as $(d+1)n(n-1)\cdots n'(d+3)\pi'(d+2)$ with $\pi'<n'$, so we have $\frac{x}{1-x}E_d(x)$.
\item If $d+3$ between the letters $j$ and $j+1$, $j=1,2,\ldots,d-1$, then $\pi$ can be written as
$$(d+1)n\alpha^{(1)}1\alpha^{(2)}2\cdots\alpha^{(j)}j\alpha^{(j+1)}(d+3)\pi'(j+1)(j+2)\cdots d(d+2)$$
such that the subsequence $n\alpha^{(1)}\cdots\alpha^{(j+1)}$ is decreasing and greater than $\pi'$ and $\pi'>d+3$, where $\pi'$ avoids $\{231,213\}$. So, we have $\frac{x^{d+4}}{(1-x)^j}L(x)$.
\item If $d+3$ between the letters $d$ and $d+2$, then $\pi$ can be written as
$$(d+1)n\alpha^{(1)}1\alpha^{(2)}2\cdots\alpha^{(d)}d\alpha^{(d+1)}(d+3)\pi'(d+2)$$
such that the subsequence $n\alpha^{(1)}\cdots\alpha^{(d+1)}$ is decreasing and greater than $\pi'$ and $\pi'>d+3$, where $\pi'$ avoids $\{231,213\}$. So, we have $\frac{x^{d+4}}{(1-x)^{d+1}}L(x)$.
\end{itemize}
Thus, $E_d(x)=x^{d+3}+\frac{x}{1-x}E_d(x)+\sum_{j=2}^{d+1}\frac{x^{d+4}L}{(1-x)^j}$, as claimed.

Similarly, we can write an equation for $D_d(x)$ and obtain
$$D_d(x)=x^{d+2}+\frac{x}{1-x}D_d(x)+\sum_{j=2}^d\frac{x^{d+3}L}{(1-x)^j}+\frac{E_d(x)}{1-x}+\frac{x^{d+3}}{(1-x)^{d+1}}(B-1/(1-x)),$$
where $x^{d+2}$, $\frac{x}{1-x}D_d(x)$, $\frac{x^{d+3}L}{(1-x)^{j+1}}$ and $\frac{E_d(x)}{1-x}+\frac{x^{d+3}}{(1-x)^{d+1}}(B-1/(1-x))$ are the respective contributions for the cases $n=d+2$, the letter $d+2$ on the left side of the letter $1$, the letter $d+2$ between the letters $j$ and $j+1$ with $j=1,2,\ldots,d-1$, and the letter $d+2$ on the right side of the letter $d$.

By summing over $d\geq1$, we complete the proof.
\end{proof}

\begin{lemma}\label{lem150a2}
For $m\ge 3$, the generating function for $T$-avoiders $\pi=i_1\pi^{(1)}\cdots i_m\pi^{(m)}$ with $m$ left-right maxima and $\pi^{(1)}\neq\emptyset$ is given by
$$G_{m,1}(x)=\frac{x^{m+1}}{(1-x)^2}(A-1)L+x^m(A-1)B+(m-2)\frac{x^{m+1}}{1-x}(A-1)L.$$
\end{lemma}
\begin{proof}
Let $\pi=i_1\pi^{(1)}\cdots i_m\pi^{(m)}\in S_n(T)$ with $m$ left-right maxima and $\pi^{(1)}\neq\emptyset$. By the definitions we see that $\pi^{(2)}\cdots\pi^{(m-1)}=\emptyset$.
We have the following contributions to $G_{m,1}(x)$:
\begin{itemize}
\item If $\pi^{(m)}>i_{m-1}$, then we have $x^m(A-1)B$.
\item If $\pi^{(m)}$ has a letter smaller than $i_1$, then $\pi^{(1)}<\pi^{(m)}=\pi'(k+1)(k+2)\cdots(i_1-1)(i_1+1)\cdots(i_2-1)$, $\pi'>i_{m-1}$, $\pi'$ avoids $\{213,231\}$ and $\pi^{(1)}$ avoids $\{132,3421,3241\}$. This leads to a contribution of $\frac{x^{m+1}}{(1-x)^2}(A-1)L$.
\item If $\pi^{(m)}>i_1$ and $\pi^{(m)}$ has a letter between $i_j$ and $i_{j+1}$, where $j=1,2,\ldots,m-2$, then $\pi^{(1)}<i_1<\pi^{(m)}$, where $\pi^{(m)}=\pi'(i_j+1)\cdots(i_{j+1}-1)$ with $\pi'$ a $\{213,231\}$-avoider, and $\pi^{(1)}$ avoids $\{132,3421,3241\}$. This leads to a contribution of $\frac{x^{m+1}}{1-x}(A-1)L$.
\end{itemize}
Adding these contributions yields the result.
\end{proof}

Using similar arguments, we obtain the following result.

\begin{lemma}\label{lem150a3}
The generating function for $T$-avoiders $\pi=i_1\pi^{(1)}\cdots i_m\pi^{(m)}$ with $m$ left-right maxima and $\pi^{(1)}=\cdots=\pi^{(s-1)}=\emptyset$ and $\pi^{(s)}\neq\emptyset$, $2\leq s\leq m-2$, is given by
$$G_{m,s}(x)=\frac{x^{m+2}L}{(1-x)^2}(1+sx/(1-x))+\frac{x^{m+1}}{1-x}B+(m-2)\frac{x^{m+2}L}{(1-x)^2}.$$
\end{lemma}

\begin{lemma}\label{lem150a4}
Let $T=\{1324,3421,3241\}$. Then the generating function for $T$-avoiders $\pi=i_1\pi^{(1)}\cdots i_m\pi^{(m)}$ with $m$ left-right maxima and $\pi^{(1)}=\cdots=\pi^{(m-2)}=\emptyset$
and $\pi^{(m-1)}\neq\emptyset$ is given by $$G_{m,m-1}(x)=N_m+N'_m+M_m+M'_m,$$
where
\begin{align*}
N_m&=x^{m+1}B/(1-x),\\
N'_m&=(m-2)x^{m+2}L/(1-x)^2,\\
M_m&=\frac{xN_m/(1-x)+x^{m+2}(L-1/(1-x))/(1-x)^3}{1-x/(1-x)} ,\\ M'_m&=\frac{xN_m'+(m-2)x^{m+3}(L-1)/((1-x)^2(1-2x))}{1-x}.
\end{align*}
\end{lemma}
\begin{proof}
To write an equation for $G_{m,m-1}(x)$, suppose $\pi=i_1\pi^{(1)}\cdots i_m\pi^{(m)}$ with $m$ left-right maxima and $\pi^{(1)}=\cdots=\pi^{(m-2)}=\emptyset$ and $\pi^{(m-1)}\neq\emptyset$. Let $\alpha$ be the subsequence of $\pi^{(m)}$ consisting of letters smaller than $\pi^{(1)}$, and let $\beta$ be the subsequence of $\pi^{(m)}$ consisting of letters between $i_1$ and $i_{m-1}$. Note that $\pi^{(m)}<\alpha\beta$. Now let us consider the following four cases:
\begin{itemize}
\item $\alpha=\beta=\emptyset$: Here $\pi^{(m-1)}<i_1$ and $\pi^{(m)}>i_{m-1}$. So we have a contribution $N_m$.
\item $\alpha=\emptyset$ and $\beta\neq\emptyset$: Here $\pi^{(m-1)}<i_1$, and there exits $1\leq j\leq m-2$ such that $\pi^{(m)}=\pi'(i_j+1)\cdots(i_{j+1}-1)$ and $\pi'$ avoids $\{213,231\}$. So we have a contribution $N'_m$.
\item $\alpha\neq\emptyset$ and $\beta=\emptyset$. In this case $\pi^{(m)}$ can be written as
$$\beta^{(s)}(i_1-s)\beta^{(s-1)}(i_1-s+1)\cdots \beta^{(1)}(i_1-1)\beta^{(0)},$$
where $\beta^{(s)}>\beta^{(s-1)}>\cdots>\beta^{(0)}>i_{m-1}$ and $\pi^{(m-1)}<i_1-s$. Let $M_m$ be the contribution of this case. By considering whether $\beta^{(s)}$ is a decreasing sequence (including the empty case) or contains an ascent, we obtain the contributions $\frac{x}{1-x}(N_m+M_m)$ and $x^{m+2}(L-1/(1-x))/(1-x)^3$. Thus,
$$M_m=\frac{x}{1-x}(N_m+M_m)+x^{m+2}(L-1/(1-x))/(1-x)^3,$$
with solution for $M_m$ as above.
\item $\alpha\neq\emptyset$ and $\beta\neq\emptyset$. In this case, there exists $j$, $1\leq j\leq m-2$ such that $\pi^{(m)}$ can be written as
$$\beta^{(s)}(i_1-s)\beta^{(s-1)}(i_1-s+1)\cdots\beta^{(1)}(i_1-1)\beta^{(0)}(i_j+1)(i_j+2)\cdots(i_{j+1}-1),$$
where $\beta^{(s)}>\beta^{(s-1)}>\cdots>\beta^{(0)}>i_{m-1}$ and $\pi^{(m-1)}<i_1-s$. Let $M'_m$ be the contribution of this case. By considering whether $\beta^{(0)}$ is empty or not, we obtain the contributions $x(N'_m+M'_m)$ and $(m-2)x^{m+3}(L-1)/((1-x)^2(1-2x))$. Thus,
$$M'_m=x(N'_m+M'_m)+(m-2)x^{m+3}(L-1)/((1-x)^2(1-2x)),$$
with solution for $M'_m$ as above.
\end{itemize}
This completes the proof.
\end{proof}

\begin{theorem}\label{th150a}
Let $T=\{1324,3421,3241\}$. Then
$$F_T(x)=\frac{1-11x+52x^2-136x^3+214x^4-204x^5+111x^6-28x^7}{(1-x)^3(1-2x)^3(1-3x+2x^2)}.$$
\end{theorem}
\begin{proof}
First, we study the generating function $G'_m(x)$  for $T$-avoiders $\pi$ with $m$
left-right maxima such that $\pi_1=1$. Clearly, $G'_m(x)/x$ is the generating function
for the $\{213,3421,3241\}$-avoiders with  $m-1$ left-right maxima.
It is not hard to see that
$$G'_m(x)=(m-2)\,\frac{x^{m+1}}{1-x}\,L+x^mB.$$

Now let $G_m(x)$ be the generating function for $T$-avoiders with $m$
left-right maxima. Clearly, $G_0(x)=1$ and $G_1(x)=xF_T(x)$.

Let us find the generating function $G_2(x)$. Suppose $\pi=i\pi'n\pi''\in S_n(T)$ with 2 left-right maxima. If $\pi''>i$, then we have a contribution of $x^2(A-1)B$. Otherwise, we can write $\pi$ as
$$\pi=i\pi'n\beta^{(s)}(i-s)\beta^{(s-1)}(i-s+1)\cdots \beta^{(1)}(i-1)\beta^{(0)},$$
where $\beta^{(s)}>\beta^{(s-1)}>\cdots>\beta^{(0)}>i$ and $\pi'<i-s$. Let us denote the contribution of this case by $K$. By considering whether $\beta^{(s)}$ is decreasing or contains a rise, we obtain the contributions of $\frac{x}{1-x}(K+x^2(A-1)B)$ and $\frac{x^3}{(1-x)^2}(A-1)(L-\frac{1}{1-x})$. Thus,
$$K=\frac{x}{1-x}(K+x^2(A-1)B)+\frac{x^3}{(1-x)^2}(A-1)\left(L-\frac{1}{1-x}\right),$$
and $G_2(x)-G'_2(x)=x^2(A-1)B+K$.

Now let us find an explicit formula for $G_m(x)$. By Lemmas \ref{lem150a1}--\ref{lem150a4}, we see that
$$G_m(x)-G'_m(x)=G_{m,1}(x)+\sum_{j=2}^{m-2}G_{m,j}(x)+G_{m,m-1}(x)+x^{m-2}(D_m(x)+(m-2)E_m(x)/(1-x)),$$
where $G_{m,m}(x)$ counted by $x^{m-2}(D_m(x)+(m-2)E_m(x)/(1-x))$. Thus,
$$F_T(x)=1+xF_T(x)+\sum_{m\geq2}(G_m(x)-G'_m(x))+\sum_{m\geq2}G'_m(x).$$
By substituting the expressions for $G_m(x)-G'_m(x)$ and solving for $F_T(x)$, we complete the proof.
\end{proof}

\subsection{Case 151: $\{1324,1342,3421\}$}
Here, we focus on the number of left-right maxima and begin with the case where $n$ is the second letter. Let $J_d(x)$ be the generating for permutations of the form $(d+1)n\alpha\in S_n(T)$.

\begin{lemma}\label{lem151a1}
$J_0(x)=x^2L(x)=\frac{x^2(1-x)}{1-2x}$ and for all $d\geq1$,
$$J_d(x)=\frac{x^{d+2}+x^2\left(\frac{x}{1-x}\right)^{d+1}L(x)+\sum_{j=0}^{d-2}x^{3+j}\left(\frac{x}{1-x}\right)^{d-j}}{1-\frac{x}{1-x}}\,.$$
\end{lemma}
\begin{proof}
Clearly, $J_0(x)=x^2F_{\{213,231\}}(x)=x^2L(x)$. Thus, we assume that $d\geq1$ and let us write an equation for $J_d(x)$.
If $d=n-2$ then we have a contribution of $x^{d+2}$. Otherwise, $d\leq n-3$, and consider the position of the letter $d+2$. Note that the letters $1,2,\dots,d$ occur in that order since $\pi$ starts $(d+1)n$ and avoids 3421.
\begin{itemize}
\item The letter $d+2$ appears on the left side of the letter $1$. In this case,  $(d+1)n\alpha=(d+1)n(n-1)\cdots(n'+1)(d+2)\alpha'$, which leads to a contribution of $\frac{x}{1-x}J_d(x)$.
\item The letter $d+2$ appears between the letters $p$ and $p+1$, where $1\leq p\leq d-1$. Then, $(d+1)n\alpha=(d+1)n\alpha'(d+2)(p+1)(p+2)\cdots d$, where all the letters in $\alpha'$ which are $>d+2$ (resp. $<p+1$) are decreasing (resp. increasing). Thus, we have a contribution of
    $$x^{2+d-p}\left(\frac{x}{1-x}\right)^{p+1}\,.$$
\item The letter $d+2$ appear on the right side of $d$. In this case, $\pi$ has the form $(d+1)n\alpha'(d+2)\beta$ where all the letters in $\alpha'$ greater than $d+2$ form a decreasing subsequence $\gamma$ with $\gamma>\beta>d+2$, and $\beta$ avoids $\{213,231\}$. Thus, we have a contribution of
    $$x^2\left(\frac{x}{1-x}\right)^{d+1}L(x)\,.$$
\end{itemize}
Hence, by adding all the contributions, we have
$$J_d(x)=x^{d+2}+\frac{x}{1-x}J_d(x)+x^2\left(\frac{x}{1-x}\right)^{d+1}L(x)+\sum_{p=1}^{d-1}x^{2+d-p}\left(\frac{x}{1-x}\right)^{p+1},$$
and the stated expression for $J_d(x)$ follows.
\end{proof}
Now set $J(x)=\sum_{d\geq0}J_d(x)$, the \gf for $T$-avoiders whose largest letter occurs in second position.
\begin{corollary}\label{cor151a1}
$$J(x)=\frac{(1-5x+9x^2-5x^3-x^4)x^2}{(1-2x)^3(1-x)}\,.$$ \qed
\end{corollary}

\begin{lemma}\label{lem151a3}
Let $G_2(x)$ be the generating function for permutations in $S_n(T)$ with exactly two left-right maxima. Then
$$G_2(x)=\frac{1-8x+26x^2-40x^3+25x^4-2x^5-x^6)x^2}{(1-3x+x^2)(1-2x)^4}.$$
\end{lemma}
\begin{proof}
Define $G_2(x;d)$ to be the generating function for permutations in $i\pi'n\pi''\in S_n(T)$ with  two left-right maxima such that $\pi''$ contains $d$ letters smaller than $i$. If $d=0$ so that $\pi$ has the form $i\al n\be $ with $\al<i<\be$,
then $\al$ avoids $\{132,3421\}$ and $\be$ avoids $\{213,231\}$. Hence
$G_2(x;0)=x^2A(x)L(x)$,
where $A(x)=1+\frac{x(1-3x+3x^2)}{(1-x)^2(1-2x)^2}$ is the generating function for $S_n(\{132,3421\})$ \cite{MV}.
From now on, we assume that $d\geq1$. The letters $j_1,\dots,j_d$ in $\pi''$ that are
less than $i$ are increasing, and so $\pi$ can be decomposed as
$$\pi=i\alpha^{(1)}\cdots\alpha^{(d+1)}n\pi'',$$
where $i>\alpha^{(1)}>j_d>\alpha^{(2)}>\cdots>\alpha^{(d)}>j_1>\alpha^{(d+1)}$.
Now, we treat the following cases:
\begin{itemize}
\item $\alpha^{(j)}=\emptyset$ for all $j=2,3,\ldots,d+1$. In this case, we have a contribution of $L(x)J_d(x)$, see Lemma \ref{lem151a1}.
\item $\alpha^{(s)}\neq\emptyset$ and $\alpha^{(j)}=\emptyset$ for all $j=s+1,s+2,\ldots,d+1$, where $s=2,3,\ldots,d+1$. In this case, $\alpha^{(j)}$ is a decreasing sequence for $j=1,2,\ldots,s-1$, $\alpha^{(s)}$ avoids $\{132,231\}$, and $\pi''=\beta^{(1)}j_1\cdots\beta^{(d)}j_d\gamma$, where $\beta=\beta^{(1)}\cdots\beta^{(d)}$ is a decreasing subsequence such that $\beta>\gamma$ and $\gamma$ avoids $\{213,231\}$. So, we have a contribution of $$x^2\left(\frac{x}{1-x}\right)^dL(x)\frac{1}{(1-x)^{s-1}}\big(L(x)-1\big)$$
    for $s=2,3,\ldots,d$, and
    $$x^2\left(\frac{x}{1-x}\right)^dL(x)\frac{1}{(1-x)^d}\big(A(x)-1\big)$$
    for $s=d+1$.
\item $\alpha^{(d+1)}\neq\emptyset$ and $\alpha^{(s)}$ contains a rise (a rise of $\pi$ is an index $i$ such that $\pi_i<\pi_{i+1}$) with $s=1,2,\ldots,d$. In this case, $\alpha^{(1)}\cdots\alpha^{(s-1)}$ is a decreasing sequence, $\alpha^{(s+1)}\cdots\alpha^{(d)}$ is empty, $\alpha^{(d)}$ is increasing and $\pi''$ decomposes as in the previous bullet. So the contribution is
    $$x^2\left(\frac{x}{1-x}\right)^dL(x)\left(L(x)-\frac{1}{1-x}\right)\frac{1}{(1-x)^{s-1}}\frac{x}{1-x}.$$
\end{itemize}
Hence, for $d\geq1$,
\begin{align*}
G_2(x;d)&=L(x)J_d(x)+\sum_{j=1}^{d-1}\frac{x^{d+2}L(x)\big(L(x)-1\big)}{(1-x)^{j+d}}\\
&+\frac{x^{d+3}L(x)\big(A(x)-1\big)}{(1-x)^{2d}}
+\sum_{j=0}^{d-1}\frac{x^{d+3}\big(L(x)-1/(1-x)\big)L(x)}{(1-x)^{d+1+j}}.
\end{align*}
Summing over $d\geq0$ and using Lemma \ref{lem151a1}, we obtain the stated formula for $G_2(x)$.
\end{proof}
\begin{lemma}\label{lem151a2}
Let $G_3(x)$ be the generating for the number of permutations in $S_n(T)$ with exactly three left-right maxima. Then
$$G_3(x)=\frac{(1-4x+5x^2+x^3-5x^4)x^3}{(1-2x)^4}.$$
\end{lemma}
\begin{proof}
We consider permutations $\pi=i_1\pi^{(1)}i_2\pi^{(2)}i_3\pi^{(3)}\in S_n(T)$ with exactly three left-right maxima.
Since $\pi$ avoids $1324$ and $1342$, we see that $\pi^{(1)}>\pi^{(2)}\beta$ where $\beta$ is all the letters in $\pi^{(3)}$ smaller than $i_1$. If $\pi^{(2)}=\emptyset$ and $\beta=\emptyset$ then we have a contribution of $x^3A(x)L(x)$. If $\pi^{(2)}=\emptyset$ and $\beta\neq\emptyset$, then we have a contribution of $x^3L(x)\big(J(x)-x^2L(x)\big)$. If $\pi^{(2)}\neq\emptyset$, we see that  $\pi^{(1)}>\pi^{(2)}\beta$, $\pi^{(1)}$ avoids $\{132,231\}$, $\pi^{(2)}\beta$ is an increasing sequence. Suppose $\beta=j_1j_2\cdots j_d$, then $\pi^{(3)}=\beta^{(1)}j_1\cdots\beta^{(d)}j_d\beta^{(d+1)}$, where $\beta^{(1)}\cdots\beta^{(d)}$ is a decreasing sequence
which is greater than $\beta^{(d+1)}$, and $\beta^{(d+1)}$ avoids $\{213,231\}$. Thus, we have a contribution of $x^3L(x)\frac{x}{1-x}L(x)\frac{1}{1-x/(1-x)}$. Hence,
$$G_3(x)=x^3A(x)L(x)+x^3L(x)(J(x)-x^2L(x))+x^3L(x)\frac{x}{1-x}L(x)\frac{1}{1-x/(1-x)},$$
and, using Corollary \ref{cor151a1}, the result follows.
\end{proof}

\begin{theorem}\label{th151a}
Let $T=\{1324,1342,3421\}$. Then
$$F_T(x)=\frac{1-12x+61x^2-169x^3+275x^4-263x^5+136x^6-29x^7+x^8}{(1-3x+x^2)(1-2x)^4(1-x)^2}.$$
\end{theorem}
\begin{proof}
Let $G_m(x)$ be the generating function for $T$-avoiders with $m$
left-right maxima. Clearly, $G_0(x)=1$ and $G_1(x)=xF_T(x)$, and $G_2(x)$ and $G_3(x)$ are given by Lemmas \ref{lem151a3} and \ref{lem151a2}, respectively. Now let $m\geq4$ and let us write an equation for $G_m(x)$.
Suppose $\pi=i_1\pi^{(1)}i_2\pi^{(2)}\cdots i_m\pi^{(m)}\in S_n(T)$ with exactly $m$ left-right maxima. If $\pi^{(2)}$ is empty then we have a contribution of $xG_{m-1}(x)$.
Otherwise, $\pi^{(j)}=\emptyset$ for all $j=3,4,\ldots,m-1$, $\pi^{(m)}>i_{m-1}$ and
$\pi^{(1)}>\pi^{(2)}$, where $\pi^{(m)}$ avoids $\{213,231\}$, $\pi^{(1)}$ avoids
$\{132,231\}$, and $\pi^{(2)}$ is an increasing sequence. Recall that
$F_{\{213,231\}}(x)=F_{\{132,231\}}(x)=\frac{1-x}{1-2x}$  \cite{SiS}. Thus, we have
$$G_m(x)=xG_{m-1}(x)+\frac{x^{m+1}}{1-x}\frac{1-x}{1-2x},$$
for all $m\geq4$. Summing over $m\geq4$, we obtain
$$F_T(x)-1-G_1(x)-G_2(x)-G_3(x)=x(F_T(x)-1-G_1(x)-G_2(x))+\frac{x^5}{(1-x)(1-2x)}\,.$$
Solving for $F_T(x)$ using the expressions above for $G_1(x),G_2(x),G_3(x)$, we complete the proof.
\end{proof}

\subsection{Case 153: $\{4231,1324,1342\}$}
Here, we use the same techniques as in Case 84.

\begin{lemma}\label{lem153a1}
Let $m\geq3$. The generating function for permutations $\pi=\pi_1\pi_2\cdots\pi_n\in S_n(T)$ with exactly $m$ left-right maxima and $\pi_1=1$ is given by $\frac{x^m(1-x)}{1-2x}$.
\end{lemma}
\begin{proof}
Such permutations can be written as $12\cdots(m-1)n\pi'$ where $\pi'$ avoids $213$ and $231$, and the result follows.
\end{proof}

\begin{lemma}\label{lem153a2}
Let $m\geq3$. The generating function $G_m(x)$ for permutations $\pi=\pi_1\pi_2\cdots\pi_n\in S_n(T)$ with exactly $m$ left-right maxima satisfies the recurrence
$$G_m(x)=x^mL^2(x)+\frac{x}{1-x}\big(G_{m-1}(x)-x^{m-1}L(x)\big)\, .$$
\end{lemma}
\begin{proof}
To write an equation for $G_m(x)$, let $\pi=i_1\pi^{(1)}\cdots i_m\pi^{(m)}\in S_n(T)$ with $m$ left-right maxima. Since $\pi$ avoids $1324$ and $1342$ we see that $\pi^{(2)}\cdots\pi^{(m)}$ has no letter between $i_1$ and $i_{m-1}$. If $\pi^{(2)}\cdots\pi^{(m)}$ has no letter
smaller than $\pi^{(1)}$, the contribution is $x^mF_{\{132,231\}}(x)F_{\{213,231\}}(x)=x^mL^2(x)$. Otherwise, $\pi^{(1)}$ is a decreasing sequence, and then we have a contribution of $$\frac{x}{1-x}\big(G_{m-1}(x)-x^{m-1}L(x)\big)$$ (see Lemma \ref{lem153a1}). Now add the two contributions.
\end{proof}

The next two lemmas give the generating function for permutations in $S_n(T)$ with exactly $2$ left-right maxima. We leave the proof to the diligent reader (very similar to Case 84).

\begin{lemma}\label{lem153a3}
Let $\pi=i\alpha n\beta\in S_n(T)$ with exactly $2$ left-right maxima and $i\geq2$.
\begin{itemize}
\item If $i-1$ is the left most letter of $\alpha$, then the generating function for such permutations $\pi$ is given by $xG_2(x)$, where $G_2(x)$ is the generating function for permutations in $S_n(T)$ with exactly $2$ left-right maxima.
\item If $i-1$ is in $\alpha$ but not the first letter in $\alpha$, then the generating function for such permutations $\pi$ is given by $x^3L(x)\big((L(x)-1\big)$.
\item If $i-1$ is a letter in $\beta$ such that there are exactly $d$ letters between $n$ and $i-1$ in $\pi$ that are greater than $i$. We denote the generating function for the number of such permutations by $A_d(x)$. Then
    $A_0(x)=\frac{x^3(1-3x+5x^2-4x^2)}{(1-x)(1-2x)^3}$, $A_1(x)-xA_0(x)=\frac{x^5(2-5x+4x^2)}{(1-2x)^5}$,
    and for all $d\geq2$,
    $$A_d(x)-xA_{d-1}(x)=\frac{x^{d+4}(1-x)^d}{(1-2x)^{d+3}}.$$
\end{itemize}
\end{lemma}

\begin{lemma}\label{lem153a4}
Let $G_2(x)$ be the generating function for permutations in $S_n(T)$ with exactly $2$ left-right maxima. Then
$$G_2(x)=\frac{x^2(1 -7x+22x^2-36x^3 +33x^4 -20x^5 + 7x^6 -x^7)}{(1-3x+x^2)(1-x)^3(1-2x)^2}.$$
\end{lemma}
\begin{proof}
With the definitions introduced in Lemma \ref{lem153a3}, we have
$$M(x):=\sum_{d\geq0}A_d(x)=\frac{x^3(1 -4x+9x^2-12x^3+6x^4-x^5)}{(1-3x+x^2)(1-x)^2(1-2x)^2}.$$
Thus, by Lemma \ref{lem153a3},
$$G_2(x)=\frac{M(x)+x^3L(x)(L(x)-1)+x^2L(x)}{1-x},$$
where $x^2L(x)$ counts the permutations $1n\pi'\in S_n(T)$,
and the result follows.
\end{proof}

\begin{theorem}\label{th153a}
Let $T=\{4231,1324,1342\}$. Then
$$F_T(x)=\frac{1-10x+41x^2-87x^3+101x^4-61x^5+15x^6-x^7}{(1-x)^2(1-2x)^3(1-3x+x^2)}.$$
\end{theorem}
\begin{proof}
Note that the number of permutations in $S_n(T)$ with leftmost letter is $n$ is given by $G_1(x)=xF_{\{231,1324\}}(x)=xF_{\{132,4231\}}(x)=\frac{x(1+x(L(x)-1)L(x))}{1-x}$, where $L(x)=\frac{1-x}{1-2x}$, see \cite{MV}. Thus, by Lemma \ref{lem153a2} and Lemma \ref{lem153a4}, we obtain
$$F_T(x)-1-G_1(x)=\frac{x^3}{1-x}L^2(x)+\frac{x}{1-x}(F_T(x)-G_1(x)-1)-\frac{x^3}{(1-x)^2}L(x),$$
which, by solving for $F_T(x)$, completes the proof.
\end{proof}

\subsection{Case 156: $T=\{1324,2341,2431\}$}
\begin{theorem}\label{th156a}
Let $T=\{1324,2341,2431\}$. Then
$$F_T(x)=\frac{1-8x+23x^2-25x^3+3x^4+7x^5}{(1-2x)^2(1-3x+x^2)(1-2x-x^2)}.$$
\end{theorem}
\begin{proof}
Let $G_m(x)$ be the generating function for $T$-avoiders with $m$
left-right maxima. Clearly, $G_0(x)=1$ and $G_1(x)=xF_T(x)$.

Let us write an equation for $G_2(x)$. Let $\pi=i\pi'n\pi''\in S_n(T)$ with exactly $2$ left-right maxima. The contributions for the cases $n-2\geq i=1$ and $i=n-1\geq 1$ are $x^2(M(x)-1)$ and $x\big(F_T(x)-1\big)$ respectively, where $M(x)=\frac{1-x-x^2}{1-2x-x^2}$ is the generating function for $\{213,2341,2431\}$-avoiders (for example, see \cite{MV}). Denote the contribution for the case $2\leq i\leq n-2$ by $H(x)$. Then $G_2(x)=x^2\big(M(x)-1\big)+x\big(F_T(x)-1\big)+H(x)$. To find a formula for $H(x)$ we consider the position of $i-1$ in $\pi$, which leads to either $\pi=i\alpha(i-1)\beta'n\beta''\beta'''$ with $\beta'''>i\alpha>\beta'\beta''$ or $\pi=i\alpha'n\alpha''(i-1)\beta'\beta''$ with $\beta''>i>\alpha'\alpha''>\beta'$. By examining the four possibilities, $\alpha,\beta'\beta''$  either empty or not, in the first case and examining the four possibilities for $\alpha'\alpha'',\beta'$ in the second case, we obtain that
\begin{align*}
H(x)&=x^3\big(M(x)-1\big)+x^3\big(K(x)-1\big)\big(M(x)-1\big)+\frac{x^4}{1-x}\big(K(x)-1\big)\big(M(x)-1\big)\\
&+x^3\big(M(x)-1\big)+x^3\big(K(x)-1\big)\big(M(x)-1\big)+\frac{x^4}{1-x}\big(K(x)-1\big)\big(M(x)-1\big)+2xH(x),
\end{align*}
where $K(x)=\frac{1-2x}{1-3x+x^2}$ is the generating function for $\{132,2341\}$-avoiders. Thus,
$$H(x)=\frac{2x^4(1-x)^2}{(1-2x)(1-3x+x^2)(1-2x-x^2)}.$$

Next, let us write an equation for $G_3(x)$. Let $\pi=i_1\pi'i_2\pi''i_3\pi'''\in S_n(T)$
with exactly $3$ left-right maxima. Since $\pi$ avoids $T$, we can
write $\pi=i_1\pi'i_2\pi''i_3\beta\alpha$ where $\pi''<\pi'<i_1<\beta<i_2<\alpha<i_3$.
By considering the four cases where $\pi'',\beta$ are empty or not, we obtain
$$G_3(x)=x^3\left(K(x)M(x)+K(x)\big(M(x)-1\big)+\frac{\big(K(x)-1\big)M(x)}{1-x}
+\frac{\big(K(x)-1\big)\big(M(x)-1\big)}{1-x}\right).$$

Finally, let us write an equation for $G_m(x)$ with $m\geq4$. Let $\pi=i_1\pi^{(1)}\cdots i_m\pi^{(m)}\in S_n(T)$ with exactly $m$ left-right maxima. Since $\pi$ avoids $T$, we see that $i_1>\pi^{(1)}>\pi^{(2)}$, $\pi^{(j)}=\emptyset$ for $j=3,4,\ldots,m-1$, and $\pi^{(m)}$ has the form $\beta\alpha$ where $i_m>\alpha>i_{m-1}>\beta$. Thus, $\pi$ avoids $T$ if and only if the permutation that is obtained from $\pi$ by removing $i_{m-2}$ avoids $T$. Hence $G_m(x)=xG_{m-1}(x)=\dots =x^{m-3}G_3(x)$, for $m\geq3$.

By summing over $m\geq4$ and using the expressions for $G_0(x),G_1(x),G_2(x)$ and $G_3(x)$, we obtain
\begin{align*}
&F_T(x)=1+xF_T(x)+x\big(F_T(x)-1\big)+x^2\big(M(x)-1\big)+H(x)\\  &\,+\frac{x^3}{1-x}\left(K(x)M(x)+K(x)\big(M(x)-1\big)+\frac{\big(K(x)-1\big)M(x)}{1-x}+\frac{\big(K(x)-1\big)\big(M(x)-1\big)}{1-x}\right),
\end{align*}
which, by solving for $F_T(x)$, completes the proof.
\end{proof}

\subsection{Case 158: $\{1324,1342,3412\}$}
Here, it is convenient to consider $J(x)$, the generating function for $T$-avoiders whose maximal letter occurs in second position, and its refinement to $J_d(x)$, the generating function for $T$-avoiders whose maximal letter occurs in second position and whose first letter is $d$, that is,
permutations of the form $dn\pi''\in S_n(T)$.
\begin{lemma}\label{lem158a1} With the preceding notation,
\[
J_1(x)=\frac{x^2(1-x)}{1-2x},\quad J_2(x)=\frac{x^3(1-2x+2x^2)}{(1-2x)^2}, \quad J(x)=\frac{x^2(1-2x)}{(1-x)(1-3x)}\, .
\]
\end{lemma}
\begin{proof}
Suppose $\pi \in S_n(T)$ has the form $dn\pi''$. Since $\pi$ avoids $3412$, we see that $\pi''$ contains the subsequence $(d-1)(d-2)\cdots1$. Therefore, if $d=n-1$, the contribution is $x^{d+1}$ and otherwise, by considering whether $n-1$ occurs before or after the string $(d-1)(d-2)\cdots1$, we find
$$J_d(x)=x^{d+1}+\sum_{j=1}^dx^jJ_{d+1-j}+\sum_{j=1}^dx^j(J_{d+1-j}-x^{d+2-j}),$$
for all $d\geq1$.
Solving this recurrence successively for $d=1$ and $d=2$ gives the first two stated expressions.
Also, summing over all $d\geq1$ yields
$$J(x)=\frac{x^2}{1-x}+\frac{2x}{1-x}J(x)-\frac{x^3}{(1-x)^2},$$
from which the expression for $J(x)$ follows.
\end{proof}

\begin{lemma}\label{lem158a2}
Let $G_2(x)$ be the generating function for the number of $T$-avoiders
with exactly $2$ left-right maxima. Then
$G_2(x)=\frac{x^2(1-7x+19x^2-23x^3+9x^4)}{(1-x)^2(1-2x)(1-3x)(1-3x+x^2)}$.
\end{lemma}
\begin{proof}
Again, we refine  $G_2(x)$ by defining $G_2(x;d)$ to be the generating function for the number of permutations $i\pi'n\pi''\in S_n(T)$ with exactly $2$ left-right maxima where $\pi''$ has $d$ letters smaller than $i$. Since $\pi$ avoids $3412$, we see that $\pi''$ is decreasing. Now, let us write an equation for $G_2(x;d)$. Clearly, $G_2(x;0)=x^2F_{\{132,3412\}}(x)F_{\{213,231\}}(x)=\frac{x^2(1-x)}{1-2x}K(x)$, where $F_{\{132,3412\}}(x)=K(x)=\frac{1-2x}{1-3x+x^2}$ (see \cite[Seq A001519]{Sl}) and $F_{\{213,231\}}(x)=\frac{1-x}{1-2x}$ (see \cite{SiS}).

For $d=1$, let $j$ denote the letter in $\pi''$ smaller than $i$. Thus $i>2$ and $\pi$ can be written, according as $j>1$ or $j=1$, as either
(1) $\pi=i(i-1)\cdots(j+1)\alpha^{(1)}n(n-1)\cdots (i'+1)i'j\beta^{(1)}$ where
$1 \in \alpha^{(1)}$ and $j>\alpha^{(1)}$, $i'>\beta^{(1)}$ and $\beta^{(1)}$ avoids $213$ and $231$, and $\alpha^{(1)}$ avoids $132$ and $3412$, or (2) $\pi=i\alpha n\beta$ where $\alpha$ consists of the letters $2,3,\dots,i-1$ in some order
and $\alpha$ avoids $132$ and $3412$.
Hence, $$G_2(x;1)=\frac{x^3}{(1-x)(1-2x)}(K(x)-1)+K(x)J_2(x),$$
where $J_2(x)$ is defined above.

Now, let $d\geq2$. Since $\pi$ avoids $1324$ and $1342$, we can express $\pi$ as
$$\pi=i\alpha^{(1)}\alpha^{(2)}\cdots\alpha^{(d+1)}n\beta^{(1)}j_1\cdots\beta^{(d)}j_d\beta^{(d+1)}$$
where $i>\alpha^{(d+1)}>j_1>\alpha^{(d)}>j_2>\cdots>\alpha^{(2)}>j_d>\alpha^{(1)}$. We consider three cases:
\begin{itemize}
\item $\alpha^{(s)}\neq\emptyset$ and $\alpha^{(s+1)}=\cdots=\alpha^{(d+1)}=\emptyset$ with $s=3,4,\ldots,d+1$. In this case $\alpha^{(1)},\ldots,\alpha^{(s-1)}$ are decreasing and
$\alpha^{(s)}$ avoids $132$ and $3412$, while $\beta^{(2)}=\cdots=\beta^{(d+1)}=\emptyset$ and $\beta^{(1)}$ is decreasing. Hence, we have a contribution of $\frac{x^{d+2}}{(1-x)^{s}}\big(K(x)-1\big)$.

\item $\alpha^{(3)}=\cdots=\alpha^{(d+1)}=\emptyset$ and $\alpha^{(2)}\neq\emptyset$. In this case, $\alpha^{(1)}$ and $\beta^{(1)}$ are decreasing, and $\beta^{(1)}>\beta^{(2)}> \cdots >\beta^{(d+1)}$. Thus, we have a contribution of $\frac{x}{(1-x)^2}\big(K(x)-1\big)J_d(x)$.

\item $\alpha^{(2)}=\cdots=\alpha^{(d+1)}=\emptyset$. In this case, we have a contribution of $K(x)J_{d+1}(x)$.
\end{itemize}
Hence, for all $d\geq2$,
$$G_2(x;d)=\sum_{s=3}^{d+1}\frac{x^{d+2}}{(1-x)^s}\big(K(x)-1\big)+\frac{x}{(1-x)^2}\big(K(x)-1\big)J_d(x)+K(x)J_{d+1}(x).$$
Summing over $d\geq 0$ and using Lemma \ref{lem158a1}, we obtain the stated expression for
$G_2(x)$.
\end{proof}

\begin{theorem}\label{th158a}
Let $T=\{1324,1342,3412\}$. Then
$$F_T(x)=\frac{1-10x+40x^2-81x^3+88x^4-50x^5+11x^6}{(1-x)^3(1-2x)(1-3x)(1-3x+x^2)}.$$
\end{theorem}
\begin{proof}
Let $G_m(x)$ be the generating function for $T$-avoiders with $m$ left-right maxima. Clearly, $G_0(x)=1$ and $G_1(x)=xF_T(x)$, and $G_2(x)$ is given by Lemma \ref{lem158a2}.

Now, let us write an equation for $G_m(x)$ with $m\geq3$. A $T$-avoider $i_1\pi^{(1)}\cdots i_m\pi^{(m)}$ with $m\ge 3$ left-right maxima
$i_1,i_2,\dots,i_m$  has the restricted form shown in Figure \ref{figAK3},
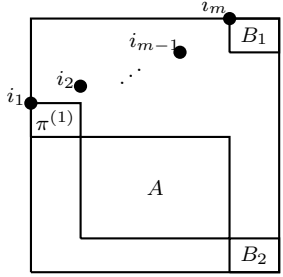
\begin{figure}[htp]
\begin{center}
\begin{pspicture}(4,4.2)
\psset{xunit=.66cm}
\psset{yunit=.45cm}
\psline(0,1)(5,1)(5,8.5)(0,8.5)(0,1)
\psline(1,2)(4,2)(4,5)(1,5)(1,2)
\psline(0,5)(1,5)(1,6)(0,6)(0,5)
\psline(4,1)(5,1)(5,2)(4,2)(4,1)
\psline(4,7.5)(5,7.5)(5,8.5)(4,8.5)(4,7.5)
\psdots[linewidth=1.5pt](0,6)(1,6.5)(3,7.5)(4,8.5)
\rput(.5,5.5){\textrm{{\footnotesize $\pi^{(1)}$}}}
\rput(2.5,3.5){\textrm{{\footnotesize $A$}}}
\rput(4.5,1.5){\textrm{{\footnotesize $B_2$}}}
\rput(4.5,8){\textrm{{\footnotesize $B_1$}}}
\rput(-.3,6.2){\textrm{{\footnotesize $i_1$}}}
\rput(.7,6.7){\textrm{{\footnotesize $i_2$}}}
\rput(2.5,7.8){\textrm{{\footnotesize $i_{m-1}$}}}
\rput(3.7,8.9){\textrm{{\footnotesize $i_m$}}}
\rput(2.,7){\textrm{{\footnotesize $\iddots$}}}
\end{pspicture}
\caption{A $\{1324, 1342, 3412\}$-avoider with $m\ge 3$ left-right maxima}\label{figAK3}
\end{center}
\end{figure}
where $i_1,i_2,\dots,i_{m-1}$ are increasing consecutive integers, $A$ is a list of zero or more decreasing consecutive integers, $\pi^{(m)}$ is composed of $B_1$ and $B_2$, and regions not marked are empty. Furthermore, $\pi^{(1)}$ avoids $\{132,3412\}$ and, of course, $i_1 i_m \pi^{(m)}$
avoids $T$. Conversely, every permutation of this form with $m\ge 3$ satisfying the latter two conditions is a $T$-avoider. Hence we get contributions as follows:  $K(x)$ from $\pi^{(1)}$; $J(x)$ from $i_1 i_m \pi^{(m)}$; $x^{m-2}$ from $i_2,\dots,i_{m-1}$; and $1/(1-x)^{m-2}$
from $A$.
So $G_m(x)=x^{m-2}K(x)J(x)/(1-x)^{m-2}$. By summing over $m\geq3$, we obtain
$$F_T(x)-1-xF_T(x)-G_2(x)=\frac{x^3(1-2x)}{(1-x)(1-3x)(1-3x+x^2)}\, .$$
Now solve for $F_T(x)$ using Lemma \ref{lem158a2} to complete the proof.
\end{proof}

\subsection{Case 180: $\{1342,2314,4231\}$}
\begin{theorem}\label{th180a}
Let $T=\{1342,2314,4231\}$. Then
\[
F_T(x)=\frac{1-7 x +18 x^2 -22 x^3 +16 x^4 -6  x^5 +x^6 - \left(x- 5 x^2+8 x^3 - 2 x^4 -2 x^5 +x^6\right)C(x)  }{(1-2 x) (1-x)^2 \left(1-5 x+4 x^2-x^3\right)}\,.
\]
\end{theorem}
\begin{proof}
Let $G_m(x)$ be the generating function for $T$-avoiders with $m$
left-right maxima. Clearly, $G_0(x)=1$ and $G_1(x)=xF_{\{231\}}(x)=xC(x)$ (see \cite{K}).
Now let $m\geq3$ and let us write equation for $G_m(x)$.
Let $\pi=i_1\pi^{(1)}i_2\pi^{(2)}\cdots i_m\pi^{(m)}\in S_n(T)$ with exactly $m$ left-right maxima. Since $\pi$ avoids $2314$ and 1342, we see that $\pi^{(j)}>i_{j-1}$ for all $j=2,3,\ldots,m-1$. If $\pi^{(m)}$ has a letter smaller than any letter in $\pi^{(1)}$, then $i_j\pi^{(j)}=i_j(i_j-1)\cdots(i_{j-1}+1)$ for all $j=1,2,\ldots,m-1$, and $\pi^{(m)}=i_m(i_m-1)\cdots(i_{m-1}+1)\beta$ such that $\beta$ is non empty permutation in $S_{i_0}(231)$. Hence, we get a contribution of $\frac{x^m}{(1-x)^m}(C(x)-1)$. Otherwise, $\pi^{(j)}>i_{j-1}$ for all $j=2,3,\ldots,m$, which gives a contribution of $x^mC^m(x)$. Hence,
$$G_m(x)=\frac{x^m}{(1-x)^m}(C(x)-1)+x^mC^m(x),$$
for all $m\geq3$.

Now, let us focus on the case $m=2$. First, let $H$ be the generating function for permutations $i\pi'n\pi''\in S_n(T)$ with exactly $2$ left-right maxima and containing the subsequence $n(n-1)\cdots(i+1)$. Let us write an equation for $H$. If $i=1$, then we have a contribution of $x^2/(1-x)$. Otherwise, $n-1\geq i\geq2$ and consider the position of $i-1$.
If $(i-1)$ is the first letter in  $\pi'$ then we have a contribution of $xH$.
If $(i-1)\in \pi'$ but is not the first letter in  $\pi'$, then $\pi$ must have the form
$i\alpha(i-1)\beta n(n-1)\cdots(i+1)$ with $\alpha<\beta<i-1$, $\alpha \ne \emptyset$,
and $\alpha,\beta$ both $231$-avoiders, which implies a contribution of
$x^3\big(C(x)-1\big)C(x)/(1-x)$. Otherwise, $i-1 \in \pi''$ and then $\pi$ can be decomposed
as in Figure \ref{figAK4}.
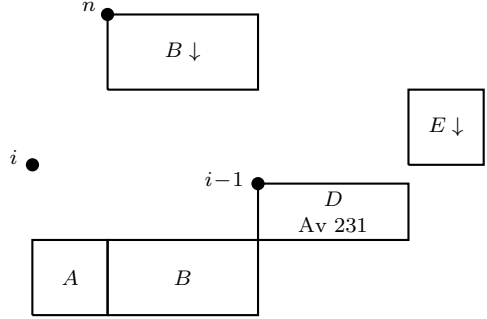
\begin{figure}[htp]
\begin{center}
\begin{pspicture}(5,4.3)
\psset{xunit=.5cm}
\psset{yunit=.5cm}
\psline(0,0)(2,0)(2,2)(0,2)(0,0)
\psline(2,6)(6,6)(6,8)(2,8)(2,6)
\psline(2,0)(6,0)(6,3.5)(10,3.5)(10,2)(2,2)(2,0)
\psline(10,4)(12,4)(12,6)(10,6)(10,4)
\psdots[linewidth=1.5pt](0,4)(2,8)(6,3.5)
\rput(1,1){\textrm{{\footnotesize $A$}}}
\rput(4,7){\textrm{{\footnotesize $B\downarrow$}}}
\rput(4,1){\textrm{{\footnotesize $B$}}}
\rput(8,3.1){\textrm{{\footnotesize $D$}}}
\rput(8,2.4){\textrm{{\footnotesize Av $231$}}}
\rput(11,5){\textrm{{\footnotesize $E\downarrow$}}}
\rput(-.5,4.2){\textrm{{\footnotesize $i$}}}
\rput(5.1,3.6){\textrm{{\footnotesize $i\!-\!1$}}}
\rput(1.5,8.2){\textrm{{\footnotesize $n$}}}
\end{pspicture}
\caption{A permutation counted by $H$ with $i-1$ after $n$}\label{figAK4}
\end{center}
\end{figure}
where $\downarrow$ indicates a region of decreasing entries. Since St($iAnB$) ($B$ is spread over two regions) is of the type counted by $H$ and $D$ contributes $C(x)$, we get a contribution of $x/(1-x)C(x)H$.
Hence,
$$H=x^2/(1-x)+xH+x^3(C(x)-1)C(x)/(1-x)+xC(x)H/(1-x),$$
which leads to
$$H=\frac{x^2/(1-x)+x^3(C(x)-1)C(x)/(1-x)}{1-x-xC(x)/(1-x)}.$$

Now let us write an equation for $G_2(x)$. Let $\pi=i\pi'n\pi''$ with exactly $2$ left-right maxima. If $i=1$, then we have a contribution of $x^2C(x)$. Otherwise, $n-1\geq i\geq2$ and again consider the position of $i-1$. If $\pi'=(i-1)\pi'''$, then we have a contribution of $xG_2(x)$. If $\pi'=\alpha(i-1)\beta$ such that $\alpha$ is not empty then $\pi=i\alpha(i-1)\beta n\pi''$ with $\alpha<\beta<i-1<\pi''<n$, which gives a contribution of $x^3\big(C(x)-1\big)C(x)^2$. Thus, we can assume that $i-1$ belongs to $\pi''$. In this case, $\pi$ can be written as $\pi=i\pi'n\alpha'(i-1)\alpha''\alpha'''$ such that $i\pi'n\alpha'$ has $2$ left-right maxima and contains the subsequence $n(n-1)\cdots (j+1)$, each letter in $\alpha''$ is greater than each letter smaller than $i$ in $\pi'\alpha'$, and $i<\alpha'''<j+1$, where $\alpha'',\alpha'''$ avoids $231$ and $i\pi'n\alpha'$ avoids $T$. Thus, we have a contribution of $xHC(x)^2$. Hence,
$$G_2(x)=x^2C(x)+xG_2(x)+ x^3\big(C(x)-1\big)C(x)^2+xHC(x)^2,$$
which implies
$$G_2(x)=\frac{x^2C(x)+x^3\big(C(x)-1\big)C(x)^2+xHC(x)^2}{1-x}.$$

Summing over $m\geq0$, we obtain
$$F_T(x)=1+xC(x)+\frac{x^2C(x)+x^3\big(C(x)-1\big)C^2(x)+xHC(x)^2}{1-x}+\frac{x^3\big(C(x)-1\big)}{(1-x)^2(1-2x)}+\frac{x^3C^3(x)}{1-xC(x)}\,,$$
and this expression simplifies to the stated form.
\end{proof}

\subsection{Case 184: $\{1324,2431,3241\}$}
\begin{theorem}\label{th184a}
Let $T=\{1324,2431,3241\}$. Then
$$F_T(x)=\frac{1-8x+24x^2-32x^3+19x^4-3x^5}{(1-x)(1-2x)(1-3x+x^2)^2}.$$
\end{theorem}
\begin{proof}
Let $G_m(x)$ be the generating function for $T$-avoiders with $m$
left-right maxima. Clearly, $G_0(x)=1$ and $G_1(x)=xF_T(x)$.

Let us write an equation for $G_m(x)$ with $m\geq2$. Suppose $\pi\in S_n(T)$ has exactly $m$ left-right maxima. Since $\pi$ avoids $3241$ and $2431$, we can express $\pi$ as
$$\pi=i_1\pi^{(1)}i_2\pi^{(2)}\cdots i_{m-1}\pi^{(m-1)}
i_m\pi^{(m)}\rho^{(1)}\rho^{(2)}\cdots\rho^{(m-1)}$$
where
$\pi^{(1)}<\cdots<\pi^{(m)}<i_1<\rho^{(1)}<i_2<\rho^{(2)}<\cdots<i_{m-1}<\rho^{(m-1)}<i_m$. Since $\pi$ avoids $1324$ we see that at most one element of
$L:=\{\pi^{(1)},\ldots,\pi^{(m-1)}\}$ is nonempty and at most one element of
$R:=\{\rho^{(1)},\ldots,\rho^{(m-1)}\}$ is nonempty. Thus, we have the following cases:
\begin{itemize}
\item if $L$ and $R$ are both lists of empty words, then $\pi$ avoids $T$ if and only if $\pi^{(m)}$ avoids $T$, so we have a contribution $x^mF_T(x)$.
\item if, say, $\pi^{(s)}\in L$ is nonempty and $R$ is a list of empty words, then $\pi$ avoids $T$ if and only if $\pi^{(s)}$ avoids $132$ and $3241$, and $\pi^{(m)}$ avoids $213$ and $2431$. Thus, we have a contribution of $x^mK(x)\big(K(x)-1\big)$, where $K(x)=\frac{1-2x}{1-3x+x^2}$
is the generating function for $\{132,3241\}$-avoiders \cite[Seq. A001519]{Sl}. By symmetry,
$K(x)$ is also the generating function for
$\{213,2431\}$-avoiders since $R\circ C \circ I(\{213,2431\} )= \{132,3241\})$, where $R=$ reverse, $C=$ complement, and $I=$ inverse on permutations.
\item if $\rho^{(t)}\in R$ is nonempty and $L$ is a list of empty words, then, similarly, we have the same contribution of $x^mK(x)\big(K(x)-1\big)$.
\item lastly, suppose $\pi^{(s)}\in L$ and $\rho^{(t)}\in R$ are nonempty. If $t\geq s+1$, then
$\pi$ avoids $T$ if and only if $\pi^{(m)}=\emptyset$, $\pi^{(s)}$ avoids $132$ and $3241$,
and $\rho^{(t)}$ avoids $213$ and $2431$, which gives a contribution of $x^m\big(K(x)-1\big)^2$.
Otherwise, $1\leq t\leq s$, and we have the same conditions except $\pi^{(m)}$ must be increasing rather than empty, which leads to a contribution of $\frac{x^m}{1-x}\big(K(x)-1\big)^2$.
\end{itemize}
By adding all the contributions, we find that for all $m\geq2$,
$$G_m(x)=x^mF_T(x)+2(m-1)x^m\big(K(x)-1\big)K(x)+x^m\left(\binom{m-1}{2}+\binom{m}{2}\frac{1}{1-x}\right)\big(K(x)-1\big)^2.$$
Summing  over $m\geq2$, we obtain
$$F_T(x)-1-xF_T(x)=\frac{x^2}{1-x}F_T(x)+\frac{2x^3(1-3x+2x^2)}{(1-x)^2(1-3x+x^2)^2}+\frac{x^4(1+x-x^2)}{(1-x)^2(1-3x+x^2)^2},$$
and solving for $F_T(x)$ completes the proof.
\end{proof}

\subsection{Case 187: $\{1324,2314,2431\}$}
\begin{theorem}\label{th187a}
Let $T=\{1324,2314,2431\}$. Then
$$F_T(x)=\frac{1-9x+31x^2-49x^3+34x^4-7x^5}{(1-3x+x^2)^2(1-2x)^2}.$$
\end{theorem}
\begin{proof}
Let $G_m(x)$ be the generating function for $T$-avoiders with $m$
left-right maxima. Clearly, $G_0(x)=1$ and $G_1(x)=xF_T(x)$.

Let us write an equation for $G_2(x)$. Suppose $\pi \in S_n(T)$ has 2 left-right maxima so that
$\pi=i\pi'n\pi''$. If $i=n-1$, then the contribution is given by $x(F_T(x)-1)$.
Otherwise, we denote the contribution by $H$. So $G_2(x)=x(F_T(x)-1)+H$. For $H$, $\pi$ can be written as $\pi=i\pi'n\alpha\beta$ with $1\le i \le n-2$ and $\pi'\alpha<i<\beta$ and hence
$\beta\neq\emptyset$. Note that $\beta$ avoids both $213$ and $2431$ and $\pi'\alpha$
avoids both $231$ and $132$. By considering whether $\pi'\alpha$ is empty or not and whether
$i-1$ belongs to $\pi'$ or to $\alpha$, we find that
$$H=x^2(L-1)+\big(xH+x^3K(L-1)\big)+\big(xH+x^3K(L-1)\big),$$
where $L=\frac{1-2x}{1-3x+x^2}$ is the generating function for $\{213,2431\}$-avoiders
\cite[Sequence A001519]{Sl} and $K=\frac{1-x}{1-2x}$ is the generating function for $\{132,231\}$-avoiders \cite{SiS}.
Hence, $H=\frac{x^3(1-x)(1-2x+2x^2)}{(1-3x+x^2)(1-2x)^2}$, which implies
$$G_2(x)=x\big(F_T(x)-1\big)+\frac{x^3(1-x)(1-2x+2x^2)}{(1-3x+x^2)(1-2x)^2}.$$

Now, let us write an equation for $G_m(x)$ with $m\geq3$. Suppose $\pi=i_1\pi^{(1)}i_2\pi^{(2)}\cdots i_m\pi^{(m)}\in S_n(T)$ with exactly $m$ left-right maxima. Since $\pi$ avoids $1324$ and $2314$, we see that $\pi^{(2)}=\cdots=\pi^{(m-1)}=\emptyset$. We consider the following two cases:
\begin{itemize}
\item $\pi^{(m)}<i_{m-1}$. In this case, by removing the letter $n$, we obtain a bijection between such permutations of length $n$ and $T$-avoiders with $m-1$ left-right maxima of length $n-1$. Therefore, we have a contribution of $xG_{m-1}(x)$.

\item $\pi^{(m)}$ contains a letter between $i_{m-1}$ and $i_m=n$. In this case, since $\pi$ avoids $1324$ and $2314$, we see that $\pi^{(m)}>i_{m-1}$. So $\pi$ avoids $T$ if and only if $\pi^{(m)}$ avoids $213$ and $2431$ and $\pi^{(1)}$ avoids $231$ and $132$. So the contribution
for this case is $x^mK(L-1)$.
\end{itemize}
Thus, $G_m(x)=xG_{m-1}(x)+x^mK(L-1)$.
Summing  over $m\geq3$, we obtain
$$F_T(x)-G_2(x)-G_1(x)-G_0(x)=x(F_T(x)-G_1(x)-G_0(x))+\frac{x^3}{1-x}K(L-1).$$
Hence, using the expressions for $G_2(x)$, $G_1(x)$ and $G_0(x)$,  we obtain
$$(1-2x)F_T(x)-\frac{(2x^3+6x^2-5x+1)(1-x)^3}{(1-3x+x^2)(1-2x)^2}
=x(1-x)F_T(x)-\frac{x(1-5x+7x^2-3x^3+x^4)}{(1-3x+x^2)(1-2x)},$$
and solving for $F_T(x)$ completes the proof.
\end{proof}

\subsection{Case 193: $\{1324,2431,3142\}$} Observe that each pattern here contains 132. So if a permutation avoids 132, then it certainly avoids $T$.
\begin{theorem}\label{th193a}
Let $T=\{1324,2431,3142\}$. Then
\[
F_T(x)=\frac{x-1+ \left(x^2-5 x+2\right)C(x)}{1-3 x+x^2}\, .
\]
\end{theorem}
\begin{proof}
Let $G_m(x)$ be the generating function for $T$-avoiders with $m$
left-right maxima. Clearly, $G_0(x)=1$ and $G_1(x)=xF_T(x)$.

Now let us write an equation for $G_m(x)$ with $m\geq2$. Let $\pi=i_1\pi^{(1)}i_2\pi^{(2)}\cdots i_m\pi^{(m)}\in S_n(T)$ with exactly $m$ left-right maxima. We see that
$i_1>\pi^{(1)}>\pi^{(2)}>\cdots>\pi^{(m-1)}$ (to avoid 1324) and $\pi^{(m)}$ can be written as $\alpha^{(1)}\alpha^{(2)}\cdots\alpha^{(m)}$ such that $\pi^{(m-1)}>\alpha^{(1)}$
and $i_1<\alpha^{(2)}$ (to avoid 3142) and
$\alpha^{(2)}<i_2<\alpha^{(3)}<\cdots<i_{m-1}<\alpha^{(m)}<i_m$ (to avoid 2431).
Furthermore, $\pi^{(j)}$ avoids $132$ for all $j=1,2,\ldots,m-1$  (or $i_m$ is the 4 of a 1324),
and at most one of $\alpha^{(2)}, \ldots,\alpha^{(m)}$ is nonempty (to avoid 1324).
We consider two cases according as $\alpha^{(2)}=\cdots=\alpha^{(m)}=\emptyset$ or not:
\begin{itemize}
\item
$\alpha^{(2)}=\cdots=\alpha^{(m)}=\emptyset$. Here, $\alpha^{(1)}$ only needs to avoid $T$
and we have a contribution of $x^mC(x)^{m-1}F_T(x)$
\item  There is a unique $j$ in the interval $[2,m]$ such that $\alpha^{(j)}\ne \emptyset$.
Here, $\alpha^{(1)}$ must avoid $132$ (or $\alpha^{(j)}$ contains the 4 of 1324), and
$\alpha^{(j)}$ must avoid both 213 $\approx 324$ (or $i_1$ is the 1 of a 1324) and $2431$.
Also, $\pi^{(j)}= \cdots = \pi^{(m)}=\emptyset$ (to avoid 3142). So $\pi^{(1)},\ldots,\pi^{(j-1)},\alpha^{(1)}$ each contribute $C(x)$ and we get a contribution of
$x^m C(x)^j \big(K(x)-1\big)$ where $K(x)=\frac{1-2x}{1-3x+x^2}$ is the generating function
for $\{213,2431\}$ avoiders  \cite[Seq. A001519]{Sl}.
\end{itemize}
Thus, $$G_m(x)=x^mC(x)^{m-1}F_T(x)+\sum_{j=2}^m x^m C(x)^j\big(K(x)-1\big)\,.$$

Summing over $m\geq2$ and using the expressions for $G_1(x)$ and $G_0(x)$, we obtain
$$F_T(x)=1+xF_T(x)+\sum_{m\geq2}\left(x^mC(x)^{m-1}F_T(x)+\sum_{j=2}^mx^mC(x)^j(K(x)-1)\right)\, .$$
Solving for $F_T(x)$ gives the stated \gf after simplification.
\end{proof}

\subsection{Case 195: $\{1324,2341,1243\}$}
In this subsection, let $A=\frac{1-2x}{1-3x+x^2}$ and $B=\frac{1-x+x^3}{(1-x)(1-x-x^2)}$  denote the generating functions for $F_{\{213,2341\}}(x)$ and $F_{\{132,213,2341\}}(x)$ (they can be derived from results in \cite{MV}). The first few lemmas refer to permutations with exactly 2 left-right maxima.

\begin{lemma}\label{lem195a1}
The generating function $J_m(x)$ for permutations of the form $$(n-m-1)\alpha^{(m)}n\alpha^{(m-1)}(n-1)\cdots\alpha^{(0)}(n-m)\in S_n(T)$$
with $2$ left-right maxima satisfies
$$J_m(x)=x^{m+2}+xJ_m(x)+\frac{x^{m+3}}{(1-x)^{m+1}}(A-1)+\sum_{j=1}^mx^jJ_{m+1-j}(x)\,.$$
Moreover, if $J(x,w)=\sum_{m\geq1}J_m(x)w^{m-1}$, then
$$J(x,w)=\frac{x^3((1+2w)x^3-(5+3w)x^2+(4+w)x-1)}{(wx+x-1)(x^2-3x+1)(wx^2-wx-2x+1)}\,.$$
\end{lemma}
\begin{proof}
Let us write an equation for $J_m(x)$. Let $\pi=(n-m-1)\alpha^{(m+1)}n\alpha^{(m-1)}(n-1)\cdots\alpha^{(0)}(n-m)\in S_n(T)$. If $n=m+2$ then we have a contribution of $x^{m+2}$. Otherwise, we can consider the position of the letter $n-m-2$. If the letter $n-m-2$ belongs to $\alpha^{(s)}$ with $s=1,2,\ldots,m$, then there is no letter smaller than $n-m-2$ on its left side (otherwise, $\pi$ contains $1243$), so we have a contribution of $x^{m+1-s}J_{s}$. So, we can assume that the letter $n-m-2$ belongs to $\alpha^{(0)}$, that is, $\alpha^{(0)}=\gamma(n-m-2)\gamma'$. Since $\pi$ avoids $1324$, we have $\alpha^{(m)}\cdots\alpha^{(1)}\gamma>\gamma'$. In the case $\gamma'=\emptyset$ then we have a contribution of $xJ_m(x)$. Otherwise, since $\pi$ avoids $2341$, we have that $\alpha^{(m)}\cdots\alpha^{(1)}\gamma$ is decreasing and $\gamma'$ avoids $\{132,2341\}$, which gives a contribution of $\frac{x^{m+3}}{(1-x)^{m+1}}(A-1)$. Hence, by adding all contributions, we obtain
$$J_m(x)=x^{m+2}+xJ_m(x)+\frac{x^{m+3}}{(1-x)^{m+1}}(A-1)+\sum_{j=1}^mx^jJ_{m+1-j}(x)\,.$$
Multiplying by $w^{m-1}$ and summing over $m\geq1$, we complete the proof.
\end{proof}

\begin{lemma}\label{lem195a2}
The generating function $K_m(x)$ for $T$-avoiders of the form
$$\pi=i\pi'n\alpha^{(1)}(i+m)\cdots\alpha^{(m)}(i+1)$$
with $2$ left-right maxima satisfies $K_1(x)=J_1(x)+K_2(x)$ and  for all $m\geq2$,
$$K_m(x)=J_m(x)+K_m(x)+\frac{x^{m+3}}{(1-x)(1-2x)(1-x-x^2)},$$
where $J_m(x)$ is defined in Lemma \ref{lem195a1}. Moreover,
$$K_1(x)=\frac{x^3(3x^7-10x^6+x^5+30x^4-42x^3+26x^2-8x+1)}{(2x-1)(x^2-3x+1)^2(x-1)^2(x^2+x-1)}\,.$$
\end{lemma}
\begin{proof}
Let us write an equation for $K_m(x)$. Clearly, $K_1(x)=J_1(x)+K_2(x)$. Let $m\geq2$, and let $\pi=i\pi'n\alpha^{(1)}(i+m)\cdots\alpha^{(m)}(i+1)\in S_n(T)$ with 2 left-right maxima. If $i=n-m-1$ then we have a contribution of $J_m(x)$ (see Lemma \ref{lem195a1}). Otherwise, since $\pi$ avoids $1243$ and $1324$, the letter $i+m+1$ belongs to either $\alpha^{(1)}$ or $\alpha^{(2)}$. The former case gives contribution of $K_{m+1}(x)$. The latter case gives a contribution of $\frac{x^{m+2}}{(1-x)(1-2x)(1-x-x^2)}$, where the proof details are left to the reader. Thus, for all $m\geq2$,
$$K_m(x)=J_m(x)+K_m(x)+\frac{x^{m+3}}{(1-x)(1-2x)(1-x-x^2)}\,.$$
Summing over $m\geq2$, we obtain
$$K_2(x)=\sum_{m\geq2}J_m(x)+\frac{x^5}{(1-x)^2(1-2x)(1-x-x^2)}\,.$$
Hence, since $K_1(x)=J_1(x)+K_2(x)$, we have
$$K_1(x)=\sum_{m\geq1}J_m(x)+\frac{x^5}{(1-x)^2(1-2x)(1-x-x^2)}\,,$$
where $\sum_{m\geq1}J_m(x)=J(x,1)$ is given in Lemma \ref{lem195a1}.
\end{proof}

\begin{lemma}\label{lem195a3}
The generating function $M(x)$ for permutations of the form $i\pi'n\pi''(i+1)\pi'''(i+2)\in S_n(T)$ with $2$ left-right maxima is given by
\begin{align*}
M(x)&=xK_1(x)+\frac{x^4}{(1-x)(1-2x)}(A-1)\\
&=\frac{x^4(2x^7-6x^6-x^5+22x^4-30x^3+20x^2-7x+1)}{(2x-1)(x^2-3x+1)^2(x-1)^2(x^2+x-1)}.
\end{align*}
\end{lemma}
\begin{proof}
Let us write an equation for $M(x)$. By Lemma \ref{lem195a2}, the case $\pi'''=\emptyset$ gives a contribution of $xK_1(x)$. Otherwise, $\pi'\pi''>\pi'''$ and the subsequence of $\pi'\pi''$ consisting of letters smaller than $i$ is decreasing, and the subsequence of $\pi''$ consisting of letters greater than $i$ is also decreasing, and $\pi'''$ is a nonempty permutation that avoids $\{132,2341\}$. Hence, we have a contribution of $\frac{x^4}{(1-x)(1-2x)}(A-1)$. By adding all the contributions, we complete the proof.
\end{proof}

\begin{lemma}\label{lem195a4}
The generating function $N(x)$ for permutations of the form $i\pi'n\pi''(i+2)\pi'''(i+3)(i+1)\in S_n(T)$ with $2$ left-right maxima is given by
\begin{align*}
N(x)&=\frac{x^5}{(2x-1)(x^2+x-1)}\,.
\end{align*}
\end{lemma}
\begin{proof}
Let us write an equation for $N(x)$. If $\pi'''\neq\emptyset$, then, as in the previous lemma, the subsequence of $\pi'\pi''$ consisting of letters smaller than $i$ is decreasing, and the subsequence of $\pi''$ consisting of letters greater than $i$ is also decreasing. Thus, we have a contribution of $\frac{x^6}{(1-x)(1-2x)}$. If $\pi'''=\emptyset$, we get a contribution of $\frac{x^5(x^3+(1-x)^2)}{(1-x)(1-2x)(1-x-x^2)}$ (proof omitted). Adding the two contributions, we complete the proof.
\end{proof}

\begin{lemma}\label{lem195a5}
Define the following generating functions for $T$-avoiders with $2$ left-right maxima:
\begin{itemize}
\item $B_m(x)$  for permutations of the form  $i\pi'n\pi''\in S_n(T)$  such that $\pi''$ contains the subsequence $(i+m)(i+m-1)\cdots(i+1)$;
\item $H_m(x)$ for permutations of the form $$(n-m-1)\alpha^{(m+1)}n\alpha^{(m)}(n-1)\cdots\alpha^{(1)}(n-m)\alpha^{(0)}\in S_n(T);$$
\item $E_m(x)$ for permutations of the form  $$(n-m-1)\alpha^{(m+1)}n\alpha^{(m)}(n-1)\cdots\alpha^{(1)}(n-m)\alpha^{(0)}\in S_n(T)$$ such that the letter $(n-m-2)$
belongs to $\alpha^{(1)}$;
\item $D_m(x)$ for permutations of the form $$(n-m-1)\alpha^{(m+1)}n\alpha^{(m)}(n-1)\cdots\alpha^{(1)}(n-m)\alpha^{(0)}\in S_n(T)$$ such that the letter $(n-m-2)$
belongs to $\alpha^{(0)}$.
\end{itemize}
Then for all $m\geq1$,
\begin{align*}
(a)\,\,& B_m(x)=H_m(x)+B_{m+1}(x)+\frac{x^{m-2}}{1-x}N(x)\mbox{ with }B_1(x)=H_1(x)+B_2(x)+\frac{1}{1-x}M(x);\\
(b)\,\,& H_m(x)=x^{m+2}+E_m(x)+D_m(x)+\sum_{j=1}^mx^jH_{m+1-j}(x);\\
(c)\,\,& D_m(x)=H_{m+1}(x);\\
(d)\,\,& E_m(x)=x^{m+2}(F_T(x)-1)+\frac{x^{m+4}B}{1-x}(t^{m+1}-1)+x^{m+3}(t^{m+1}-1)(A-1/(1-x))\\
&\qquad\qquad+x(J_m(x)-x^{m+2})+x^{m+3}(B-1)(t^{m+1}-1),
\end{align*}
with $t=1/(1-x)$, where $M(x),N(x),J_m(x)$ are defined in the three preceding lemmas.
\end{lemma}
\begin{proof}
(a) To write an equation for $B_1(x)$, suppose $\pi=i\pi'n\pi''\in S_n(T)$ with 2 left-right maxima. If $n=i+1$, then we have $H_1(x)$. Otherwise, the letter $i+2$ appears either on the left side of $i+1$, which leads to a contribution of $B_2(x)$, or on right side of $i+1$. In the latter case $\pi$ can be written as
$\pi=i\pi'n\gamma(i+1)\gamma'(i+2)(i+3)\cdots(i+s)$ with $s\geq2$, which gives a contribution of $\frac{1}{1-x}M(x)$, see Lemma \ref{lem195a3}. Thus, $B_1(x)=H_1(x)+B_2(x)+\frac{1}{1-x}M(x)$.

Let us write an equation for $B_m(x)$ with $m\geq2$. As in the case $B_1(x)$, we obtain
$$B_m(x)=H_m(x)+B_{m+1}(x)+\frac{x^{m-2}}{1-x}N(x)\,,$$
where $N(x)$ is defined in Lemma \ref{lem195a4}.

(b) The recurrence for the generating function $H_m(x)$ can be obtained by using very similar techniques as in Lemma \ref{lem195a1} using the definitions of $E_m(x)$ and $D_m(x)$.

(c) By mapping each permutation $(n-m-1)\alpha^{(m+1)}n\alpha^{(m)}(n-1)\cdots\alpha^{(1)}(n-m)\alpha^{(0)}$ to
$(n-m-2)\alpha^{(m+1)}n\alpha^{(m)}(n-1)\cdots\alpha^{(1)}(n-m-1)\alpha^{(0)}$, we obtain the required relation.

(d) Let us write an equation for $E_m(x)$. Let $$(n-m-1)\alpha^{(m+1)}n\alpha^{(m)}(n-1)\cdots\alpha^{(1)}(n-m)\alpha^{(0)}\in S_n(T)$$
with 2 left-right maxima and $\alpha^{(1)}=\gamma(n-m-2)\gamma'$ where the letter $n-m-2$ belongs to $\alpha^{(1)}$. The contribution of the case
$\alpha'=\alpha^{(m+1)}\cdots\alpha^{(2)}\gamma=\emptyset$ is $x^{m+2}(F_T(x)-1)$. So we can assume that $\alpha'\neq\emptyset$. If $\gamma'\neq\emptyset$ then by considering whether
$\gamma'$ is decreasing or not, we get contributions of $\frac{x^{m+4}}{1-x}(t^{m+1}-1)B$ and $x^{m+1}(t^{m+1}-1)(A-1/(1-x))$, respectively. Thus, we can assume that $\alpha'\neq\emptyset$ and $\beta=\emptyset$, and by considering either $\alpha^{(0)}$ is empty or not, we obtain the contributions $x(J_m(x)-x^{m+2})$ and $x^{m+3}(t^{m+1}-1)(B-1)$, respectively. Sum all  contributions to complete the proof.
\end{proof}

By Lemma \ref{lem195a5} (b) and (c), we have
$$H_m(x)=x^{m+2}+E_m(x)+H_{m+1}(x)+\sum_{j=1}^mx^jH_{m+1-j}(x)\,.$$
Define $H(x,w)=\sum_{m\geq1}H_m(x)w^{m-1}$ and $E(x,w)=\sum_{m\geq1}E_m(x)w^{m-1}$. Thus, this recurrence can be written as
$$H(x,w)=\frac{x^3}{1-xw}+\frac{x}{1-xw}H(x,w)+E(x,w)+\frac{1}{w}\big(H(x,w)-H(x,0)\big)\, ,$$
which implies
$$\left(w-1-\frac{xw}{1-xw}\right)H(x,w)=\frac{x^3w}{1-xw}+E(x,w)-H(x,0)\, .$$
This type of functional equation can be solved systematically using
the {\em kernel method} (see, e.g., \cite{HM} for an exposition) by taking $w=C(x)$. We find
\begin{align}
H(x,0)=\frac{x^3C(x)}{1-xC(x)}+E(x,C(x))=x^3C^2(x)+E(x,C(x))\, .  \label{eq195H1}
\end{align}
Lemma \ref{lem195a5}(d) and some mathematical programming yields
\begin{align*}
E(x,w)&=\frac{x^3}{1-wx}(F_T(x)-1)\\
&-\frac{x^5(-3x^7+11x^6-11x^5-5x^4+20x^3-22x^2+11x-2)}{(x-1)^3(wx+x-1)(x^2-3x+1)(wx^2-wx-2x+1)(wx-1)(x^2+x-1)}\\
&-\frac{x^6w((x^2-5x+3)(x^3-x^2-2x+1)+x(x-1)(x^3-x^2-2x+1)w)}{(x-1)^3(wx+x-1)(x^2-3x+1)(wx^2-wx-2x+1)(wx-1)}\, ,
\end{align*}
which leads to
\begin{align}
E(x,C(x))&=x^3C(x)(F_T(x)-1)\notag\\
&-\frac{x^5C^6(x)(-3x^7+11x^6-11x^5-5x^4+20x^3-22x^2+11x-2)}{(x-1)^3(x^2-3x+1)(x^2+x-1)}\label{eq195E1}\\
&-\frac{x^6C^7(x)((x^2-5x+3)(x^3-x^2-2x+1)+x(x-1)(x^3-x^2-2x+1)C(x))}{(x-1)^3(x^2-3x+1)}.\notag
\end{align}

Now we are ready to give the main result of this subsection.

\begin{theorem}\label{th195a}
Let $T=\{1324,2341,1243\}$. Then $F_T(x)=$
\[
\frac{ (1 - 7 x + 19 x^2 - 25 x^3 + 13 x^4 + 4 x^5 - 5 x^6 + x^7)C(x) -1 + 7 x - 19 x^2 + 23 x^3 - 7 x^4 - 7 x^5 + 4 x^6}{x(1 - x)^2  (1 - 3 x + x^2) (1 - x - x^2)}\, .
\]

\end{theorem}
\begin{proof}
Let $G_m(x)$ be the generating function for $T$-avoiders with $m$
left-right maxima. Clearly, $G_0(x)=1$ and $G_1(x)=xF_T(x)$.

First, we treat $G_2(x)$. Each $\pi=i\pi'n\pi''\in S_n(T)$ with 2 left-right maxima can be decomposed as either $(n-1)\pi'n\pi''$ or $i\pi'n\pi''$ with $1\leq i<n-2$ with respective contributions to $G_2(x)$ of $x(F_T(x)-1)$ and $B_1(x)$. From Lemma \ref{lem195a5}(a) we have
$$B_1(x)=\sum_{m\geq1}H_m(x)+\frac{M(x)}{1-x}+\frac{xN(x)}{(1-x)^2},$$
which implies
$$B_1(x)=H(x,1)+\frac{M(x)}{1-x}+\frac{xN(x)}{(1-x)^2},$$
where $M(x)$ and $N(x)$ are given in Lemmas \ref{lem195a3} and \ref{lem195a4}, respectively, and
$$H(x,1)=\frac{1-x}{x}\left(H(x,0)-\frac{x^3}{1-x}-E(x,1)\right),$$
where
$$E(x,1)=\frac{x^3\big(F_T(x)-1\big)}{1-x}-\frac{x^5(2x^7-11x^6+15x^5+2x^4-21x^3+25x^2-12x+2)}
{(x-1)^3(2x-1)(x^2-3x+1)^2(x^2+x-1)}$$
and $H(x,0)$ and $E(x,C(x))$ are given in \eqref{eq195H1} and \eqref{eq195E1}.

Now, let us write a formula for $G_3(x)$. Let $\pi=i_1\pi'i_2\pi''i_3\pi'''\in S_n(T)$ with 3 left-right maxima. then $i_1<\pi'''<i_2$ and $\pi'>\pi''$. By considering whether $\pi',\pi''$ are empty or not, we obtain the contributions $x^3A$ ($\pi'=\pi''=\emptyset$), $x^3(A-1)/(1-x)$ ($\pi'\neq\emptyset,\pi''=\emptyset$), $x^3B/(1-x)^2$ ($\pi'=\emptyset,\pi''\neq\emptyset$) and $0$. Thus,
$$G_3(x)=x^3A+\frac{x^3(A-1)}{1-x}+\frac{x^3B}{(1-x)^2}\,.$$

Now, let us write a formula for $G_m(x)$ with $m\geq4$. Let $i_1\pi^{(1)}\cdots i_m\pi^{(m)}\in S_n(T)$ with $m$ left-right maxima. Then $\pi^{(3)}=\cdots=\pi^{(m)}=\emptyset$ and $\pi^{(1)}>\pi^{(2)}$. By considering whether $\pi^{(2)}$ is empty or not, we have
$$G_m(x)=x^mA+\frac{x^m(A-1)}{1-x}\,.$$

Summing over $m\geq0$, we obtain
$$F_T(x)=1+xF_T(x)+G_2(x)+\frac{x^3}{1-x}A+\frac{x^3}{(1-x)^2}(A+B-1).$$
Substituting the formula for $G_2(x)$, we get
\begin{align*}
F_T(x)&=1+xF_T(x)+x(F_T(x)-1)+\frac{1-x}{x}H(x,0)-x^2F_T(x)\\
&+\frac{x^4(2x^7-11x^6+15x^5+2x^4-21x^3+25x^2-12x+2)}{(x-1)^2(2x-1)(x^2-3x+1)^2(x^2+x-1)}\\
&+\frac{M(x)}{1-x}+\frac{xN(x)}{(1-x)^2}+\frac{x^3}{1-x}A+\frac{x^3}{(1-x)^2}(A+B-1).
\end{align*}
where $M(x)$ and $N(x)$ are given in Lemmas \ref{lem195a3} and \ref{lem195a4}, respectively, and
$H(x,0)$ and $E(x,C(x))$ are given in \eqref{eq195H1} and \eqref{eq195E1}, respectively.
Solving for $F_T(x)$ and simplifying the result, we complete the proof.
\end{proof}

\subsection{Case 210: $\{1243,1324,2431\}$}
\begin{theorem}\label{th210a}
Let $T=\{1243,1324,2431\}$. Then
\[
F_T(x)=\frac{\left(1-6 x+13 x^2-11 x^3+4 x^4\right)}{x^2(1-x)^2 }C(x)-
\frac{1-6 x+12 x^2-8 x^3+2 x^4}{x^2(1-x)  (1-2 x)}.
\]
\end{theorem}
\begin{proof}
Let $G_m(x)$ be the generating function for $T$-avoiders with $m$
left-right maxima. Clearly, $G_0(x)=1$ and $G_1(x)=xF_T(x)$.

For $G_m(x)$ with $m\geq2$, suppose $\pi=i_1\pi^{(1)}i_2\pi^{(2)}\cdots i_m\pi^{(m)}\in S_n(T)$ with exactly $m$ left-right maxima. Since $\pi$ avoids $1324$, $\pi^{(s)}<i_1$ for all $s=1,2,\ldots,m-1$. Since $\pi$ avoids $1243$, $\pi^{(m)}<i_2$.
\begin{figure}[htp]
\begin{center}
\begin{pspicture}(7,3.1)
\psset{xunit=.5cm,yunit=.3cm}
\psline[linecolor=lightgray](6,6)(6,9)
\psline[linecolor=lightgray](2,8)(12,8)
\psline(6,0)(6,6)(0,6)(0,3)
\rput(3,4.5){\textrm{{\footnotesize $A\downarrow$}}}
\rput(10.5,4.5){\textrm{{\footnotesize $C\uparrow$}}}
\rput(13.5,7){\textrm{{\footnotesize $D\uparrow,\,\ne \emptyset$}}}
\rput(7.5,2){\textrm{{\footnotesize $B$}}}
\rput(7.5,1){\textrm{{\footnotesize Av $132$}}}
\psdots[linewidth=1.2pt](0,6)(2,8)(6,9)
\rput(-.5,6.3){\textrm{{\footnotesize $i_1$}}}
\rput(1.5,8.3){\textrm{{\footnotesize $i_2$}}}
\rput(3.6,9){\textrm{{\footnotesize $\iddots$}}}
\rput(5.4,9.3){\textrm{{\footnotesize $i_m$}}}
\psline(6,0)(9,0)(9,6)(15,6)(15,8)(12,8)(12,3)(0,3)
\end{pspicture}
\caption{A $T$-avoider with at least one entry between $i_1$ and $i_2$, and these entries increasing}\label{figTK1}
\end{center}
\end{figure}
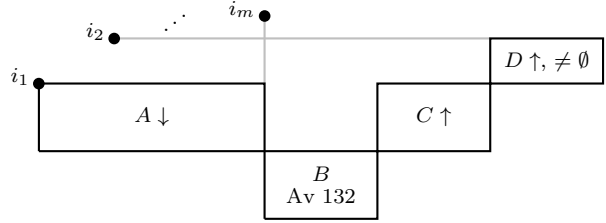
First, suppose $\pi^{(m)}$ has a letter greater than $i_1$ and consider two cases:
\begin{itemize}
\item all the letters between $i_1$ and $i_2$ in $\pi^{(m)}$ are increasing.
Here, $\pi^{(1)}\cdots\pi^{(m-1)}$ is decreasing ($\pi$ avoids $1243$), and $\pi^{(m)}$
can be decomposed as ($\pi$ avoids $2431$)  $\alpha(i_1+1)\cdots(i_2-1)$ and,
writing $\pi^{(1)}\cdots\pi^{(m-1)}$ as $j_sj_{s-1}\cdots j_1$, $\alpha$
can be further decomposed as $\alpha^{(0)}\alpha^{(1)}\cdots\alpha^{(s)}$ with
$\alpha^{(0)}<j_1<\alpha^{(1)}<\cdots<j_s<\alpha^{(s)}<i_1$, where $\alpha^{(0)}$ avoids $132$
and $\alpha^{(i)}$ is an increasing subword for all $i=1,2,\ldots,s$.
In short, $\pi$ has the schematic form shown in the Figure \ref{figTK1} where $\downarrow$ denotes decreasing, $\uparrow$ denotes increasing, and blank regions are empty.
If $A$ is empty, then so is $C$ and, in any case, the position of $i_m$ is determined by $AC$.
From Figure \ref{figTK1}, $\pi$ can be written as $i_1 A_1i_m BCD$,
where $A_1$ consists of $A$ and the left-right maxima $i_2,\dots,i_{m-1}$,
and $\pi$ is determined by St($A_1C$), St($B$), and St($D$). The latter two contribute $C(x)$ and $\frac{x}{1-x}$ respectively. Now $AC$ is a $\{132,231\}$-avoider and St($A_1C$) can be viewed as a $\{132,231\}$-avoider decorated with $m-2$ dots (for
$i_2,\dots,i_{m-1}$) placed in the spaces between its entries but not after the first ascent. For example, the 4 boxes shown accept dots in $\Box 5 \Box 2 \Box1\Box34$. The \gf for such constructs with $M$ dots, each dot counting as a letter, can be shown to be $x^M L(x)^{M+1}$ where $L(x)=\frac{1-x}{1-2x}$ is the \gf for $\{132,231\}$-avoiders.
Thus, with $M=m-2$, the contribution of this case is $\frac{x^{m+1}}{1-x}L(x)^{m-1}C(x)$.
\item the letters between $i_1$ and $i_2$ in $\pi^{(m)}$ do not form an increasing sequence, where the sequence  $\pi^{(1)}\cdots\pi^{(m-1)}\alpha$ is decreasing, and $\pi^{(m)}=\alpha\beta$ with $\alpha<i_1<\beta<i_2$. Thus, we have a contribution of $\frac{x^m}{(1-x)^m}\big(L(x)-\frac{1}{1-x}\big)$.
\end{itemize}
Hence, for all $m\geq2$,
$$G_m(x)=H_m(x)+\frac{x^{m+1}}{1-x}L(x)^{m-1}C(x)+\frac{x^m}{(1-x)^m}\left(L(x)-\frac{1}{1-x}\right),$$
where $H_m(x)$ is the generating function for those permutations $\pi\in S_n(T)$ with $\pi^{(m)}<i_1$ (in other words $i_1=n+1-m$).

Now let us write an equation for $H_m(x)$.
Let $\pi=i_1\pi^{(1)}i_2\pi^{(2)}\cdots i_m\pi^{(m)}\in S_n(T)$ with exactly $m$ left-right maxima such that $i_j=n+j-m$ for all $j$. If $\pi^{(1)}=\emptyset$ then we have a contribution of $xH_{m-1}(x)$. Otherwise, if we assume that $\pi^{(1)}$ has exactly $s$ left-right maxima, it is not hard to see that those permutations are bijection with the permutations with exactly $m+s-1$ left-right maxima. Thus,
$$H_m(x)=xH_{m-1}(x)+x(G_m(x)+G_{m+1}(x)+G_{m+2}(x)+\cdots)$$
with $H_1(x)=G_1(x)$ (by definitions).

Define $G(x;t)=\sum_{m\geq0}G_m(x)t^m$ and $H(x;t)=\sum_{m\geq1}H_m(x)t^m$. Note that $G(x;1)=F_T(x)$. Hence, the above recurrences can be written as
\begin{align*}
H(x;t)&=xtF_T(x)+xtH(x;t)+\frac{xt^2}{1-t}(F_T(x)-1)-\frac{xt}{1-t}(G(x;t)-1),\\
G(x;t)&=1+H(x;t)+\frac{x^3t^2L(x)C(x)}{(1-x)(1-xt)}+\frac{x^2t^2}{(1-x)(1-x-xt)}(L(x)-1/(1-x)).
\end{align*}
Hence,
\begin{align*}
\left(1+\frac{xt}{(1-t)(1-xt)}\right)G(x;t)&=\frac{1}{1-xt}+\frac{xt}{(1-t)(1-xt)}F_T(x)\\
&+\frac{x^3t^2L(x)C(x)}{(1-x)(1-xt)}+\frac{x^2t^2}{(1-x)(1-x-xt)}(L(x)-1/(1-x)).
\end{align*}
This equation can be solved by the kernel method (see, e.g., \cite{HM} for an exposition) taking $t=C(x)$ and leads, after simplification, to our theorem.
\end{proof}

\subsection{Case 211: $\{1234,1324,2341\}$} All three patterns contain 123 and so $T$-avoiders consist of 123-avoiders together with $T$-avoiders that contain a 123. The former are counted by $C(x)$. To count the latter, let $abc$ at positions $i,j,k$ be the leftmost 123 pattern (smallest $i$, then smallest $j$, then smallest $k$). This 123 pattern divides the permutation diagram into rectangular regions as shown
in Figure \ref{figAK6} where the bullets represent the leftmost 123.
\begin{figure}[htp]
\begin{center}
\begin{pspicture}(6,7)
\psset{unit =.8cm, linewidth=.5\pslinewidth}
\psline[fillstyle=hlines,hatchcolor=lightgray,hatchsep=0.8pt](6,0)(8,0)(8,2)(6,2)(6,0)
\psline[fillstyle=hlines,hatchcolor=lightgray,hatchsep=0.8pt](4,6)(6,6)(6,8)(4,8)(4,6)
\psline[fillstyle=hlines,hatchcolor=lightgray,hatchsep=0.8pt](6,6)(8,6)(8,8)(6,8)(6,6)
\psline[fillstyle=hlines,hatchcolor=lightgray,hatchsep=0.8pt](2,2)(6,2)(6,6)(4,6)(4,4)(2,4)(2,2)
\psline[fillstyle=hlines,hatchcolor=lightgray,hatchsep=0.8pt](2,0)(2,4)(0,4)(0,0)(2,0)
\psline(0,0)(8,0)(8,8)(0,8)(0,0)
\psline(0,2)(8,2)
\psline(0,4)(8,4)
\psline(4,6)(8,6)
\psline(2,0)(2,8)
\psline(4,0)(4,8)
\psline(6,0)(6,8)
\pscircle*(2,2){0.1}\pscircle*(4,4){0.1}\pscircle*(6,6){0.1}
\put(.9,5.9){$\al$}
\put(2.8,5.9){$\al'$}
\put(6.7,5.1){$\be\,\downarrow$}
\put(6.4,4.4){\small $(1234)$}
\put(6.8,2.9){ $\be'$}
\put(2.7,1.1){$\ga\,\downarrow$}
\put(2.4,0.4){\small $(1234)$}
\put(4.9,0.9){$\ga'$}
\put(.5,0.8){\small (min)}
\put(.5,2.8){\small (min)}
\put(2.5,2.8){\small (min)}
\put(4.4,2.8){\small$(1324)$}
\put(4.5,4.8){\small (min)}
\put(4.5,6.8){\small (min)}
\put(6.4,6.8){\small$(1234)$}
\put(6.4,0.8){\small$(2341)$}
\put(2,-.5){\small $i$}
\put(4,-.5){\small $j$}
\put(6,-.5){\small $k$}
\put(8.3,2){\small $a$}
\put(8.3,4){\small $b$}
\put(8.3,6){\small $c$}
\end{pspicture}
\caption{The leftmost (min) 123 in a $T$-avoider}\label{figAK6}
\end{center}
\end{figure}
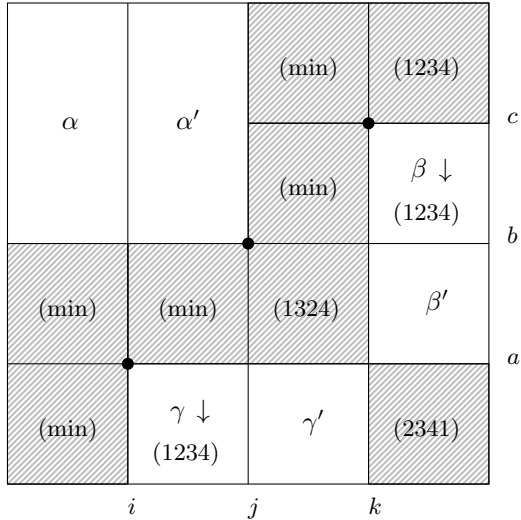
Shaded regions are empty for the indicated reason, where (min) refers to the minimal, that is, leftmost property of the 123 pattern, and unshaded regions are labeled
$\al,\al'$ etc. A down arrow ($\downarrow$) indicates necessarily decreasing entries, again for the indicated reason.
Also, $\al\al'$ avoids 123 (or the middle bullet ends a 2341), while $\be$ lies to the left of $\be'$ and  $\ga>\ga'$
(both to avoid 1324). 
We consider 4 cases according as $\be',\ga'$ are empty or not:
\begin{itemize}
\item $\be',\ga'$ both empty. Here, St($\al\, a\, \al'\, c$) is a a 123-avoider $\tau$ of length $\ge 2$ with last entry $>1$. Say $\tau$ has length $n$, with 1 in position $i<n$ and last entry $\ell\ge 2$. Now $\be$ serves to ``decorate'' $\tau$ with 0 or more dots inserted arbitrarily in the $\ell-1$ spaces between 1 and $\ell$. Similarly, $\ga$ amounts to 0 or more dots inserted in the $n-i$ spaces between $i$ and $n$. Thus, the contribution of this case is
\[
x \sum_{n\ge 2}\sum_{i=1}^{n-1}\sum_{\ell=2}^{n}C(n,i,\ell)x^n \sum_{u,v\ge 0}\binom{u+\ell-2}{u}\binom{v+n-i-1}{v}\,,
\]
where $C(n,i,\ell)$ is the number of $123$-avoiders of length $n$ with 1 in position $i$ and last entry $\ell$.
It is known that $C(n,i,\ell)=\frac{n-2-i+\ell}{n+i-\ell}\binom{n+i-\ell}{n-1}$, a generalized Catalan number, and the displayed sum evaluates to the compact expression $x^3 C(x)^5$.

\item $\be'$ empty, $\ga'$ nonempty. Here, $\ga'$ avoids 123 and 132 and so the contribution is $x^3 C(x)^5 (L-1)=x^3 C(x)^5\frac{x}{1-2x}$ where $L=\frac{1-x}{1-2x}$ is the \gf for $\{123,132\}$-avoiders.

\item $\be'$ nonempty, $\ga'$ empty. Here, $\be'$ avoids 123 and 213 and so the contribution is $x^3 C(x)^5 (L-1)= x^3 C(x)^5\frac{x}{1-2x}$ since $L=\frac{1-x}{1-2x}$ is also the \gf for $\{123,213\}$-avoiders.

\item $\be',\ga'$ both nonempty. In case $\ga'$ is decreasing, the only restriction on $\be'$ is to avoid 123 and 213 (since $\ga' \ne \emptyset$). So $\ga'$ contributes $\frac{x}{1-x}$ and $\be'$ contributes $L-1$. In case $\ga'$ is not decreasing, $\be'$ \emph{is} decreasing (to avoid 1234) and $\ga'$ contributes $L-\frac{1}{1-x}$ while $\be'$ contributes $\frac{x}{1-x}$. So, $\be',\ga'$ together contribute $\frac{x}{1-x}(L-1) +  \big(L-\frac{1}{1-x}\big)\frac{x}{1-x}= \frac{x^2}{(1-x)^2 (1-2x)}$ for an overall contribution of $x^3 C(x)^5 \frac{x^2}{(1-x)^2 (1-2x)}$

\end{itemize}
Hence, summing over 123-avoiders and the 4 cases,
\[
F_T(x)=C(x)+x^3 C(x)^5\left(1 + \frac{x}{1-2x} + \frac{x}{1-2x} + \frac{x^2}{(1-x)^2 (1-2x)}\right)\, .
\]
Simplifying this expression, we have established
\begin{theorem}\label{th211a}
Let $T=\{1234,1324,2341\}$. Then
\[
F_T(x)=\frac{(1 - 4 x + 5 x^2 - 3 x^3)\, C(x) - (1 - 4 x + 6 x^2 - 4 x^3)}{x (1 - x)^2 (1 - 2 x)}\, .
\]
\end{theorem}

\subsection{Case 212: $\{1324,2413,2431\}$} Note that each pattern contains 132.
\begin{theorem}\label{th212a}
Let $T=\{1324,2413,2431\}$. Then
$$F_T(x)=1+\frac{x(1-4x+4x^2-x^3-x(1-4x+2x^2)C(x))}{(1-3x+x^2)(1-3x+x^2-x(1-2x)C(x))}.$$
\end{theorem}
\begin{proof}
Let $G_m(x)$ be the generating function for $T$-avoiders with $m$
left-right maxima. Clearly, $G_0(x)=1$ and $G_1(x)=xF_T(x)$.
For $m=2,$ $\pi$ has the form $i \al n \be$ with $\al<i$, and $\be$ cannot contain both
letters $>i$ and $<i$ (2413 and 2431). If $\be$ has no letter $>i$, then $i=n-1$ and deleting it gives a contribution of $x\big(F_T(x)-1\big)$. Otherwise, $\be>i$ and $\al$ avoids 132
(due to $n$) and $\be$ avoids both 213 (due to $i$) and 2431 and is nonempty,
contributing $x^2 C(x)\big(K(x)-1\big)$, where $K(x)=\frac{1-2x}{1-3x+x^2}$
is the generating function for $\{213,2431\}$-avoiders \cite[Seq. A001519]{Sl}.
So $G_2(x)=x\big(F_T(x)-1\big)+x^2C(x)\big(K(x)-1\big)$.

Now let us write equation for $G_m(x)$ for $m\geq3$.
Suppose $\pi=i_1\pi^{(1)}i_2\pi^{(2)}\cdots i_m\pi^{(m)}\in S_n(T)$ with exactly $m$ left-right maxima. Since $\pi$ avoids $1324$, we have $\pi^{(s)}<i_1$ for all $s=1,2,\ldots,m-1$. Since $\pi$ avoids $2413$ and $2431$, either $\pi^{(2)}$ is empty or $\pi^{(m)}$ has no letter between $i_1$ and $i_2$ (or both). Thus, we have three cases
\begin{itemize}
\item $\pi^{(2)}=\emptyset$ and $\pi^{(m)}$ has no letter between $i_1$ and $i_2$. Here, $i_2$ and its position is determined by the rest of the avoider, and we have a contribution of $xG_{m-1}(x)$.
\item $\pi^{(2)}=\emptyset$ and $\pi^{(m)}$ has a letter between $i_1$ and $i_2$.
Here, $\pi^{(3)}=\cdots=\pi^{(m-1)}=\emptyset$ (or they would contain the 1 of a 2413),
and since $\pi$ avoids both $2413$ and $2431$, $i_1<\pi^{(m)}$.
Additionally, $\pi^{(m)}$ is not empty and avoids both $213$ (or $i_1$ is the 1 of a 1324)
and $2431$, and $\pi^{(1)}$ avoids $132$ (or $i_m$ is the 4 of a 1324).
Hence, the contribution is $x^mC(x)\big(K(x)-1\big)$.
\item $\pi^{(2)}\neq\emptyset$ and $\pi^{(m)}$ has no letter between $i_1$ and $i_2$.
Since $\pi$ avoids $2413$, we see that $i_1>\pi^{(1)}>\pi^{(s)}$ for all $s=2,3,\ldots,m-1$.
Also, $\pi^{(1)}$ avoids 132 and $i_2=i_1+1$. So $\pi$ can be recovered from St($i_1 \pi^{(1)}$)
and St($i_2\pi^{(2)}\cdots i_m\pi^{(m)}$), giving respective contributions of
$xC(x)$ and $G_{m-1}(x)-G'_{m-1}(x)$ where $G'_{m-1}(x)$ counts the $\pi^{(2)}$ empty case, that is, $G'_m(x)$ is the generating function for $T$-avoiders with $m$ left-right maxima in which the first letter is smaller than the second.
Thus, the contribution is $xC(x)\big(G_{m-1}(x)-G'_{m-1}(x)\big)$.
\end{itemize}
Hence, for all $m\geq3$,
\begin{align}
G_m(x)=xG_{m-1}(x)+x^mC(x)\big(K(x)-1\big)+xC(x)\big(G_{m-1}(x)-G'_{m-1}(x)\big)\, .\label{eq212a1}
\end{align}

For $G'_2(x)$, a $T$-avoider has the form $\pi= i n \pi'$.  If $\pi'$ is empty,
the contribution is $x^2$. So suppose $\pi'\neq\emptyset$.
If $i=n-1$ then we have $x^2(F_T(x)-1)$, but if $i<n-1$, then, since $\pi$ avoids 2413 and 2431
we see that $i=1$ and $\pi=1n\pi'$ with $\pi'\neq\emptyset$ avoiding 213 and 2431, giving a contribution of $x^2(K(x)-1)$. So
\begin{equation}\label{212g2p}
G'_2(x)=x^2\big(K(x)+F_T(x)-1\big)\,.
\end{equation}

For $G'_m(x)$ with $m\ge 3$, using the same three cases as above, we find by similar arguments
the recurrence
\begin{eqnarray}
G'_m(x)& = & xG'_{m-1}(x)+x^m(K(x)-1)+x\big(G_{m-1}(x)-G'_{m-1}(x)\big) \notag \\
&=& x^m \big(K(x)-1\big)+xG_{m-1}(x)\, .\label{eq212a2}
\end{eqnarray}

Substituting \eqref{212g2p} and \eqref{eq212a2} into \eqref{eq212a1}, we have
$$G_m(x)=x\big(C(x)+1\big)G_{m-1}(x)-x^2C(x)G_{m-2}(x)$$
for $m\ge 3$, with $G_2,G_1,G_0$ as above.
Summing over $m\geq3$, we obtain
$$F_T(x)-1-xF_T(x)-G_2(x)=x\big(C(x)+1\big)\big(F_T(x)-1-xF_T(x)\big)-
x^2C(x)\big(F_T(x)-1\big)\, .$$
Solving for $F_T(x)$ completes the proof.
\end{proof}

\subsection{Case 213: $\{2431,1324,1342\}$}
\begin{theorem}\label{th213a}
Let $T=\{2431,1324,1342\}$. Then
\[
F_T(x)=\frac{ (1 - 5 x + 8 x^2 - 5 x^3)\,C(x)-1 + 4 x - 4 x^2 + x^3}{x^2(1 - 2 x) }\, .
\]

\end{theorem}
\begin{proof}
Let $G_m(x)$ be the generating function for $T$-avoiders with $m$
left-right maxima. Clearly, $G_0(x)=1$ and $G_1(x)=xF_T(x)$.

First, we find a formula for $G_2(x)$. Since $\pi$
avoids $2431$, we can write a $T$-avoider $\pi$ with 2 left-right maxima as $\pi=i\alpha n\beta\gamma$ where $\alpha\beta<i<\gamma$, and $\gamma$ avoids both 132 and 231 (or $i$ is the 1 of a 1324 or 1342). If $\gamma=\emptyset$, then $\pi=n-1\, \al n\be$ and, by deleting $n-1$, the contribution is $x\big(F_T(x)-1\big)$.  Now suppose $\gamma\neq\emptyset$
and let $H$ denote the contribution of such permutations to $G_2(x)$.
If $\alpha\beta=\emptyset$, then $\pi=1n\gamma$ with $\gamma=\emptyset$ and the contribution is
$x^2(J-1)$ where $J=\frac{1-x}{1-2x}$ is the generating function for $\{213,231\}$-avoiders. Henceforth, $\alpha\beta\ne\emptyset$, and so $i>1$ and consider 3 cases:

- $i-1 \in \al$. Here, $\pi=i\al'\, i-1\,\al''n\be \gamma$ where $\al'>\al''\beta\ (1324)$ and $\al'$ avoids 132. So the contribution is  $xC(x)H$.

- $i-1$ is the last letter of $\be$. Deleting $i-1$ gives a one-size-smaller avoider and the
contribution is $xH$.

- $i-1\in \be$ but is not the last letter of $\be$. Here $\al=\emptyset$ (2431) and $\pi=in\be'\,i-1\,\be'' \gamma$ where $\be''\ne \emptyset$, $\be'>\be''$ (1324) and $\be',\be''$ each avoid 132. So, with $C(x),C(x)-1,J-1$ respectively from $\be',\be'',\gamma$, the  contribution
is $x^3C(x)\big(C(x)-1\big)(J-1)$.

Thus, $H = x^2(J-1)+ xC(x)H +xH+ x^3 C(x)\big(C(x)-1\big)(J-1)$ and so
$$H=\frac{x^2(J-1)+x^3\big(C(x)-1\big)C(x)(J-1)}{1-x-xC(x)},$$
which leads to
$$G_2(x)=x\big(F_T(x)-1\big)+H=x\big(F_T(x)-1\big)+\frac{x^2(J-1)+x^3\big(C(x)-1\big)C(x)(J-1)}{1-x-xC(x)}.$$

Now let $m\geq3$ and let us write an equation for $G_m(x)$.
Let $\pi=i_1\pi^{(1)}i_2\pi^{(2)}\cdots i_m\pi^{(m)}\in S_n(T)$ with exactly $m$ left-right maxima. Since $\pi$ avoids $1324$ and $1342$, we can express $\pi$ as
$$\pi=i_1\pi^{(1)}i_2\pi^{(2)}\cdots i_{m-2}\pi^{(m-2)}i_{m-1}\alpha i_m\beta\gamma$$
such that $i_1>\pi^{(1)}>\cdots>\pi^{(m-2)}>\alpha\beta$ and $i_{m-1}<\gamma$. Note that $\pi$ avoids $T$ if and only if $\pi^{(j)}$ avoids $132$ for $j=1,2,\ldots,m-2$ and the standard form of $i_1\alpha n\beta\gamma$ is a $T$-avoider with two left-right maxima. Hence,
$$G_m(x)= x^{m-2} C(x)^{m-2}G_2(x).$$

By summing  for $m\geq2$ and using the expressions for $G_1(x)$ and $G_0(x)$, we obtain
$$F_T(x)=1+xF_T(x)+\frac{G_2(x)}{1-xC(x)}\, ,$$
with solution as stated.
\end{proof}

For the next two cases we set
$a_T(n)=|S_n(T)|$ and let $a_T(n;j_1,j_2,...,j_s)$ denote the number of permutations in $S_n(T)$ whose first $s$ letters are $j_1j_2...j_s$.

\subsection{Case 231: $\{1324,1342,2341\}$} Set $b(n;j)=a_T(n;j,j+1)$.
\begin{lemma}\label{lem231a1}
For $1\leq j\leq n-2$,
\begin{align*}
a_T(n;j)&=a_T(n-1;1)+\cdots+a_T(n-1;j)+b(n;j)
\end{align*}
and $a_T(n;n)=a_T(n;n-1)=a_T(n;n-2)=a_T(n-1)$.
\end{lemma}
\begin{proof}
The initial conditions  $a_T(n;n)=a_T(n;n-1)=a_T(n-1)$ easily follow from the definitions. For $1\leq j\leq n-2$, we have
\begin{align*}
a_T(n;j)&=\sum_{i=1}^{j-1}a_T(n;j,i)+\sum_{i=j+1}^na_T(n;j,i).
\end{align*}
So assume $1 \leq j\leq n-2$ and let $\pi=ji\pi'$ be a member of $S_n(T)$. We consider several cases on $i$.  Since $\pi$ avoids $1324$ and $1342$, we have that either $1\leq i\leq j+1$ or $i=n$. If
$i=n$, then $\pi$ avoids $T$ if and only if $j\pi'$ avoids $T$, so $a_T(n;j,n)=a_T(n-1,j)$. So,
\begin{align*}
a_T(n;j)&=\sum_{i=1}^{j-1}a_T(n;j,i)+a_T(n-1;j)+b(n;j).
\end{align*}
Let $1\leq i\leq j-1$. If $\pi$ avoids $T$ then $i\pi'$ avoids $T$.  On other hand, if $\pi$ contains $1324$ (resp. $1342$), then $i\pi'$ contains $1324$ (resp. $1324$). Also, if $\pi$ contains $2341$ where $j$ does not occur in the corresponding occurrence, then $i\pi'$ contains $2341$. Thus, we assume that $\pi$ contains $jabc$ with $c<j<a<b$. If $c<i$, then $i\pi'$ contains $2341$, otherwise $i\pi'$ contains $1342$. Therefore, $i\pi'$ avoids $T$ if and only if $\pi$ avoids $T$, which implies $a_T(n-1;j,i)=a_T(n;i)$. Hence,
\begin{align*}
a_T(n;j)&=\sum_{i=1}^{j-1}a_T(n-1;i)+a_T(n-1;j)+b(n;j),
\end{align*}
as required.
\end{proof}

Define $A_T(n;v)=\sum_{j=1}^na_T(n;j)v^{j-1}$ and $B(n;v)=\sum_{j=1}^{n-2}b(n;j)v^{j-1}$. Then Lemma \ref{lem231a1} can be written as
$$A_T(n;v)=\frac{1}{1-v}(A_T(n-1;v)-v^nA_{T}(n-1;1))+B(n;v)$$
with $A_T(0;v)=A_T(1;v)=1$.

Define  $A_T(x,v)=\sum_{n\geq0}A_T(n;v)x^n$ and $B(x,v)=\sum_{n\geq3}B(n;v)x^n$. Then, the above recurrence can be written as
\begin{align}
A(x,v)=1+\frac{x}{1-v}(A(x,v)-vA(xv,1))+B(x,v).\label{eq231a1}
\end{align}
\begin{lemma}\label{lem231a2}
$$B(x,v)=\frac{x^3(1-2xv)}{(1-3xv+x^2v^2)(1-2x)}.$$
\end{lemma}
\begin{proof}
Let $\pi=j(j+1)\pi'\in S_n(T)$ with $1\leq j\leq n-2$. Since $\pi$ avoids $2341$, we can express $\pi$ as $\pi=j(j+1)\alpha\beta$ with $\alpha<j<j+1<\beta$. Note that $\pi$ avoids $T$ if and only if $\alpha$ avoids both $132,2341$ and $\beta$ avoids both $213,231$. By \cite[Seq. A001519]{Sl} we have that $F_{\{132,2341\}}(x)=\frac{1-2x}{1-3x+x^2}$, and by \cite{SiS} we have that $F_{\{213,231\}}(x)=\frac{1-x}{1-2x}$, yielding
\begin{align*}
B(x,v)&=\sum_{n\geq3}\sum_{j=1}^{n-2}|S_{j-1}(132,2341)||S_{n-1-j}(213,231)|v^{j-1}x^n\\
&=x^2F_{\{132,2341\}}(xv)(F_{\{213,231\}}(x)-1),
\end{align*}
which leads to $B(x,v)=\frac{x^3(1-2xv)}{(1-3xv+x^2v^2)(1-2x)}$, as required.
\end{proof}

By \eqref{eq231a1} and Lemma \ref{lem231a2}, we obtain
$$A(x/v,v)=1+\frac{x}{v(1-v)}(A(x/v,v)-vA(x,1))+\frac{x^3(1-2x)}{v^2(1-3x+x^2)(v-2x)}.$$
To solve the preceding functional equation, we apply the kernel method (see, e.g., \cite{HM} for an exposition) and take $v=\frac{1+\sqrt{1-4x}}{2}=1/C(x)$. Then
$$F_T(x)=A(x,1)=C(x)+\frac{x^3(1-2x)C^4(x)}{(1-3x+x^2)(1-2xC(x))}\,.$$
After simplification, this gives the following result.
\begin{theorem}\label{th231a}
Let $T=\{1324,1342,2341\}$. Then
\[
F_T(x)=\frac{(1 - 3 x) \big(1 - 2x-xC(x)\big)}{(1 - 4 x) (1 - 3 x + x^2)}
\]
\end{theorem}

\subsection{Case 241: $\{1324,1243,1234\}$} Set $b(n;j)=a_T(n;j,n-1)$.
\begin{lemma}\label{lem241a1}
For $1\leq j\leq n-2$,
\begin{align*}
a_T(n;j)&=a_T(n-1;1)+\cdots+a_T(n-1;j)+b(n;j),\\
b(n;j)&=b(n-1;1)+\cdots+b(n-1;j-1)+a_T(n-2;j)
\end{align*}
and $a_T(n;n)=a_T(n;n-1)=a_T(n;n-2)=a_T(n-1)$,
$b(n;n)=b(n;n-2)=a_T(n-2)$ and $b(n;n-1)=0$.
\end{lemma}
\begin{proof}
The initial conditions  $a_T(n;n)=a_T(n;n-1)=a_T(n;n-2)=a_T(n-1)$,
$b(n;n)=b(n;n-2)=a_T(n-2)$ and $b(n;n-1)=0$ easily follow from the definitions. For $1\leq j\leq n-2$, we have
\begin{align*}
a_T(n;j)&=\sum_{i=1}^{j-1}a_T(n;j,i)+\sum_{i=j+1}^{n}a_T(n;j,i).
\end{align*}
So assume $1 \leq j\leq n-3$ and let $\pi=ji\pi'$ be a member of $S_n(T)$. We consider several cases for $i$.  If $i=n$, then $a_T(n;j,n)=a_T(n-1;j)$. If $1\leq i<j$ then $ji\pi'$ (respectively, $jn\pi'$) avoids $T$ if and only if $i\pi'$ (respectively, $j\pi'$) avoids $T$, so $a_T(n;j,i)=a_T(n-1;i)$ for all $i=1,2,\ldots,j-1$ and $a_T(n;j,n)=a_T(n-1;j)$. Since $\pi$ avoids $1234$ and $1243$, we see that either $i<j$ or $j>n-2$, and $a_T(n;j,n)=a_T(n-1;j)$. Thus,
\begin{align*}
a_T(n;j)&=\sum_{i=1}^{j-1}a_T(n-1;i)+a_T(n-1;j)+b(n;j),
\end{align*}
which completes the proof of the first recurrence relation.

For the second relation, by similar reasons, we have
\begin{align*}
b(n;j)&=\sum_{i=1}^{j-1}a_T(n;j,n-1,i)+\sum_{i=j+1}^{n-2}a_T(n;j,n-1,i)+a_T(n;j,n-1,n).
\end{align*}
Clearly, $a_T(n;j,n-1,n)=a_T(n-2,j)$. The permutation $j(n-1)(n-2)\pi''$ contains $1324$, so $a_T(n;j,n-1,n-2)=0$. The permutation $j(n-1)i\pi''$ with $j+1\leq i\leq n-3$ contains either $1234$ or $1243$, so $a_T(n;j,n-1,i)=0$.
Thus,
\begin{align*}
b(n;j)&=\sum_{i=1}^{j-1}a_T(n;j,n-1,i)+a_T(n-2;j).
\end{align*}
Let $j(n-1)i\pi''\in S_n$. If $i(n-1)\pi''$ contains a pattern in $T$ then $j(n-1)i\pi''$ contains the same pattern in $T$. Now, suppose $i(n-1)\pi''$ avoids $T$, so if $\pi=j(n-1)i\pi''$ contains $1234$ or $1243$, then $i(n-1)\pi''$ contains $1234$ or $1243$, which implies that $\pi$ avoids $1234$ and $1243$. If $\pi$ contains $1324$, then we can assume that $1324$ occurs in
$\pi$ as a subsequence $jabc$ with $j<b<a<c$, otherwise $i(n-1)\pi'$ contains $1324$. Also, we can assume that $b=n-1$ which gives $c=n$, otherwise $iabc$ occurs in $i(n-1)\pi''$.  Since $j\leq n-3$,  $\pi$ has an element $d$ such that either $j<d<b$ or $b<d<n-1$. Thus, $i(n-1)\pi''$ contains either $idba$ or $ibda$ or $ibad$, that is $i(n-1)\pi''$ does not avoid $T$. Therefore, $\pi$ avoids $T$.  Hence, $$a_T(n;j,n-1,i)=a_T(n-1;i,n-2)=b(n-1;i),$$ for all $i=1,2,\ldots,j-1$, which implies
\begin{align*}
b(n;j)&=\sum_{i=1}^{j-1}b(n-1;i)+a_T(n-2;j)\, .
\end{align*}
This completes the proof.
\end{proof}

Define $A_T(n;v)=\sum_{j=1}^na_T(n;j)v^{j-1}$ and $B(n;v)=\sum_{j=1}^nb(n;j)v^{j-1}$. Then Lemma \ref{lem241a1} can be written as
\begin{align*}
A_T(n;v)&=A_T(n-1;1)(v^{n-1}+v^{n-2}+v^{n-3})+\frac{1}{1-v}(A_T(n-1;v)-v^{n-3}A_T(n-1;1))\\
&+B(n;v)-A_T(n-2;1)v^{n-1},\\
B(n;v)&=A_T(n-2;1)(v^{n-1}+v^{n-3})+\frac{1}{1-v}(vB(n-1;v)-B(n-1;1)v^{n-3})\\
&+A_T(n-2;v)+A_T(n-3;1)v^{n-2}
\end{align*}
with $A_T(0;v)=A_T(1;v)=1$, $A_T(2;v)=1+v$, $B(0;v)=B(1;v)=0$ and $B(2;v)=v$.

Define  $A_T(x,v)=\sum_{n\geq0}A_T(n;v)x^n$ and $B(x,v)=\sum_{n\geq0}B(n;v)x^n$. Then, the above recurrence can be written as
\begin{align}
A_T(x,v)&=1+\left(x+\frac{x}{v}+\frac{x}{v^2}\right)A_T(xv,1)\nonumber\\
&+\frac{x}{1-v}\left(A_T(x,v)-\frac{1}{v^2}A_T(xv,1)\right)+B(x,v)-x^2vA_T(xv,1)\label{eq241a1}\\
B(x,v)&=\frac{x}{1-v}\left(vB(x,v)-\frac{1}{v^2}B(xv,1)\right)+x^3vA_T(xv,1)\nonumber\\
&+\frac{x^2(1+v)}{v}(A_T(xv,1)-1)+x^2vA_T(xv,1).\label{eq241a2}
\end{align}
By multiplying \eqref{eq241a2} by $v^2(1-v)^2\left(1-\frac{xv}{1-v}\right)$ and \eqref{eq241a1} by $v^2(1-v)^2$, then adding the results, we obtain
\begin{align*}
K(x,v)A_T(x/v;v)&=\frac{x(v^3(1-x-v)-x(1-v)(1-v+xv))}{v}A_T(x;1)+\frac{x(1-v)}{v}B(x;1)\\
&+\frac{(1-v)(x^2(1-v^2)-v^3(1-x-v))}{v},
\end{align*}
where $K(x,v)=xv(1-v^2)+x^2(1-3v+v^2)-v^2(1-v)^2$.

To solve the preceding functional equation, we apply the kernel method (see, e.g., \cite{HM} for an exposition) and take
\begin{align*}
v&=v_-=\frac{2+(\sqrt{5}-1)x+\sqrt{4-12x+6x^2-4\sqrt{5}x-2\sqrt{5}x^2}}{4},\,\mbox{and}\\
v&=v_+=\frac{2-(\sqrt{5}+1)x+\sqrt{4-12x+6x^2+4\sqrt{5}x+2\sqrt{5}x^2}}{4},\\
\end{align*}
which satisfies $K(x;v_+)=K(x,v_-)=0$. Hence, since $F_T(x)=A_T(x,1)$, we obtain the following result.
\begin{theorem}\label{th241a} Let $T=\{1324,1243,1234\}$.
Then $F_T(x)$ is given by
\begin{align*}
\frac{(v_--1)(v_+-1)\,\big((v_-+v_+)(v_-^2+v_+^2-x^2)+(x-1)(v_-^2+v_+^2+v_-v_+)\big)}{x\,\big(x-(v_--1)(v_+-1)(v_-^2+v_+^2+v_-v_++x(v_-+v_++2-x))\big)}\,.
\end{align*}
\end{theorem}

\vspace*{10mm}
\centerline{\bf{Acknowledgement}}
We thank Richard Mather for pointing out some inaccuracies in the previous version. 

\newpage


\vspace*{5mm}

{\footnotesize\begin{longtable}[c]{|l|c|c|}
\caption{Small classes of three 4-letter patterns counted by INSENC.\label{longinsenc}}\\ \hline
\multicolumn{3}{| c |}{Start of Table}\\ \hline
No. &$T$&$F_T(x)$ \\ \hline
\endfirsthead  \hline
\multicolumn{3}{|c|}{Continuation of Table \ref{longinsenc}}\\ \hline
No. & $T$&$F_T(x)$ \\ \hline
\endhead \hline
\endfoot \hline
\multicolumn{3}{| c |}{End of Table}\\ \hline\hline
\endlastfoot
1&$\{4321,3412,1234\}$&$73x^9+ 199x^8+ 240x^7+ 162x^6+ 69x^5+ 21x^4+ 6x^3+ 2x^2+ x + 1$\\[4pt]\hline
2&$\{4321,3142,1234\}$&$85x^9+ 221x^8+ 252x^7+ 164x^6+ 69x^5+ 21x^4+ 6x^3+ 2x^2+ x + 1$\\[4pt]\hline
3&$\{2143,4312,1234\}$&$\frac{18x^7+31x^6+22x^5+8x^4+2x^3+2x^2-2x+1}{(1-x)^{3}}$\\[4pt]\hline
4&$\{4231,2143,1234\}$&$\frac{2x^{10}-6x^9+6x^8+4x^7+4x^6+8x^5+6x^4-4x^3+7x^2-4x+1}{(1-x)^{5}}$\\[4pt]\hline
5&$\{2143,3412,1234\}$&$\frac{2x^5+10x^4-11x^3+11x^2-5x+1}{(1-x)^6}$\\[4pt]\hline
7&$\{3421,4312,1234\}$&$\frac{-9x^7+24x^6+23x^5+8x^4+2x^3+2x^2-2x+1}{(1-x)^3}$\\[4pt]\hline
8&$\{2431,4213,1234\}$&$\frac{26+21x+15x^2}{2(1-x-x^2-x^3)}+\frac{4x^{10}(1+3x)-62x^9-28x^8+66x^7+27x^6-15x^5-53x^4+41x^3+51x^2-75x+24}{2(x-1)^3(1-x-x^2)^2}$\\[4pt]\hline
9&$\{2134,4312,1243\}$&$\frac{-3x^7-5x^6+3x^5+10x^4-11x^3+11x^2-5x+1}{(1-x)^6}$\\[4pt]\hline 
10&$\{4213,1432,1234\}$&$\frac{2x^{11}+4x^{10}+10x^9+12x^8+6x^7-19x^6-19x^5-7x^4-x^3+2x-1}{(x-1)(x^5+3x^4+2x^3+x^2+x-1)(x^3+x^2+x-1)}$\\[4pt]\hline
11&$\{4231,1432,1234\}$&$\frac{5x^9-2x^8-x^7+9x^6+9x^5+6x^4-4x^3+7x^2-4x+1}{(1-x)^5}$\\[4pt]\hline
12&$\{2341,4312,1324\}$&$\frac {x^{10}-4x^9+3x^8+ 5x^7-7x^5+ 21x^4- 22x^3+ 16x^2- 6x+ 1}{(1-x)^{7}}$\\[4pt]\hline
13&$\{3214,1432,1234\}$&$-\frac{x^5+x^3+x^2+x-1}{x^{12}+16x^{11}+10x^{10}+17x^9+25x^8+25x^7-7x^6-14x^5-5x^4-2x^3-x^2-2x+1}$\\[4pt]\hline
14&$\{4231,2134,1243\}$&$\frac{4x^9-11x^8+10x^7+2x^6-7x^5+21x^4-22x^3+16x^2-6x+1}{(1-x)^7}$\\[4pt]\hline
16&$\{2314,1432,4123\}$&$\frac{2x^9-x^8+x^7+3x^6+6x^5-6x^4+11x^3-13x^2+6x-1}{(x^2-3x+1)(x^3+x^2+x-1)(1-x)^3}$\\[4pt]\hline
17&$\{2341,2143,4123\}$&$\frac{x^7- 13x^5+ 25x^4 - 29x^3+ 20x^2- 7x + 1}{(x^2- 3x + 1)(1-x)^5}$\\[4pt]\hline
18&$\{2341,1432,4123\}$&$\frac{x^{10}-7x^9+19x^8-25x^7+12x^6+10x^5-20x^4+25x^3-19x^2+7x-1}{(x^2+1)(x^2-3x+1)(x^3-x^{2}-2x+1)(x-1)^3}$\\[4pt]\hline
19&$\{2431,4312,1234\}$&$\frac{6x^9- 7x^8- 7x^7 + 4x^6+ 10x^5+ 6x^4- 4x^3+ 7x^2- 4x+ 1}{(1-x)^5}$\\[4pt]\hline
20&$\{4312,1432,1234\}$&$\frac{(x + 1)(2x^9-18x^8+33x^7-20x^6+12x^5-22x^4+16x^3-12x^2+5x-1)}{(x-1)^5}$\\[4pt]\hline
21&$\{4312,3142,1234\}$&$\frac{2x^9- 3x^8- 2x^{6} - 6x^5+ 21x^4- 22x^3+ 16x^2- 6x + 1}{(1-x)^7}$\\[4pt]\hline
22&$\{2134,4312,1432\}$&$\frac {x^6+ 6x^5- 21x^4+22x^3- 16x^2+ 6x - 1}{(x - 1)^7}$\\[4pt]\hline
23&$\{2431,4132,1234\}$&$\frac{1-2x}{x^2-3x+1}+\frac{(24x^9-116x^8+213x^7-158x^6+9x^5+37x^4-9x^3+x^2-3x+1)x^3}{(x-1)^5(2x-1)^3}$\\[4pt]\hline
24&$\{4231,3412,1234\}$&$\frac{2x^6- 6x^5+ 21x^4- 22x^3+ 16x^2- 6x + 1}{(1-x)^7}$\\[4pt]\hline
25&$\{3412,4132,1234\}$&$\frac{3x^6- 6x^5+ 21x^4- 22x^3+ 16x^2- 6x + 1}{(1-x)^7}$\\[4pt]\hline
26&$\{2134,4312,1342\}$&$-\frac {x^8- 9x^6+ 27x^5- 43x^4+ 38x^3- 22x^2+ 7x - 1}{(1-x)^8}$\\[4pt]\hline
27&$\{2314,4312,1432\}$&$-\frac{3x^9-x^8-18x^7+17x^6+15x^5-44x^4+47x^3-27x^2+8x-1}{(2x-1)(x^2+x-1)(x-1)^6}$\\[4pt]\hline
28&$\{4231,3142,1234\}$&$\frac{2x^8- 10x^7+ 40x^6- 70x^5+ 81x^4- 60x^3+ 29x^2- 8x + 1}{(1-x)^9}$\\[4pt]\hline
31&$\{2314,4312,1342\}$&$\frac{5x^{10}-22x^9+12x^8+89x^7-249x^6+354x^5-316x^4+179x^3-62x^2+12x-1}{(x^2-3x+1)(2x-1)^3(x-1)^4}$\\[4pt]\hline
32&$\{2134,1432,4123\}$&$\frac{x^{10}-4x^9+4x^8-x^6-5x^5+6x^4-11x^3+13x^2-6x+1}{(x^2-3x+1)(x^3+x^2+x-1)(x-1)^3}$\\[4pt]\hline
33&$\{2134,3412,4132\}$&$\frac {2x^7- 16x^5+ 36x^4 - 42x^3+ 26x^2- 8x + 1}{(2x - 1)^{3}(x - 1)^3}$\\[4pt]\hline
34&$\{2143,4132,1234\}$&$-\frac{x^9-2x^8-x^7+4x^6-x^5-2x^4+3x^3-8x^2+5x-1}{(x^2-3x+1)(x-1)(2x-1)}$\\[4pt]\hline
36&$\{3412,3124,1432\}$&$-\frac{x^8+2x^7-26x^6+62x^5-83x^4+69x^3-34x^2+9x-1}{(x-1)^{5}(x^2-3x+1)(2x-1)}$\\[4pt]\hline
37&$\{3142,1432,1234\}$&$\frac {(x^3- 2x^2+ 3x - 1)^2}{x^8- x^7+ 4x^6- 7x^5+ 19x^4- 24x^3 + 18x^2- 7x+1}$\\[4pt]\hline
38&$\{4321,1423,1234\}$&$147x^9+ 359x^8+ 367x^7+ 198x^6+72x^5+ 21x^4+ 6x^3+ 2x^2+ x + 1$\\[4pt]\hline
39&$\{4321,4123,1234\}$&$185x^9+ 400x^8+ 396x^7+ 205x^6+72x^5+ 21x^4+ 6x^3+ 2x^2+ x + 1$\\[4pt]\hline
40&$\{2341,4312,1234\}$&$\frac{x^9- 5x^8+ 6x^7+ x^6+ 5x^5- 21x^4+ 22x^3- 16x^2+ 6x - 1}{(x - 1)^7}$\\[4pt]\hline
41&$\{4312,1342,1234\}$&$-\frac {2x^7- 8x^6+ 26x^5- 43x^4+ 38x^3- 22x^2+ 7x - 1}{(x - 1)^8}$\\[4pt]\hline
42&$\{2341,4132,1234\}$&$\frac{4x^8- 5x^7- 7x^{6} - 7x^5+ 22x^4- 28x^3+ 20x^2- 7x + 1}{(2x-1)^2(x-1)^4}$\\[4pt]\hline
43&$\{2314,4213,1432\}$&$-\frac{9x^6-35x^5+54x^4-49x^3+27x^2-8x+1}{(3x^3-5x^2+4x-1)(2x-1)(x-1)^3}$\\[4pt]\hline
44&$\{4213,1342,1234\}$&$\frac{x^{10}-6x^9+9x^8+9x^7-54x^6+94x^5-104x^4+76x^3-35x^2+9x-1}{(x^3-2x^2+3x-1)(2x-1)(x-1)^5}$\\[4pt]\hline
45&$\{4213,2134,1432\}$&$\frac{x^{10}-2x^9-x^8-13x^7+54x^6-99x^5+108x^4-77x^3+35x^2-9x+1}{(x-1)^{2}(3x^3-5x^2+4x-1)^2}$\\[4pt]\hline
46&$\{2341,4132,1324\}$&$\frac{2x^7+5x^6-3x^{5}+3x^4+6x^3-12x^2+6x-1}{(x-1)(2x-1)(x^2-3x+1)(x^2+x-1)}$\\[4pt]\hline
47&$\{2413,4132,1234\}$&$-\frac{3x^6-21x^5+40x^4-43x^3+26x^2-8x+1}{(2x-1)(x-1)^{4}(x^2-3x+1)}$\\[4pt]\hline
48&$\{4312,3124,1342\}$&$-\frac{x^9-15x^8+73x^7-175x^6+247x^5-228x^4+138x^3-52x^2+11x-1}{(x^2-3x+1)^{2}(x-1)^6}$\\[4pt]\hline
51&$\{4213,3124,1432\}$&$\frac{x^6- 7x^4+ 12x^3-13x^2+ 6x - 1}{(x^2- 3x + 1)(3x^3- 5x^2+4x - 1)}$\\[4pt]\hline
52&$\{1432,4123,1234\}$&$-\frac{x^8- 4x^7+ 3x^6+ 4x^5- 11x^4+ 20x^3- 18x^2+ 7x - 1}{(x-1)^2(x^2-3x+1)^2}$\\[4pt]\hline
53&$\{2134,4132,1243\}$&$\frac{x^{10}-4x^9-6x^8+68x^7-186x^6+291x^5-283x^4+170x^3-61x^2+12x-1}{(2x-1)^2(x^2-3x+1)^2(x-1)^3}$\\[4pt]\hline
54&$\{3124,1432,1234\}$&$\frac{(1-x)^3(2x^3-2x^2+3x-1)}{2x^9-7x^8+7x^7-10x^6+16x^5-27x^4+29x^3-19x^2+7x-1}$\\[4pt]\hline
57&$\{2143,1432,1234\}$&$\frac {x^7+ x^6- x^5+ 3x^3+ 2x^2+ 2x - 1}{x^7+ x^6- x^5- x^4+ 2x^3+x^2+3x-1}$\\[4pt]\hline
58&$\{4321,1243,1234\}$&$144x^9+ 396x^8+ 382x^7+ 202x^6+73x^5+ 21x^4+ 6x^3+ 2x^2+ x + 1$\\[4pt]\hline
59&$\{4321,1324,1234\}$&$334x^9+ 669x^8+ 484x^7+ 215x^6+73x^5+ 21x^4+ 6x^3+ 2x^2+ x + 1$\\[4pt]\hline
60&$\{4312,4132,1234\}$&$\frac{x^7+16x^6+12x^5+6x^4-4x^3+7x^2-4x+1}{(1-x)^5}$\\[4pt]\hline
61&$\{4312,1243,1234\}$&$\frac{x^{10} -4x^9+ 3x^8+2x^7+ x^6+ 4x^5- 21x^4+ 22x^3- 16x^{2}+ 6x - 1}{(x-1)^7}$\\[4pt]\hline
62&$\{4231,4312,1234\}$&$\frac{3x^8- 8x^7+ 4x^6 - 4x^5+ 21x^4- 22x^3+ 16x^2- 6x + 1}{(1-x)^7}$\\[4pt]\hline
63&$\{4312,1324,1234\}$&$\frac{x^{10}-5x^9+ 6x^8+2x^7- 5x^6+ 4x^5- 21x^4+ 22x^3- 16x^2 + 6x -1}{(x-1)^7}$\\[4pt]\hline
64&$\{4312,3412,1234\}$&$\frac {3x^7+ 5x^6- 4x^5 + 21x^4- 22x^3+ 16x^2- 6x + 1}{(1-x)^7}$\\[4pt]\hline
65&$\{4213,4132,1234\}$&$-\frac{3x^8+5x^7+13x^6+7x^5+2x^4+x^3+5x^2-4x+1}{(x^2+x-1)(x^3+x^2+x-1)(x-1)^3}$\\[4pt]\hline
66&$\{4231,4132,1234\}$&$\frac{2x^7+8x^6-4x^5+21x^4-22x^3+16x^2-6x+1}{(1-x)^7}$\\[4pt]\hline
67&$\{4312,1324,1243\}$&$\frac{2x^{10}-7x^8+65x^7-187x^6+274x^5-248x^4+145x^3-53x^2+11x-1}{(x-1)^6(2x-1)^3}$\\[4pt]\hline
68&$\{4312,1342,1243\}$&$\frac{3x^7-4x^6-14x^5+36x^4-42x^3+26x^2-8x+1}{(x-1)^3(2x-1)^3}$\\[4pt]\hline
70&$\{4312,3124,1243\}$&$-\frac{11x^7-62x^6+128x^5-146x^4+102x^3-43x^2+10x-1}{(2x-1)^3(x-1)^5}$\\[4pt]\hline
71&$\{4231,1243,1234\}$&$-\frac{4x^8-2x^7-17x^6+25x^5-43x^4+38x^3-22x^2+7x-1}{(x-1)^8}$\\[4pt]\hline
73&$\{4231,1324,1234\}$&$-\frac{x^{10}-15x^8+55x^7-111x^6+149x^5-141x^4+89x^3-37x^2+9x-1}{(x-1)^{10}}$\\[4pt]\hline
79&$\{2134,4132,1234\}$&$\frac{2x^{11}+x^{10}-10x^9-9x^8+12x^7+17x^6-30x^5+2x^4+28x^3-24x^2+8x-1}{(x^2+2x-1)(2x-1)(x-1)^3(x^2+x-1)^2}$\\[4pt]\hline
81&$\{2431,4312,1324\}$&$\frac{145x^3+11x-1-248x^4-193x^6+274x^5-53x^2-x^9-13x^8+80x^7)}{(2x-1)^3(x-1)^6}$\\[4pt]\hline
82&$\{4312,3142,1243\}$&$\frac{x^7+2x^6-27x^5+59x^4-61x^3+33x^2-9x+1}{(x-1)^{2}(2x-1)^4}$\\[4pt]\hline
83&$\{4312,3412,1243\}$&$\frac{x^7+2x^6+4x^5-23x^4+36x^3-25x^2+8x-1}{(x-1)(2x-1)^4}$\\[4pt]\hline
85&$\{2314,4132,1432\}$&$-\frac{x^5+5x^4-11x^3+13x^2-6x+1}{(2x-1)(x^2-3x+1)(x-1)^2}$\\[4pt]\hline
87&$\{4312,3124,1432\}$&$-\frac{2x^9-46x^7+143x^6-226x^5+221x^4-137x^3+52x^2-11x+1}{(x-1)^5(x^2-3x+1)(2x-1)^2}$\\[4pt]\hline
89&$\{3142,4132,1234\}$&$-\frac{4x^5-16x^4+24x^3-19x^2+7x-1}{(2x-1)(x^2-3x+1)(x-1)^3}$\\[4pt]\hline
91&$\{4213,1342,1243\}$&$-\frac{4x^5-14x^4+17x^3-14x^2+6x-1}{(3x-1)(x^2-x+1)(x-1)^3}$\\[4pt]\hline
92&$\{2314,3124,1432\}$&$\frac{(x^3-2x^2+3x-1)(x^2+x-1)(1-x)^3}{x^9-2x^8+6x^7-4x^6-7x^5+32x^4-40x^3+25x^2-8x+1}$\\[4pt]\hline
95&$\{2314,4132,1342\}$&$-\frac{4x^6-25x^5+51x^4-56x^3+32x^2-9x+1}{(2x-1)(x-1)^2(x^2-3x+1)^2}$\\[4pt]\hline
96&$\{2134,4132,1342\}$&$-\frac{4x^8+6x^7-45x^6+100x^5-126x^4+95x^3-42x^2+10x-1}{(2x-1)^2(x^2-3x+1)(x-1)^4}$\\[4pt]\hline
97&$\{2341,4312,4123\}$&$-\frac{(x-1)^4(x^3-2x^2+3x-1)}{x^8-4x^7+18x^6-35x^5+51x^4-47x^3+26x^2-8x+1}$\\[4pt]\hline
98&$\{2134,3124,1432\}$&$\frac{(x-1)^{3}(x^3+2x-1)}{4x^6-7x^5+9x^4-15x^3+13x^2-6x+1}$\\[4pt]\hline
100&$\{4312,1342,4123\}$&$-\frac{4x^6-16x^5+30x^4-31x^3+20x^2-7x+1}{(x-1)^3(2x^4-7x^3+8x^2-5x+1)}$\\[4pt]\hline
101&$\{3124,4132,1342\}$&$-\frac{4x^6-16x^5+30x^4-31x^3+20x^2-7x+1}{(x-1)^3(2x^4-7x^3+8x^2-5x+1)}$\\[4pt]\hline
102&$\{2413,3142,1234\}$&$-\frac{(x-1)^{3}(x^3-2x^2+3x-1)}{x^7-4x^6+12x^5-23x^4+28x^3-19x^2+7x-1}$\\[4pt]\hline
104&$\{2134,4132,1423\}$&$-\frac{3x^7-4x^6+17x^5-46x^4+55x^3-32x^2+9x-1}{(x-1)(x^2-3x+1)(2x-1)^3}$\\[4pt]\hline
105&$\{4213,2134,1342\}$&$-\frac{7x^6-25x^5+51x^4-56x^3+32x^2-9x+1}{(2x-1)(x-1)^2(x^2-3x+1)^2}$\\[4pt]\hline
107&$\{4213,3412,1342\}$&$-\frac{(x^2-x+1)(2x-1)^3}{(4x^3-7x^2+5x-1)(x-1)^3}$\\[4pt]\hline
110&$\{2134,3142,1432\}$&$-\frac{(x-1)(3x^3-5x^2+4x-1)^2}{x^9+2x^8-27x^7+86x^6-144x^5+150x^4-100x^3+42x^2-10x+1}$\\[4pt]\hline
111&$\{2143,3142,1234\}$&$\frac{(x^3-2x^2+3x-1)^2}{(2x^3-3x^2+4x-1)(x-1)^3}$\\[4pt]\hline
113&$\{2134,1432,1234\}$&$\frac{2x^5-x^4-3x^3-2x^2-2x+1}{2x^5-2x^3-x^2-3x+1}$\\[4pt]\hline
114&$\{4312,1423,1234\}$&$-\frac{x^{10}-3x^9+2x^8+4x^7-9x^6+24x^5-43x^4+38x^3-22x^2+7x-1}{(x-1)^8}$\\[4pt]\hline
115&$\{4231,1423,1234\}$&$-\frac{2x^{10}-17x^9+66x^8-158x^7+256x^6-289x^5+230x^4-126x^3+46x^2-10x+1}{(x-1)^{11}}$\\[4pt]\hline
116&$\{4312,4123,1243\}$&$\frac{x^9-23x^8+133x^7-315x^6+419x^5-350x^4+188x^3-63x^2+12x-1}{(2x-1)^4(x-1)^5}$\\[4pt]\hline
117&$\{3124,4132,1234\}$&$\frac{x^9-6x^8+22x^7-53x^6+92x^5-104x^4+76x^3-35x^2+9x-1}{(x^3-2x^2+3x-1)(2x-1)(x-1)^5}$\\[4pt]\hline
119&$\{4312,1432,1324\}$&$\frac{-350x^4-63x^2+419x^5-26x^8+138x^7-317x^6+188x^3+x^9-1+12x}{(2x-1)^4(x-1)^5}$\\[4pt]\hline
120&$\{4132,1423,1234\}$&$-\frac{x^8+4x^7-41x^6+99x^5-126x^4+95x^3-42x^2+10x-1}{(x^2-3x+1)(2x-1)^2(x-1)^4}$\\[4pt]\hline
122&$\{4213,1432,1324\}$&$-\frac{5{x}^{5}-19{x}^{4}+25{x}^{3}-19{x}^{2}+7x-1}{(x-1)({x}^{2}-3x+1)(3{x}^{3}-5{x}^{2}+4x-1)}$\\[4pt]\hline
123&$\{4132,1342,1234\}$&$-\frac{3x^8-46x^7+141x^6-225x^5+221x^4-137x^3+52x^2-11x+1}{(x^2-3x+1)(2x-1)^{2}(x-1)^5}$\\[4pt]\hline
124&$\{2341,4132,4123\}$&$-\frac{(2x-1)(x-1)^4}{2x^6-8x^5+19x^4-27x^3+19x^2-7x+1}$\\[4pt]\hline
128&$\{2341,3142,4123\}$&$\frac{\left({x}^{2}-3x+1\right)\left(x-1\right)^{5}}{4{x}^{7}-23{x}^{6}+55{x}^{5}-78{x}^{4}+66{x}^{3}-33{x}^{2}+9x-1}$\\[4pt]\hline
135&$\{1432,4123,1243\}$&$-\frac{5x^5-14x^4+22x^3-18x^2+7x-1}{(x-1)^4(2x^2-4x+1)}$\\[4pt]\hline
136&$\{4213,1342,4123\}$&$\frac{x^5-3x^3+4x^2-4x+1}{x^5+x^4-6x^3+7x^2-5x+1}$\\[4pt]\hline
137&$\{3124,1432,1342\}$&$-\frac{(x^2-3x+1)(x^2+2x-1)}{(x-1)(x^4-2x^3-5x^2+5x-1)}$\\[4pt]\hline
138&$\{2134,3142,1243\}$&$-\frac{(x^2-3x+1)(x^2+2x-1)}{(x-1)(x^4-2x^3-5x^2+5x-1)}$\\[4pt]\hline
139&$\{2143,3124,1342\}$&$-\frac{(x^2-3x+1)(x^2+2x-1)}{(x-1)(x^4-2x^3-5x^2+5x-1)}$\\[4pt]\hline
140&$\{3124,1432,1243\}$&$\frac{(3x-1)(x-1)^3}{9x^4-19x^3+17x^2-7x+1}$\\[4pt]\hline
141&$\{2143,1423,1234\}$&$\frac{2x^4-4x^3+7x^2-5x+1}{4x^4-9x^3+11x^2-6x+1}$\\[4pt]\hline
142&$\{1432,1342,4123\}$&$\frac{x^3+3x-1}{x^3-2x^2+4x-1}$\\[4pt]\hline
143&$\{4312,4123,1234\}$&$-\frac{x^8-3x^7-12x^6+23x^5-43x^4+38x^3-22x^2+7x-1}{(x-1)^8}$\\[4pt]\hline
144&$\{4231,4123,1234\}$&$-\frac{3x^8-15x^7+40x^6-66x^5+81x^4-60x^3+29x^2-8x+1}{(x-1)^9}$\\[4pt]\hline
145&$\{4312,1423,1243\}$&$\frac{2x^7-2x^6-25x^5+59x^4-61x^3+33x^2-9x+1}{(x-1)^2(2x-1)^4}$\\[4pt]\hline
146&$\{4132,1243,1234\}$&$\frac{x^9-4x^8+20x^6-58x^5+83x^4-69x^3+34x^2-9x+1}{(2x-1)(x^2-3x+1)(x-1)^5}$\\[4pt]\hline
147&$\{4132,1324,1234\}$&$-\frac{13x^{10}-45x^9+83x^8-38x^7-141x^6+308x^5-306x^4+178x^3-62x^2+12x-1}{(x^2-3x+1)(x^2+x-1)(2x-1)^2(x-1)^5}$\\[4pt]\hline
148&$\{2134,4132,1324\}$&$-\frac{5x^8-51x^7+172x^6-288x^5+283x^4-170x^3+61x^2-12x+1}{(2x-1)^2(x^2-3x+1)^2(x-1)^3}$\\[4pt]\hline
152&$\{4231,2341,4123\}$&$\frac{(x-1)^{6}(x^2-3x+1)}{5x^8-31x^7+83x^6-134x^5+144x^4-99x^3+42x^2-10x+1}$\\[4pt]\hline
154&$\{4312,1342,1423\}$&$-\frac{3x^5-14x^4+21x^3-18x^2+7x-1}{(x-1)(2x^3-4x^2+4x-1)(x^2-3x+1)}$\\[4pt]\hline
155&$\{3124,4132,1243\}$&$-\frac{3x^5-14x^4+21x^3-18x^2+7x-1}{(x-1)(2x^3-4x^2+4x-1)(x^2-3x+1)}$\\[4pt]\hline
160&$\{4312,1432,1342\}$&$\frac{2x^5-4x^4-10x^3+16x^2-7x+1}{(x-1)(3x-1)(2x-1)(x^2+2x-1)}$\\[4pt]\hline
161&$\{4312,4132,1342\}$&$-\frac{7x^5-22x^4+33x^3-24x^2+8x-1}{(x^3-3x^2+4x-1)(x-1)(2x-1)^2}$\\[4pt]\hline
167&$\{3142,3124,1432\}$&$-\frac{(2x-1)(x-1)(x^2-3x+1)}{x^5-7x^4+18x^3-17x^2+7x-1}$\\[4pt]\hline
168&$\{3124,1432,1423\}$&$-\frac{(x^2-3x+1)(2x-1)^{2}}{(x-1)(x^4-13x^3+16x^2-7x+1)}$\\[4pt]\hline
169&$\{3142,1423,1234\}$&$-\frac{(x^2-3x+1)(2x-1)^2}{(x-1)(x^4-13x^3+16x^2-7x+1)}$\\[4pt]\hline
179&$\{2134,1432,1423\}$&$-\frac{2x^5-8x^4+12x^3-12x^2+6x-1}{(x^4-5x^3+10x^2-6x+1)(x^2-x+1)}$\\[4pt]\hline
181&$\{2143,1324,1234\}$&$-\frac{2x^3+3x-1}{x^4-2x^3+2x^2-4x+1}$\\[4pt]\hline
183&$\{4132,4123,1234\}$&$\frac{x^8-8x^7+31x^{6}-75x^5+98x^4-75x^3+35x^2-9x+1}{(2x-1)^2(x-1)^6}$\\[4pt]\hline
186&$\{4132,4123,1243\}$&$-\frac{-27{x}^{5}+55{x}^{4}-57{x}^{3}+32{x}^{2}-9x+1+4{x}^{6}}{\left(3x-1\right)\left({x}^{2}-3x+1\right)\left(x-1\right)^{4}}$\\[4pt]\hline
189&$\{2143,2134,1432\}$&$-\frac{x^4-7x^3+8x^2-5x+1}{x^5-5x^4+13x^3-12x^2+6x-1}$\\[4pt]\hline
200&$\{2143,3124,1243\}$&$\frac{{x}^{3}-6{x}^{2}+5x-1}{\left(x-1\right)\left(5{x}^{2}-5x+1\right)}$\\[4pt]\hline
202&$\{1432,1423,1234\}$&$\frac{x^4-4x^3+10x^{2}-6x+1}{3x^4-11x^3+15x^2-7x+1}$\\[4pt]\hline
205&$\{1432,1324,1234\}$&$-\frac{x^7-2x^6+4x^5-17x^4+24x^3-18x^2+7x-1}{2x^6-14x^5+34x^4-38x^3+24x^2-8x+1}$\\[4pt]\hline
206&$\{1432,1243,1234\}$&$-\frac{x^6+5x^4-12x^3+12x^2-6x+1}{2x^5-13x^4+21x^3-17x^2+7x-1}$\\[4pt]\hline
\end{longtable}}
\end{document}